\documentclass{amsart}
\usepackage{amsfonts,amssymb,stmaryrd,amscd,amsmath,latexsym,amsbsy}

\numberwithin{equation}{section}
\theoremstyle{plain}

\newtheorem{thm}{Theorem}[section]

\newtheorem{prop}[thm]{Proposition}

\newtheorem{defi}[thm]{Definition}
\newtheorem{lem}[thm]{Lemma}
\newtheorem{cor}[thm]{Corollary}

\theoremstyle{remark}
\newtheorem{rema}[thm]{Remark}

\newcommand{\N}{\mathbb{N}}
\newcommand{\Z}{\mathbb{Z}}

\newcommand{\C}{\mathbb{C}}

\title[DAHA and bispectral quantum KZ equations]{Double affine
Hecke algebras and bispectral quantum Knizhnik-Zamolodchikov equations}
\author{Michel van Meer and Jasper V. Stokman}
\address{KdV Institute for Mathematics, University of Amsterdam,
Science Park 904, 1098 XH Amsterdam, The Netherlands.}
\email{m.vanmeer@uva.nl, j.v.stokman@uva.nl}
\begin{document}
\begin{abstract}
We use the double affine Hecke algebra of type $\textup{GL}_N$
to construct an explicit consistent system of $q$-difference
equations, which we call the bispectral quantum Knizhnik-Zamolodchikov
(BqKZ) equations. BqKZ includes, besides
Cherednik's quantum affine KZ
equations associated to principal series representations of the underlying
affine Hecke algebra, a compatible system of $q$-difference equations
acting on the central character of the principal series representations.
We construct a meromorphic self-dual
solution $\Phi$ of BqKZ which, upon suitable specializations of the central
character, reduces to
symmetric self-dual Laurent polynomial solutions of
quantum KZ equations.
We give an explicit correspondence between solutions of
BqKZ and solutions of a particular bispectral problem for
Ruijsenaars' commuting trigonometric $q$-difference
operators. Under this correspondence $\Phi$ becomes
a self-dual Harish-Chandra series solution $\Phi^+$
of the bispectral problem. Specializing the central character as above,
we recover from $\Phi^+$
the symmetric self-dual Macdonald polynomials.
\end{abstract}
\maketitle
\setcounter{tocdepth}{1}
\tableofcontents

\section{Introduction}

Let $M$ be an $N!$-dimensional complex vector space and
write $T=(\C\setminus\{0\})^N$.
We derive an explicit
holonomic system of $q$-difference equations on $M$-valued meromorphic
functions on $T\times T$, which we call the bispectral quantum
Knizhnik-Zamolodchikov (BqKZ) equations. It is related to
Cherednik's \cite{C} quantum affine
KZ equations in the following way.

Let $H$ be the extended affine Hecke algebra of type
$\textup{GL}_N$. For fixed $\zeta\in T$ the vector space $M$
admits an $H$-module structure, turning $M$ into the principal
series module $M(\zeta)$ of $H$ with central character $\zeta$.
The BqKZ naturally splits up into two subsystems of $q$-difference
equations. The first subsystem acts only on the first
$T$-component of $T\times T$ and as such it realizes, for any
fixed $\zeta\in T$, Cherednik's \cite{C} quantum affine KZ
equation $\textup{qKZ}_\zeta$ acting on $M(\zeta)$-valued
meromorphic functions on $T\times \{\zeta\}$. The second, dual
subsystem is obtained from the first by replacing the role of
$(t,\zeta)\in T\times T$ by $(\zeta^{-1},t^{-1})$ and conjugating
the $q$-connection matrices by an explicit complex linear
automorphism $C_\iota$ of $M$. Hence, it acts only on the second
$T$-component of $T\times T$ and it essentially realizes
$\textup{qKZ}_{t^{-1}}$ for fixed $t\in T$. In particular, this
provides a quantum isomonodromic interpretation of qKZ. This
should be compared with the interpretation of rational KZ
equations as quantizations of Schlesinger equations, see \cite{Re}
and \cite{Ha}.

The BqKZ is constructed using Cherednik's \cite{C} double affine
Hecke algebra $\mathbb{H}$ of type $\textup{GL}_N$. As a vector
space $\mathbb{H}$ is isomorphic to $\C[T]\otimes
H\simeq\C[T]\otimes H_0\otimes\C[T]\simeq\C[T\times T]\otimes H_0$
with $H_0$ the finite Hecke algebra of type $A_{N-1}$. Cherednik's
anti-algebra involution $*\colon\mathbb{H}\to\mathbb{H}$
essentially interchanges, under the above vector space
identification, the role of the two copies of $\C[T]$. For
$w,w^\prime\in W=S_N\ltimes\Z^N$ we consider the map $h\mapsto
\widetilde{S}_wh\widetilde{S}_{w^\prime}^*$ ($h\in\mathbb{H}$),
where the $\widetilde{S}_w\in\mathbb{H}$ are Cherednik's
nonnormalized ($X$-)intertwiners. Restricted to $w,w^\prime
\in\Z^N$, suitable renormalizations of these maps become the
$q$-connection matrices of BqKZ, with $H_0$ playing the role of
$M$. The anti-involution $*$ of $\mathbb{H}$ gives rise to the
automorphism $C_\iota$ interchanging the qKZ subsystem of BqKZ
with its dual subsystem in BqKZ.

$\textup{qKZ}_\zeta$ is gauge equivalent
to Frenkel and Reshetikhin's \cite{FR} quantum
KZ equations associated with the $N$-fold tensor product
of the vector representation of quantum $\mathfrak{sl}_N$
(see \cite[\S 1.3.2]{C}). A special case of $\textup{qKZ}_\zeta$
was considered earlier by Smirnov \cite{Sm}.
Etingof and Varchenko \cite{EV} used quantum group methods to
construct systems of $q$-difference equations (so-called
dynamical $q$-difference
equations) that are compatible with Frenkel and Reshetikhin's
quantum KZ equations associated to evaluation
representations of quantum affine algebras. It is likely that the system of
dynamical $q$-difference equations associated with $\textup{qKZ}_\zeta$
is equivalent to the dual qKZ subsystem in BqKZ.

Preceding the above mentioned work \cite{EV} of Etingof and
Varchenko, dynamical equations for various degenerations of
quantum KZ equations have been analyzed in detail; see, e.g.,
\cite{FMTV}, \cite{TL}, \cite{T}, \cite{TV}, \cite{TV2} and
\cite{Le}. An interesting aspect in, e.g., \cite{TL} and
\cite{TV}, is the observation that various degenerations of
quantum KZ equations are the duals of their associated dynamical
equations with respect to $(\mathfrak{gl}_r,\mathfrak{gl}_s)$
duality. In the present set-up (which corresponds to $r=s=N$),
this duality is incorporated by the automorphism $C_\iota$, which
reflects Cherednik's duality anti-involution of the double affine
Hecke algebra $\mathbb{H}$ on the level of BqKZ.

We investigate the space $\textup{SOL}$ of $M$-valued meromorphic
solutions of BqKZ in detail. We first analyze BqKZ in a suitable
asymptotic region. It leads to a solution $\Phi$ of BqKZ which is
self-dual, in the sense that $\Phi(t,\gamma)=C_\iota
\Phi(\gamma^{-1},t^{-1})$ as $M$-valued meromorphic functions in
$(t,\gamma)\in T\times T$. We construct a basis of solutions of
$\textup{SOL}$ in terms of $\Phi$, and we give an explicit formula
for the leading term of $\Phi(t,\gamma)$ as function of $t$.

Cherednik \cite[Thm. 3.4]{Cref} constructed for arbitrary root
systems a correspondence between solutions of quantum KZ equations
and solutions of a system of $q$-difference equations. In
\cite[Thm. 4.4]{Cref}, Cherednik made the correspondence precise
for $\textup{GL}_N$. It yields an explicit map $\chi_+$ from
solutions of $\textup{qKZ}_{\zeta}$ to solutions of the spectral
problem of Ruijsenaars' \cite{R} commuting trigonometric
$q$-difference operators with spectral parameter $\zeta^{-1}$ (the
Ruijsenaars operators are also frequently referred to as
Macdonald-Ruijsenaars operators). The latter result has been
generalized to arbitrary root systems in \cite[Thm. 4.6]{Ka} and
\cite{CInd}. We analyze the map $\chi_+$ in the present bispectral
setting. It leads to the interpretation of $\chi_+$ as an
embedding of SOL into the space of meromorphic solutions of a
bispectral problem involving the above Ruijsenaars operators as
well as Ruijsenaars operators acting on the spectral parameter.

The application of the correspondence to the self-dual solution
$\Phi$ of BqZK leads to a self-dual Harish-Chandra series solution
$\Phi^+$ of the bispectral problem. Harish-Chandra series
solutions of the Ruijsenaars operators with fixed spectral
parameter were investigated before in, e.g., \cite{EK1},
\cite{EK2}, \cite{KK} and \cite{LS}. The present approach to
Harish-Chandra series, which uses quantum KZ equations in an
essential way, has the advantage that it leads to new results on
the convergence and singularities of the Harish-Chandra series.
These results, together with Cherednik's recent work \cite{CWhit},
form important building blocks in deriving the $c$-function
expansion of Cherednik's global $(q,t)$-spherical function (this
will be detailed in a forthcoming paper of the second author).

The $q$-connection matrices of the dual qKZ subsystem of BqKZ can
be used to map solutions of $\textup{qKZ}_\zeta$ for a fixed
central character $\zeta\in T$ to solutions of
$\textup{qKZ}_{\zeta^\prime}$ with respect to a suitably shifted
central character $\zeta^\prime$. Applied to a constant solution
we obtain a symmetric Laurent polynomial solution $Q_\lambda$ for
any non-increasing $N$-tuple $\lambda$ of integers, which is
automatically self-dual. We express $Q_\lambda$ in terms of a
suitable specialization of the solution $\Phi$ of BqKZ. Applying
$\chi_+$ to $Q_\lambda$ we obtain a symmetric self-dual Laurent
polynomial eigenfunction of the Ruijsenaars operators, which is
the normalized Macdonald \cite{MacBook} polynomial of degree
$\lambda$. We show that the well-known duality and evaluation
formula for the Macdonald polynomials (see \cite[Chpt.
VI]{MacBook}) are direct consequences of the properties of the
Laurent polynomials $Q_\lambda$.

For various other approaches to the construction of
solutions of quantum KZ equations see, e.g., \cite{FR}, \cite{EFK},
\cite{JM}, \cite{KK}, \cite{Mi}, \cite{FZJ} and \cite{KT}.

Many results in the present paper, in particular the BqZK
equations themselves, can be extended to arbitrary root systems.
These extensions use the root system generalizations of
$\mathbb{H}$ and their fundamental properties (see \cite{C}). We
have chosen to restrict the present exposition to the
$\textup{GL}_N$ theory, not only in order to present the basic
ideas while limiting the technicalities, but also because the
$\textup{GL}_N$ theory deserves extra attention by its
particularly strong ties with quantum groups and quantum
integrable lattice models (see, e.g., \cite{EFK}, \cite{JM}). The
theory for arbitrary root systems is detailed in a forthcoming
paper of the first author.

\noindent
{\bf Conventions}\\
{\bf --} $\otimes$ always stands for tensor product over $\C$
and $\textup{End}(M)$, for a module $M$ over $\C$,
stands for $\C$-linear endomorphisms.\\
{\bf --} $\mathbb{N}=\{1,2,\ldots\}$.\\
{\bf --} For a module $M$ over a commutative ring $R$
and a ring extension $R\subset S$, we write $M^S=S\otimes_RM$.

\noindent
{\bf Acknowledgments}\\
Both authors are supported by the Netherlands
Organization for Scientific Research (NWO) in the VIDI-project
``Symmetry and modularity in exactly solvable models''.
We thank the referee for valuable comments.

\section{The double affine Hecke algebra}\label{SectionPrelim}

\subsection{The extended affine Weyl group}
Let $N\geq 2$ and let $D=D_N$ be the affine Dynkin diagram of
affine type $\widehat{A}_{N-1}$ (the cyclic graph with $N$
vertices if $N\geq 3$). The $N$ vertices are labeled by the
numbers $0,1,\ldots,N-1$ (anticlockwise if $N\geq 3$). We identify
occasionally the set of labels by the group $\mathbb{Z}_N$ of
integers modulo $N$.

Write $W_Q$ for the affine Weyl group of affine type
$\widehat{A}_{N-1}$. In terms of its Coxeter generators $s_i$
($i\in\Z_N$), the characterizing group relations are the quadratic
relations $s_i^2=1$ and, if $N\geq 3$, the braid relations
\begin{equation}\label{braid}
\begin{split}
s_is_{i+1}s_i&=s_{i+1}s_is_{i+1},\\
s_is_j&=s_js_i,\qquad\quad i-j\not=0,\pm 1.
\end{split}
\end{equation}
The subgroup generated by $s_1,\ldots,s_{N-1}$ is isomorphic to the
symmetric group $S_N$ in $N$ letters, where $s_i$ is identified with
the simple transposition $i\leftrightarrow i+1$.

Let $\textup{Aut}(D)$ be the group of automorphisms of the
affine Dynkin diagram of type $\widehat{A}_{N-1}$. Let $c\in\textup{Aut}(D)$
the element of order $N$, acting on the
label set $\Z_N$ of the vertices of $D$ by $c(i)=i+1$.
We view $c$ as automorphism of $W_Q$ by $c(s_i)=s_{i+1}$.

Let $\Omega=\langle \pi\rangle$ be the infinite cyclic group with cyclic
generator $\pi$. It acts by
group automorphisms on $W_Q$ by $\pi\mapsto c$. Accordingly we can
define the semi-direct product group $W=\Omega\ltimes W_Q$, which is
called the extended affine Weyl group (associated to
$\textup{GL}_N$). We denote $e$ for the identity element
of $W$.

Since $s_0=\pi s_{N-1}\pi^{-1}$, the subgroups $S_N$ and $\Omega$
already generate $W$ as a group. Furthermore we have $W\simeq
S_N\ltimes\Z^N$. The cyclic generator $\pi$ of $\Omega$
corresponds to $\pi=\sigma\epsilon_N$, where
$\{\epsilon_i\}_{i=1}^N$ denotes the standard $\Z$-basis of $\Z^N$
and $\sigma=s_1s_2\cdots s_{N-1}\in S_N$ is the ``clockwise
rotation" which maps $N$ to 1 and all other $i$ to $i+1$.
Conversely,
\begin{equation}\label{ej}
\epsilon_j=s_{j-1}\cdots s_2s_1\pi s_{N-1}s_{N-2}\cdots s_j
\end{equation}
for $j=1,\ldots,N$.

\begin{rema}
Under the identification $W\simeq S_N\ltimes\Z^N$, we have
$W_Q=S_N\ltimes Q$ with $Q\subset \Z^N$ the sublattice of rank
$N-1$ consisting of $N$-tuples of integers that sum up to zero
(this is the (co)root lattice of the root system
$R=\{\epsilon_i-\epsilon_j\}_{1\leq i\not=j\leq N}$ of type
$A_{N-1}$).
\end{rema}

For $w\in W$ let $w^\prime\in W_Q$ and $\omega\in\Omega$ be the unique
group elements such that $w=w^\prime\omega$, then we define
the length $\ell(w)$ of $w$
to be the length of $w^\prime\in W_Q$, i.e., it is the minimal number $r$
such that $w^\prime$ can be expressed as
\[w^\prime=s_{i_1}\cdots s_{i_r}
\]
for some $i_k\in\Z_N$ (such an expression of $w^\prime$,
as well as the resulting expression for $w=w^\prime\omega$, is called a
reduced expression). Thus $\Omega$
consists of the elements of $W$ of length zero.

A central role in this paper is played by an
action of the extended affine Weyl group $W$ by $q$-difference
reflection operators on suitable function spaces on $T:=(\C^\times)^N$,
where $\C^\times:=\C\setminus\{0\}$. Here $q$ is taken to be
real and strictly between zero and one (with minor technical
adjustments the condition on $q$ may be relaxed to
$0<|q|<1$, and a parallel theory can be
developed for $|q|>1$). Since $q$ is fixed once and for all,
we will in general suppress the dependence on $q$ in notations.
We start with an action of $W$ on $T$ by
\begin{equation}\label{WactionT}
\begin{split}
wt&=(t_{w^{-1}(1)},\ldots,t_{w^{-1}(N)}),\qquad w\in S_N,\\
\lambda t&=(q^{\lambda_1}t_1,\ldots,q^{\lambda_N}t_N),\qquad
\lambda=(\lambda_1,\ldots,\lambda_N)\in\Z^N,
\end{split}
\end{equation}
for $t=(t_1,\ldots,t_N)\in T$. It is convenient to
introduce $\kappa^{\lambda}:=(\kappa^{\lambda_1},\ldots,\kappa^{\lambda_N})$
for $\kappa\in\C^\times$ and $\lambda\in\Z^N$,
so that the action of $\lambda\in\Z^N$ on $t\in T$ can simply be
written as
\[
\lambda t=q^\lambda t
\]
in standard vector notation.
Note that the action of $\pi\in\Omega$ is given by
\[\pi(t_1,\ldots,t_N)=(qt_N,t_1,\ldots,t_{N-1}).
\]

Consider the algebra $\C[T]=\C[x_1^{\pm1},\ldots,x_N^{\pm1}]$ of
complex-valued, regular functions on $T$, where $x_i(t):=t_i$ for
$t=(t_1,\ldots,t_N)\in T$ are the standard coordinate functions.
We write
\[
x^{\lambda}:=x_1^{\lambda_1}\cdots x_N^{\lambda_N}\in\C[T]
\]
for $\lambda=(\lambda_1,\ldots,\lambda_N)\in\Z^N$, which form the monomial
basis of $\C[T]$.

Let $\C(T)$ be the field of rational functions on $T$,
$\mathcal{O}(T)$ be the ring of analytic functions on $T$, and
$\mathcal{M}(T)$ be the field of meromorphic functions on $T$.
Note that $\mathcal{M}(T)$ is the quotient field of
$\mathcal{O}(T)$ (cf. \cite[Thm. 7.4.6]{Ho}). The $W$-action on
$T$ gives rise to a left $W$-action by algebra automorphisms on
$\C[T]$, $\C(T)$, $\mathcal{O}(T)$ and $\mathcal{M}(T)$, via
\[
(wf)(t)=f(w^{-1}t)
\]
for $w\in W$, $t\in T$. We can, in particular, form the smashed
product algebra $\C(T)\#W$. Recall that if $G$ is a group and $A$
is a $G$-algebra over $\C$ (that is, a unital associative algebra
over $\C$ endowed with a left $G$-action by algebra
automorphisms), then the smashed product algebra $A\#G$ is the
unique complex unital associative algebra such that
\begin{itemize}
    \item[(i)] $A\#G=A\otimes\mathbb{C}[G]$ as a complex vector space;
    \item[(ii)] the canonical linear
    embeddings $A\hookrightarrow A\#G$, $\mathbb{C}[G]\hookrightarrow A\#G$
    are algebra homomorphisms; and
    \item[(iii)] the cross relations
    \[(a\otimes g)\cdot(b\otimes h)=ag(b)\otimes gh,\]
    are satisfied for $a,b\in A$ and $g,h\in G$.
\end{itemize}
We will always write $ag:=a\otimes g\in A\#G$ ($a\in A$, $g\in G$).
Observe that $A\#G$ canonically acts on any $G$-algebra $B$
containing $A$ as a $G$-subalgebra.

Note that the smashed product algebra $\C(T)\#W$ depends on $q$,
since the $W$-action on $\C(T)$ depends on $q$ (see
\eqref{WactionT}). Sometimes it is convenient to emphasize its
$q$-dependence, in which case we write $\C(T)\#_qW$ instead of
$\C(T)\#W$.

The canonical left $\C(T)\#W$-action on $\C(T)$ (and
$\mathcal{M}(T)$) is faithful and realizes $\C(T)\#W$ as the
algebra of $q$-difference $S_N$-reflection operators with
coefficients in $\C(T)$. If $f\in\C(T)$ then we write $f(X)$ for
the associated element in $\C(T)\#W$ (it is the operator defined
as multiplication by $f$). In particular, $X_i$ is multiplication
by the coordinate function $x_i$.

\subsection{The extended affine Hecke algebra and Cherednik's
basic representation}
In this subsection we recall some constructions and results
due to Cherednik (see, e.g., \cite[Chpt. 1]{C} and references therein).

Fix a nonzero complex number $k$.
\begin{defi}
The affine Hecke algebra $H_Q=H_Q(k)$ is the complex, unital,
associative algebra generated by $T_i$ ($i\in\Z_N$) and satisfying
\begin{enumerate}
\item[(1)] if $N\geq 3$, the braid relations
\begin{equation*}
\begin{split}
T_iT_{i+1}T_i&=T_{i+1}T_iT_{i+1},\\
T_iT_j&=T_jT_i,\qquad i-j\not=0,\pm 1;
\end{split}
\end{equation*}
\item[(2)] the quadratic relations $(T_i-k)(T_i+k^{-1})=0$.
\end{enumerate}
\end{defi}
The Dynkin diagram automorphism $c\in\textup{Aut}(D)$
can also be viewed as automorphism of $H_Q$ by $c(T_i)=T_{i+1}$.
Accordingly, $\Omega$ acts by algebra automorphisms
on $H_Q$ by $\pi\mapsto c$. The extended affine Hecke algebra
$H=H(k)$ is the associated smashed product algebra $H_Q\#\Omega$.

For a reduced expression
$w=s_{i_1}\cdots s_{i_r}\omega\in W$ ($i_k\in\Z_N$,
$\omega\in\Omega$), the element
\[T_w:=T_{i_1}\cdots T_{i_r}\omega\in H
\]
is well defined. The $T_w$ ($w\in W$) form a linear basis of $H$.
For $k=1$, the extended affine Hecke algebra $H$ is isomorphic to
the group algebra $\C[W]$ of $W$ via the identification
$T_w\leftrightarrow w$ ($w\in W$).

The finite Hecke algebra is the subalgebra $H_0$ of $H$ generated by
$T_1,\ldots,T_{N-1}$. The elements $T_w$ ($w\in S_N$) form a linear
basis of $H_0$. Note that $H$ is already generated as algebra by
$H_0$ and $\pi^{\pm 1}$, since $T_0=\pi T_{N-1}\pi^{-1}$.

Put
\begin{equation}\label{eqYdef}
Y_i:=T_{i-1}^{-1}\cdots T_2^{-1}T_1^{-1}\pi T_{N-1}T_{N-2}\cdots
T_i\in H
\end{equation}
for $i=1,\ldots,N$. Note that $Y_i$ becomes the translation element
$\epsilon_i$ in $W$ if $k=1$. We furthermore write
$Y^\lambda:=Y_1^{\lambda_1}\cdots Y_N^{\lambda_N}$ for
$\lambda=(\lambda_1,\ldots,\lambda_N)\in\Z^N$. We have the
following characterization of $H$, due to Bernstein. For details
we refer to Lusztig \cite{L} or Macdonald \cite[\S 4.2]{M}.
\begin{thm}
$H$ is the unique unital complex associative algebra, such that
\begin{itemize}
    \item[\bf{(i)}] $H_0\otimes\C[T]\simeq H$ as complex vector
    spaces, via $h\otimes f\mapsto hf(Y)$ for $h\in H_0$,
    $f\in\C[T]$, where $f(Y)=\sum_\lambda c_\lambda Y^{\lambda}$
    if $f=\sum_\lambda c_\lambda x^\lambda\in\C[T]$;
    \item[\bf{(ii)}] the canonical maps $H_0,\C[T]\hookrightarrow H$
    are algebra embeddings; we write
    $\C_Y[T]=\textup{span}_\C\{Y^\lambda\}_{\lambda\in\Z^N}$
    for the image of $\C[T]$ in $H$; and
    \item[\bf{(iii)}] the following cross relations
    \begin{equation*}
    \begin{split}
    T_i^{-1}Y_iT_i^{-1}&=Y_{i+1},\\
    Y_jT_i&=T_iY_j,\qquad i\neq j-1,j
    \end{split}
    \end{equation*}
are satisfied for $1\leq i<N$ and $1\leq j\leq N$.
\end{itemize}
\end{thm}
Cherednik realized the affine Hecke algebra $H$ inside the algebra
$\C(T)\#W$ of $q$-difference reflection operators as follows.
\begin{thm}\label{thmRho}
    There is a unique injective algebra homomorphism
    $\rho=\rho_{k,q}\colon H(k)\to\C(T)\#_qW$ satisfying
    \[
    \begin{split}
        \rho(T_i)&=k+c_k(X_i/X_{i+1})(s_i-1),\\
        \rho(\pi)&=\pi,
    \end{split}
    \]
    for $i=1,\ldots,N-1$, where
    \begin{equation}\label{ck}
    c_k(z):=\frac{k^{-1}-kz}{1-z}.
    \end{equation}
\end{thm}
Note that for the affine Hecke algebra $H=H(k)$ with fixed parameter
$k$, Theorem \ref{thmRho} yields a one-parameter family of realizations of
$H$ (the additional parameter being $q$).
\begin{rema}
    The image $\rho(H)$ preserves $\C[T]$, viewed as a
    subspace of the canonical $\C(T)\#W$-module $\C(T)$. The
    resulting representation of $H$ on $\C[T]$ is faithful and is
    called the basic representation of $H$.
\end{rema}
 We frequently
identify $H$ with its image under $\rho$ in $\C(T)\#W$.

We now come to the definition of Cherednik's double affine Hecke
algebra which depends, besides on $k$, on the additional parameter
$q$.

\begin{defi}
The double affine Hecke algebra $\mathbb{H}=\mathbb{H}(k,q)$ is
the subalgebra of $\C(T)\#_qW$ generated by $\rho_{k,q}(H)$ and by
the multiplication operators $f(X)$ ($f\in \C[T]$).
\end{defi}

Let $\mathbb{L}=\C[T]\otimes\C[T]\simeq\C[T\times T]$ denote the
complex-valued regular functions on $T\times T$. We view
$\mathbb{H}$ as $\mathbb{L}$-module by
\begin{equation}\label{Laction}
(f\otimes g)\cdot h:=f(X)hg(Y)
\end{equation}
for $f,g\in\C[T]$ and $h\in\mathbb{H}$. The following theorem is
the so-called Poincar\'e-Birkhoff-Witt (PBW) property of the
double affine Hecke algebra.
\begin{thm}
We have $\mathbb{H}\simeq H_0^{\mathbb{L}}=\mathbb{L}\otimes H_0$
as $\mathbb{L}$-modules.
\end{thm}

The PBW property is an essential ingredient in deriving the
characterizing relations for the double affine Hecke algebra
$\mathbb{H}$ in terms of its algebraic generators $T_i$ ($1\leq
i<N$), $\pi^{\pm 1}$ and $X_j^{\pm 1}$ ($1\leq j\leq N$). Since we
are not going to use this presentation explicitly in this paper,
we refer the reader to \cite{C} for further details. We use though
one of its direct consequences, namely the existence of the
duality anti-isomorphism (see Cherednik \cite[Thm. 1.4.8]{C}):

\begin{thm}\label{dualthm}
There exists a unique $\C$-linear anti-algebra involution $*\colon
\mathbb{H}\rightarrow \mathbb{H}$ determined by
\begin{equation*}
\begin{split}
T_w^*&=T_{w^{-1}},\qquad\: w\in S_N,\\
(Y^{\lambda})^*&=X^{-\lambda},\qquad \lambda\in \Z^N,\\
(X^{\lambda})^*&=Y^{-\lambda},\qquad \lambda\in \Z^N.
\end{split}
\end{equation*}
\end{thm}

\subsection{Intertwiners}\label{intsec}
In this subsection we recall the construction of the
(nonnormalized) affine intertwiners associated to the double affine
Hecke algebra $\mathbb{H}$.
The intertwiners play an important role in the construction of
a nontrivial $W\times W$-cocycle in the next section.
Consider the elements
\[
\begin{split}
\widetilde{S}_i&=(k-k^{-1}X_{i+1}/X_{i})s_i,\qquad 1\leq i<N,\\
\widetilde{S}_0&=(k-k^{-1}q^{-1}X_1/X_N)s_0,\\
\widetilde{S}_\pi&=\pi
\end{split}
\]
in $\C(T)\#_qW$. The following facts are well known (cf., e.g.,
\cite[\S 1.3]{C}). For the convenience of the reader, we give a short
sketch of the proof.
\begin{prop}\label{lemIntProp}
Let $w\in W$ and let $w=s_{j_1}\cdots s_{j_r}\pi^m$ be a reduced
expression ($j_l\in\Z_N$, $m\in\Z$).
\begin{itemize}
  \item[{\bf (i)}] $\widetilde{S}_w:=\widetilde{S}_{j_1}\cdots
\widetilde{S}_{j_r}\widetilde{S}_\pi^m$
is a well-defined element of $\C(T)\#W$;
\item[{\bf (ii)}] $\widetilde{S}_w\in\mathbb{H}$;
    \item[\bf{(iii)}] the $\widetilde{S}_{i}$ ($i\in\Z_N$)
    satisfy the $\widehat{A}_{N-1}$-type braid relations;
    \item[\bf{(iv)}] $\widetilde{S}_wf(X)=(wf)(X)\widetilde{S}_w$
    in $\C(T)\#W$ for all $f\in\C(T)$; and
    \item[\bf{(v)}]
    $\widetilde{S}_i\widetilde{S}_i=(k-k^{-1}X_{i+1}/X_{i})
    (k-k^{-1}X_{i}/X_{i+1})$ for $1\leq i<N$.
\end{itemize}
\end{prop}
\begin{proof}
(i) Set $d_i:=(k-k^{-1}X_{i+1}/X_{i})$ ($1\leq i<N$) and
$d_0:=(k-k^{-1}q^{-1}X_1/X_N)$. We have
\[\widetilde{S}_w
=d_{j_1}(s_{j_1}d_{j_2})\cdots
(s_{j_1}\cdots s_{j_{r-1}}d_{j_r})w
\]
in $\C(T)\#W$.
By, e.g., Macdonald \cite[(2.2.9)]{M}, we know that
\begin{equation}\label{cw}
d_w:=d_{j_1}(s_{j_1}d_{j_2})\cdots (s_{j_1}\cdots s_{j_{r-1}}d_{j_r})
\end{equation}
is independent of the reduced expression of $w$. Hence
$\widetilde{S}_w\in\C(T)\#W$ is well defined.\\
(ii) Note that $\widetilde{S}_i$ can be written as
\begin{equation}\label{Si}
\widetilde{S}_i=(1-X_{i+1}/X_{i})(T_i-k)+k-k^{-1}X_{i+1}/X_i
\end{equation}
for $1\leq i<N$, which shows that it lies in $\mathbb{H}$.
Furthermore, $\pi^{\pm 1}\in\mathbb{H}$, hence $\widetilde{S}_\pi^{\pm 1}
\in\mathbb{H}$
and  $\widetilde{S}_0=\pi\widetilde{S}_{N-1}\pi^{-1}\in\mathbb{H}$.
Consequently, $\widetilde{S}_w\in\mathbb{H}\subset\C(T)\#W$.\\
(iii) is immediate from (i), while (iv) and (v) are clear from the
definition of the $\widetilde{S}_i$ and $\widetilde{S}_\pi$.
\end{proof}

\begin{defi}
The elements $\widetilde{S}_w$ ($w\in W$) are called the
affine intertwiners of $\mathbb{H}$.
\end{defi}

\section{The bispectral quantum KZ equations}\label{SectionBqKZ}

\subsection{Construction of the cocycle}\label{constr}
Let $\iota$ denote the nontrivial element of the two group
$\mathbb{Z}_2$. We define the group $\mathbb{W}$ as the semidirect
product
\[
\mathbb{W}:=\mathbb{Z}_2\ltimes(W\times W),
\]
where $\iota\in\mathbb{Z}_2$ acts on $W\times W$ by switching the
components: $\iota(w,w')=(w',w)\iota$ for $w,w'\in W$.
We first use the double affine Hecke algebra, its
affine intertwiners, and its duality
anti-isomorphism to construct a group homomorphism
$\tau=\tau_k: \mathbb{W}\rightarrow \textup{GL}_{\C}(H_0^{\mathbb{K}})$
depending on the Hecke algebra parameter $k$, where
$\mathbb{K}:=\mathcal{M}(T\times T)$ is the field of meromorphic
functions on $T\times T$.

The representation $\tau$ will be constructed from
the complex linear endomorphisms
$\widetilde{\sigma}_{(w,w^\prime)}$ ($w,w^\prime\in W$) and
$\widetilde{\sigma}_{\iota}$ of the double affine Hecke algebra
$\mathbb{H}$, defined by
\begin{equation*}
\begin{split}
\widetilde{\sigma}_{(w,w^\prime)}(h)&=
\widetilde{S}_wh\widetilde{S}_{w^\prime}^*,\\
\widetilde{\sigma}_\iota(h)&=h^*
\end{split}
\end{equation*}
for $h\in\mathbb{H}$. In the following lemma we collect some
elementary properties of the maps
$\widetilde{\sigma}_{(w,w^\prime)}$ and $\widetilde{\sigma}_\iota$.
First we introduce some auxiliary notations.

For a regular function $g\in\C[T]$, we write $g(x)\in\C[T\times
T]$ (respectively $g(y)\in\C[T\times T]$) for the corresponding
regular function on $T\times T$ constant with respect to the
second (respectively first) $T$-component. In particular, the
$x_i$ (respectively $y_i$) are the standard coordinate functions
of the first (respectively second) copy of $T$ in $T\times T$.
Recall the regular function $d_w\in\C[T]$ ($w\in W$) such that
\[\widetilde{S}_w=d_w(X)w
\]
in $\C(T)\#W$; see \eqref{cw}.

\begin{lem}\label{sigmaprop}
The complex linear endomorphisms $\widetilde{\sigma}_{(w,w^\prime)}$
and $\widetilde{\sigma}_\iota$ of $\mathbb{H}$ satisfy the following
properties:
\begin{enumerate}
\item[{\bf (i)}] the $\widetilde{\sigma}_{(s_i,e)}$ ($i\in\Z_N$)
satisfy the $\widehat{A}_{N-1}$ braid relations;
\item[{\bf (ii)}] $\widetilde{\sigma}_{(s_i,e)}^2=d_{s_i}(x)(s_id_{s_i})(x)
\cdot\textup{id}$ for $i\in\Z_N$;
\item[{\bf (iii)}] $\widetilde{\sigma}_{(\pi,e)}\widetilde{\sigma}_{(s_i,e)}
\widetilde{\sigma}_{(\pi^{-1},e)}=\widetilde{\sigma}_{(s_{i+1},e)}$
for $i\in\Z_N$;
\item[{\bf (iv)}] $\widetilde{\sigma}_\iota^2=\textup{id}$ and
$\widetilde{\sigma}_{(e,w)}=
\widetilde{\sigma}_\iota\widetilde{\sigma}_{(w,e)}
\widetilde{\sigma}_{\iota}$ for $w\in W$; and
\item[{\bf (v)}] $\widetilde{\sigma}_{(w,e)}\widetilde{\sigma}_{(e,w^\prime)}=
\widetilde{\sigma}_{(w,w^\prime)}=\widetilde{\sigma}_{(e,w^\prime)}
\widetilde{\sigma}_{(w,e)}$ for $w,w^\prime\in W$.
\end{enumerate}
\end{lem}
\begin{proof}
These are direct consequences of Proposition \ref{lemIntProp} and Theorem
\ref{dualthm}.
\end{proof}
To construct a $\mathbb{W}$-action from the maps
$\widetilde{\sigma}_{(w,w^\prime)}$ and
$\widetilde{\sigma}_\iota$, we need to renormalize the maps
appropriately. To do so, we describe as a first step the behavior
of the maps $\widetilde{\sigma}_{(w,w^\prime)}$ and
$\widetilde{\sigma}_\iota$ with respect to the $\mathbb{L}$-module
structure \eqref{Laction} on $\mathbb{H}$. This will allow us to
extend the maps $\widetilde{\sigma}_{(w,w^\prime)}$ and
$\widetilde{\sigma}_\iota$ to endomorphisms of
$H_0^{\mathbb{K}}\simeq \mathbb{K}\otimes_{\mathbb{L}}\mathbb{H}$,
which is a suitably flexible surrounding for the normalizations of
the maps to take place in.

Consider the group involution ${}^\diamondsuit: W\rightarrow W$
defined by $w^\diamondsuit=w$ for $w\in S_N$ and
$\lambda^\diamondsuit=-\lambda$ for $\lambda\in \Z^N$. Then
$\mathbb{W}$ acts on $T\times T$ by
\begin{equation*}
\begin{split}
(w,w^\prime)(t,\gamma)&=(wt,w^\prime{}^\diamondsuit\gamma),\\
\iota(t,\gamma)&=(\gamma^{-1},t^{-1})
\end{split}
\end{equation*}
for $w,w^\prime\in W$, where
$t^{-1}:=(t_1^{-1},\ldots,t_N^{-1})\in T$ and the action of $W$ on
$T$ is by $q$-dilations and permutations; see \eqref{WactionT}. By
transposition, this defines an action of $\mathbb{W}$ on
$\mathbb{K}=\mathcal{M}(T\times T)$ by field automorphisms,
\begin{equation}\label{doubleaction}
(\mathrm{w}f)(t,\gamma)=f(\mathrm{w}^{-1}(t,\gamma)),
\qquad \mathrm{w}\in\mathbb{W}.
\end{equation}
Note that $\mathbb{L}=\C[T\times T]$ is a $\mathbb{W}$-subalgebra
of $\mathbb{K}$.

\begin{lem}\label{actok}
For $h\in\mathbb{H}$ and $f\in \mathbb{L}$ we have
\begin{equation}\label{cat}
\begin{split}
\widetilde{\sigma}_{(w,w^\prime)}(f\cdot h)&=
((w,w^\prime)f)\cdot \widetilde{\sigma}_{(w,w^\prime)}(h),\\
\widetilde{\sigma}_{\iota}(f\cdot h)&=
(\iota f)\cdot \widetilde{\sigma}_{\iota}(h)
\end{split}
 \end{equation}
for $w,w^\prime\in W$.
\end{lem}
\begin{proof}
From Proposition \ref{lemIntProp} we know that
$\widetilde{S}_wp(X)=(wp)(X)\widetilde{S}_w$ in $\mathbb{H}$
for $p\in\C[T]$ and $w\in W$. For $p\in\C[T]$ let $p^\diamondsuit\in\C[T]$
be defined by $p^\diamondsuit(t)=p(t^{-1})$, then we also have
\[
p(Y)\widetilde{S}_{w^\prime}^*=(\widetilde{S}_{w^\prime}p^\diamondsuit(X))^*=
((w^\prime p^\diamondsuit)(X)\widetilde{S}_{w^\prime})^*=
\widetilde{S}_{w^\prime}^*(w^\prime p^\diamondsuit)^\diamondsuit(Y)
\]
in $\mathbb{H}$. Hence for $p,r\in\C[T]$,
\[
\widetilde{\sigma}_{(w,w^\prime)}(p(X)hr(Y))=
(wp)(X)\widetilde{S}_wh\widetilde{S}_{w^\prime}^*(w^\prime
r^\diamondsuit)^\diamondsuit(Y).
\]
The first formula of \eqref{cat} now follows since
$(w^\prime r^\diamondsuit)^\diamondsuit=w^\prime{}^\diamondsuit r$.
The second is immediate
from the definition of the duality anti-involution.
\end{proof}

As a direct consequence of Lemma \ref{actok} the maps
$\widetilde{\sigma}_{(w,w^\prime)}$ ($w,w^\prime\in W$) and
$\widetilde{\sigma}_{\iota}$ uniquely extend to complex linear
endomorphisms of
$H_0^{\mathbb{K}}\simeq\mathbb{K}\otimes_{\mathbb{L}}\mathbb{H}$
such that \eqref{cat} is valid for all $f\in\mathbb{K}$ and $h\in
H_0^{\mathbb{K}}$. We keep the same notations
$\widetilde{\sigma}_{(w,w^\prime)}$ and $\widetilde{\sigma}_\iota$
for these maps. Note that the properties of
$\widetilde{\sigma}_{(w,w^\prime)}$ and $\widetilde{\sigma}_\iota$
as described in Lemma \ref{sigmaprop} also hold true as
identities between endomorphisms of $H_0^{\mathbb{K}}$.

\begin{thm}
There is a unique group homomorphism
\[
\tau\colon\mathbb{W}\to\textup{GL}_\C(H_0^{\mathbb{K}})
\]
satisfying
\begin{equation}\label{fish}
\begin{split}
\tau(w,w^\prime)(f)&=d_w(x)^{-1}d_{w^\prime}^\diamondsuit(y)^{-1}\cdot
\widetilde{\sigma}_{(w,w^\prime)}(f),\\
\tau(\iota)(f)&=\widetilde{\sigma}_{\iota}(f)
\end{split}
\end{equation}
for $w,w^\prime\in W$ and $f\in H_0^{\mathbb{K}}$. It satisfies
$\tau(\mathrm{w})(g\cdot f)=\mathrm{w}g\cdot \tau(\mathrm{w})(f)$ for
$g\in\mathbb{K}$, $f\in H_0^{\mathbb{K}}$ and
$\mathrm{w}\in\mathbb{W}$.
\end{thm}
\begin{proof}
The last statement is clear.

The action $\tau$ of $W\times \{e\}$ arises naturally from left
multiplication by normalized affine intertwiners on a suitable
localization of the double affine Hecke algebra (see Cherednik
\cite[\S 1.3]{C}). In the present set-up, one observes that Lemma
\ref{actok} and Lemma \ref{sigmaprop}(i)-(ii) imply that the
$\tau(s_i,e)$ ($i\in\Z_N$) satisfy the $\widehat{A}_{N-1}$ braid
relations and the quadratic relations $\tau(s_i,e)^2=
\textup{id}_{H_0^{\mathbb{K}}}$. Since furthermore $\tau(\pi,e)$
is a complex linear automorphism of $H_0^{\mathbb{K}}$ with
inverse $\tau(\pi^{-1},e)$, and
$\tau(\pi,e)\tau(s_i,e)\tau(\pi^{-1},e)=\tau(s_{i+1},e)$ for
$i\in\Z_N$ by Lemma \ref{actok} and Lemma \ref{sigmaprop}(iii), we
conclude that the formulas \eqref{fish} for the maps $\tau(s_i,e)$
($i\in\Z_N$) and $\tau(\pi,e)$ uniquely extend to a group
homomorphism $\tau: W\times \{e\}\rightarrow
\textup{GL}_\C(H_0^{\mathbb{K}})$. It follows from Proposition
\ref{lemIntProp} and its proof that the resulting group
homomorphism satisfies
\[\tau(w,e)f=d_w(x)^{-1}\cdot \widetilde{\sigma}_{(w,e)}f
\]
for $w\in W$. This is in accordance with formula \eqref{fish}.

Combining Lemma \ref{sigmaprop}(iv) with Lemma \ref{actok} we
can relate the complex endomorphism $\tau(e,w)$ (see \eqref{fish})
of $H_0^{\mathbb{K}}$ to $\tau(w,e)$ by the formula
\begin{equation}\label{switch}
\tau(e,w)=\tau(\iota)\tau(w,e)\tau(\iota),\qquad w\in W,
\end{equation}
where $\tau(\iota)$ is given by the second formula of \eqref{fish}.
Since $\tau(\iota)^2=\widetilde{\sigma}_\iota^2=
\textup{id}_{H_0^\mathbb{K}}$ we conclude that
$W\ni w\mapsto \tau(e,w)$ (see \eqref{fish}) defines a left
$W$-action on $H_0^{\mathbb{K}}$.

By Lemma \ref{sigmaprop} and Lemma \ref{actok} (v)
we have
\[\tau(w,e)\tau(e,w^\prime)=
\tau(w,w^\prime)=\tau(e,w^\prime)\tau(w,e)
\]
for all $w,w^\prime\in W$. Thus $\tau\colon W\times W\rightarrow
\textup{GL}_{\mathbb{C}}(H_0^{\mathbb{K}})$, defined by the first
formula of \eqref{fish}, is a group homomorphism. Combined with
\eqref{switch} and $\tau(\iota)^2=\textup{id}_{H_0^{\mathbb{K}}}$
we conclude that $\tau$ \eqref{fish} indeed defines a complex
linear action of $\mathbb{W}$ on $H_0^{\mathbb{K}}$.
\end{proof}

For $\mathrm{w}\in\mathbb{W}$ and $f\in H_0^{\mathbb{K}}=\mathbb{K}\otimes H_0$
we write $\mathrm{w}f$ for the action of $\mathrm{w}$
on the $\mathbb{K}$-coefficients
of $f$ in its expansion along a basis of $H_0$.
In other words, viewing $f(t,\gamma)$ as $H_0$-valued meromorphic function
in $(t,\gamma)\in T\times T$, the action is given by
$(\mathrm{w}f)(t,\gamma)=f(\mathrm{w}^{-1}(t,\gamma))$.
Consider $\textup{GL}_{\mathbb{K}}(H_0^{\mathbb{K}})$ as a $\mathbb{W}$-group
by the corresponding conjugation action
\begin{equation}\label{star}
(\mathrm{w},A)\mapsto
\mathrm{w} A \mathrm{w}^{-1},\qquad \mathrm{w}\in\mathbb{W},
\,\,\, A\in\textup{GL}_{\mathbb{K}}(H_0^{\mathbb{K}})
\end{equation}
by group automorphisms. We have the following direct consequence of
the previous theorem.
\begin{cor}\label{CorCoc}
The map $\mathrm{w}\mapsto
C_{\mathrm{w}}:=\tau(\mathrm{w})\mathrm{w}^{-1}$ is a cocycle of
$\mathbb{W}$ with values in the $\mathbb{W}$-group
$\textup{GL}_{\mathbb{K}} (H_0^{\mathbb{K}})$. In other words,
$C_{\mathrm{w}}\in\textup{GL}_{\mathbb{K}}(H_0^{\mathbb{K}})$ and
\[C_{\mathrm{w}\mathrm{w}^\prime}=
C_{\mathrm{w}}\mathrm{w}C_{\mathrm{w}^\prime}\mathrm{w}^{-1}
\]
for all $\mathrm{w},\mathrm{w}^\prime\in\mathbb{W}$.
\end{cor}
For more details on non-abelian group cohomology, see, e.g.,
the appendix in \cite{S}.

\begin{rema}\label{Wtriv}
Interpreting $A\in\textup{End}_{\mathbb{K}}(H_0^{\mathbb{K}})$ as
$\textup{End}(H_0)$-valued meromorphic function $A(t,\gamma)$ in
$(t,\gamma)\in T\times T$, the action \eqref{star} becomes
$(\mathrm{w}A)(t,\gamma)=A(\mathrm{w}^{-1}(t,\gamma))$. In
particular, $\mathrm{w}A\mathrm{w}^{-1}=A$ for all
$\mathrm{w}\in\mathbb{W}$ if
$A\in\textup{End}_{\mathbb{K}}(H_0^{\mathbb{K}})$ is the
$\mathbb{K}$-linear extension of a complex linear endomorphism of
$H_0$. This is, for instance,  the case for the cocycle value
$C_{\iota}$ (see Subsection \ref{cocyclevalues}).
\end{rema}

\begin{rema}\label{rationality}
One may replace in this subsection $\mathbb{K}$ by the field
$\mathbb{C}(T\times T)$ of rational functions on $T\times T$.
Consequently, the cocycle value $C_{\mathrm{w}}(t,\gamma)$ for
$\mathrm{w}\in \mathbb{W}$ is a rational
$\textup{End}(H_0)$-valued function in $(t,\gamma)\in T\times T$.
We presented the results with respect to
$\mathbb{K}=\mathcal{M}(T\times T)$ since this is the natural
setting for the applications of the cocycle $C$ in the analytic
theory of the quantum KZ equations (to which we come at a later
stage).
\end{rema}

\subsection{Bispectral quantum KZ equations}\label{BqKZsection}
In this subsection, we use the cocycle
$C_{\mathrm{w}}\in\textup{GL}_{\mathbb{K}}(H_0^{\mathbb{K}})$
($\mathrm{w}\in\mathbb{W}$) to define a holonomic system of
$q$-difference equations on the space $H_0^{\mathbb{K}}$ of
$H_0$-valued meromorphic functions on $T\times T$.

The constructions thus far have led to a $\C$-linear action $\tau$
of $\mathbb{W}$ on $H_0^{\mathbb{K}}$. In terms of the cocycle
$C_{\mathrm{w}}$ ($\mathrm{w}\in\mathbb{W}$), it is given by
\begin{equation}\label{actionexpl}
(\tau(\mathrm{w})f)(t,\gamma)=
C_\mathrm{w}(t,\gamma)f(\mathrm{w}^{-1}(t,\gamma))
\end{equation}
for $\mathrm{w}\in\mathbb{W}$ and $f\in H_0^{\mathbb{K}}$, where
\eqref{actionexpl} should be read as identities between $H_0$-valued
meromorphic functions in $(t,\gamma)\in T\times T$.
It follows that $f\in H_0^{\mathbb{K}}$ is $\tau(\Z^N\times\Z^N)$-invariant
if and only if
\begin{equation}\label{BqKZeqn}
C_{(\lambda,\mu)}(t,\gamma)f(q^{-\lambda}t,q^\mu\gamma)=f(t,\gamma)
\qquad \forall\, \lambda,\mu\in\Z^N,
\end{equation}
viewed as identities between $H_0$-valued meromorphic functions on $T\times T$.
\begin{defi}\label{BqKZ}
We call the $q$-difference equations \eqref{BqKZeqn}
the bispectral quantum KZ (BqKZ) equations.
We write $\textup{SOL}$ for the set of functions $f\in H_0^{\mathbb{K}}$
satisfying the BqKZ equations \eqref{BqKZeqn}.
\end{defi}

Let $\mathbb{F}\subset\mathbb{K}$ denote the subfield consisting
of $f\in \mathbb{K}$ satisfying $(\lambda,\mu)f=f$ for all
$\lambda,\mu\in\Z^N$. Let furthermore $\mathbb{S}_N$ denote the
subgroup $\Z_2\ltimes (S_N\times S_N)$ of $\mathbb{W}$.
\begin{cor}
{\bf (i)} The BqKZ equations \eqref{BqKZeqn} form a holonomic
system of $q$-difference equations. In other words, the connection matrices
$C_{(\lambda,\mu)}$ ($\lambda,\mu\in\Z^N$) satisfy the compatibility conditions
\begin{equation}\label{compatible}
C_{(\lambda,\mu)}(t,\gamma)C_{(\nu,\xi)}(q^{-\lambda}t,q^{\mu}\gamma)=
C_{(\nu,\xi)}(t,\gamma)C_{(\lambda,\mu)}(q^{-\nu}t,q^{\xi}\gamma)
\end{equation}
for $\lambda,\mu,\nu,\xi\in\Z^N$, as $\textup{End}(H_0)$-valued meromorphic
functions in $(t,\gamma)\in T\times T$.\\
{\bf (ii)} The solution space $\textup{SOL}$ of BqKZ is a
$\tau(\mathbb{S}_N)$-invariant $\mathbb{F}$-subspace of
$H_0^{\mathbb{K}}$.
\end{cor}
\begin{proof}
(i) The cocycle condition implies that both sides of
\eqref{compatible} are equal to
$C_{(\lambda+\nu,\mu+\xi)}(t,\gamma)$.\\
(ii) Clearly, $\textup{SOL}$ is a $\mathbb{F}$-subspace of
$H_0^{\mathbb{K}}$. Note, furthermore, that $\Z^N\times\Z^N$ is a
normal subgroup of $\mathbb{W}$ with quotient group isomorphic to
$\mathbb{S}_N$. Hence the $\mathbb{F}$-subspace $\textup{SOL}$ of
$\tau(\Z^N\times\Z^N)$-invariant elements in the
$\tau(\mathbb{W})$-module $H_0^{\mathbb{K}}$ is
$\tau(\mathbb{S}_N)$-invariant.
\end{proof}

\section{The explicit form of the bispectral quantum KZ equations}
\label{SectionExplForm}

In this section we derive explicit expressions for the cocycle
values $C_{\mathrm{w}}$ ($\mathrm{w}\in\mathbb{W}$) and, in
particular, for the $q$-connection matrices $C_{(\lambda,\mu)}$
($\lambda,\mu\in\Z^N$) of the BqKZ. It will become apparent that the
$C_{(\lambda,e)}(\cdot,\zeta)$ ($\lambda\in\Z^N$) with $\zeta\in T$
fixed coincide with the $q$-connection matrices of Cherednik's
quantum affine KZ equation associated to the principal series module
of $H(k)$ with central character $\zeta$. They also turn up as
gauged $q$-connection matrices for a Frenkel-Reshetikhin \cite{FR} type
quantum KZ equation associated to the quantum affine algebra
$\mathcal{U}_k(\widehat{sl}_N)$.
\subsection{Formal principal series}\label{formal}
View the commutative subalgebra $\C_Y[T]$ of $H$ as left $\C_Y[T]$-module
by left multiplication. Let $M=\textup{Ind}_{\C_Y[T]}^H\bigl(\C_Y[T]\bigr)$
be the corresponding induced left $H$-module. With respect to the
$\C[T]\simeq\C[\{1\}\times T]$-module structure
\[
f\cdot (h\otimes_{\C_Y[T]}g(Y))=h\otimes_{\C_Y[T]}(fg)(Y)\qquad
f,g\in\C[T],\,\, h\in H
\]
on $M$ we have $M\simeq H_0^{\C[\{1\}\times T]}=\C[\{1\}\times
T]\otimes H_0$ as $\C[\{1\}\times T]$-modules. The left $H$-action
on $M$ is $\C[\{1\}\times T]$-linear, hence we obtain an algebra
homomorphism
\[
\eta:
H\rightarrow \textup{End}_{\C[\{1\}\times T]}\bigl(H_0^{\C[\{1\}\times T]}\bigr).
\]
We occasionally view $\eta(h)$ as $\textup{End}(H_0)$-valued
regular function in $\gamma\in T$, in which case we write it as
$T\ni \gamma\mapsto\eta(h)(\gamma)$. Extending the ground ring
$\C[\{1\}\times T]$ to $\mathbb{K}=\mathcal{M}(T\times T)$ we
obtain an algebra homomorphism
\[H\rightarrow \textup{End}_{\mathbb{K}}(H_0^{\mathbb{K}}),
\]
which we shall also denote by $\eta$. {}From this viewpoint,
$\eta(h)(\gamma)$ is the regular $\textup{End}(H_0)$-valued
function in $(t,\gamma)\in T\times T$ which is constant in $t$.
Note that $\eta(h)$ for $h\in H_0$ is constant as
$\textup{End}(H_0)$-valued function on $T\times T$ (see Remark
\ref{Wtriv}).
\begin{lem}\label{etaexplicit}
For $w\in S_N$ and $1\leq i<N$ we have
\begin{equation}\label{formulaTI}
\eta(T_i)T_w=
\begin{cases}
T_{s_iw}\qquad &\hbox{ if } \ell(s_iw)=\ell(w)+1,\\
(k-k^{-1})T_w+T_{s_iw}\qquad &\hbox{ if } \ell(s_iw)=\ell(w)-1,
\end{cases}
\end{equation}
and
\begin{equation}\label{formulaEtaPi}
\eta(\pi)(\gamma)T_w=\gamma_{w^{-1}(N)}T_{\sigma w}
\end{equation}
as regular $H_0$-valued functions in $\gamma\in T$.
\end{lem}
\begin{proof}
The first formula follows directly from the definitions.
For the second formula it suffices to verify that
$\pi T_w=T_{\sigma w}Y_{w^{-1}(N)}$ in $H$.

If $w=1$ then $\pi=T_\sigma Y_N$
since $\sigma=s_1s_2\cdots s_{N-1}$ is a reduced expression.
If $w=s_{N-1}$ then $$\pi T_w=\pi T_{N-1}=T_1\cdots
T_{N-2}Y_{N-1}=T_{s_1\cdots s_{N-2}}Y_{w^{-1}(N)}= T_{\sigma
w}Y_{w^{-1}(N)}.$$

Next, we prove that $\pi T_w=T_{\sigma w}Y_{w^{-1}(N)}$ in $H$ if
$w\not=1$ and $\ell(s_iw)=\ell(w)+1$ for all $1\leq i\leq N-2$.
Then $w=s_{N-1}s_{N-2}\cdots s_j$ for some $1\leq j<N$ (and this
is a reduced expression of $w$). We find, making repetitive use of
the cross relation $T_rY_{r+1}T_r=Y_r$ ($1\leq r< N$) in $H$,
\begin{equation*}
    \begin{split}
    \pi T_w&=\pi T_{N-1}T_{N-2}\cdots T_j = T_1\cdots
    T_{N-2}Y_{N-1}T_{N-2}\cdots T_j\\
        &=T_1\cdots T_{N-3}Y_{N-2}T_{N-3}\cdots T_j\\
        &\:\:\vdots\\
        &=T_1\cdots T_{j-1}Y_j = T_{s_1\cdots s_{j-1}}Y_j\\
        &=T_{\sigma w}Y_j=T_{\sigma w}Y_{w^{-1}(N)},
    \end{split}
\end{equation*}
which is the desired relation in $H$.

The general case is now proved by induction on $\ell(w)$.
Let $w\not=1$ and decompose it as $w=s_iu$ with $1\leq i<N$ and
$u\in S_N$ such that $\ell(s_iu)=\ell(u)+1$. Suppose that
$\pi T_u=T_{\sigma u}Y_{u^{-1}(N)}$ in $H$. In order to prove
that $\pi T_w=T_{\sigma w}Y_{w^{-1}(N)}$ we may, in view of the previous
paragraph, assume without loss of generality that
$1\leq i\leq N-2$. Then $s_i(N)=N$
and $\ell(s_{i+1}\sigma u)=\ell(\sigma u)+1$. (The latter equality
is equivalent to $(\sigma u)^{-1}(i+1)<(\sigma u)^{-1}(i+2)$, which
follows from $u^{-1}(i)<u^{-1}(i+1)$, which again is equivalent to
the assumption $\ell(s_iu)=\ell(u)+1$.) Then
\begin{equation*}
    \begin{split}
    \pi T_w&=\pi T_iT_u = T_{i+1}\pi T_u = T_{i+1}T_{\sigma u}Y_{u^{-1}(N)}\\
        &=T_{s_{i+1}\sigma u}Y_{u^{-1}s_i(N)} = T_{\sigma s_{i}
        u}Y_{w^{-1}(N)}\\
        &=T_{\sigma w}Y_{w^{-1}(N)}
    \end{split}
\end{equation*}
in $H$, which completes the proof.
\end{proof}

In view of the explicit expression \eqref{Si} for the intertwiner
$\widetilde{S}_i\in\mathbb{H}$ ($1\leq i<N$) and the definition of
the duality anti-involution, we have $\widetilde{S}_w^\ast\in H$
for all $w\in S_N$. We now set
\[
\xi_w:=\eta(\widetilde{S}_{w^{-1}}^*)T_e\in H_0^{\mathbb{K}},\qquad
w\in S_N.
\]
Note that $\xi_w\in H_0^{\C[\{1\}\times T]}\subset
H_0^{\mathbb{K}}$ for $w\in S_N$.
We view $\xi_w$ as regular $H_0$-valued function in $\gamma\in T$, as well as
meromorphic $H_0$-valued function in $(t,\gamma)\in T\times T$
constant in $t\in T$.
\begin{lem}\label{commoneig}
$\{\xi_w\}_{w\in S_N}$ is a $\mathbb{K}$-basis of $H_0^{\mathbb{K}}$
consisting of common eigenfunctions for the $\eta$-action of $\C_Y[T]$
on $H_0^{\mathbb{K}}$. For $p\in\C[T]$ and $w\in S_N$ we have
\begin{equation}\label{eigfc}
\eta(p(Y))(\gamma)\xi_w(\gamma)=(w^{-1}p)(\gamma)\xi_w(\gamma)
\end{equation}
as $H_0$-valued regular functions in $\gamma\in T$.
\end{lem}
\begin{proof}
For $p\in\C[T]$ we have $\eta(p(Y))(\gamma)T_e=p(\gamma)T_e$.
Note, furthermore, that
$p(Y)\widetilde{S}_{w^{-1}}^\ast=\widetilde{S}_{w^{-1}}^\ast
(w^{-1}p)(Y)$ in $H$ for $p\in\C[T]$ and $w\in S_N$; see the proof
of Lemma \ref{actok}. Combining the two observations gives
\eqref{eigfc}. It follows from Proposition
\ref{lemIntProp}(iv)-(v) that the $\xi_w$ ($w\in S_N$) are nonzero
in $H_0^{\mathbb{K}}$. The eigenvalue equations \eqref{eigfc} then
show that the $\xi_w$ ($w\in S_N)$ are $\mathbb{K}$-linearly
independent in $H_0^{\mathbb{K}}$.
\end{proof}

\subsection{The cocycle values}\label{cocyclevalues}
We define
\[
R_i(z)=c_k(z)^{-1}(\eta(T_i)-k)+1,\qquad 1\leq i<N,
\]
viewed as a rational $\textup{End}(H_0)$-valued function in $z$.
The results of the previous subsection implies that the $R_i(z)$
satisfy the following Yang-Baxter type equations (see Cherednik
\cite[\S 1.3.2]{C}).
\begin{lem}\label{YBlem}
We have
\[C_{(s_i,e)}(t,\gamma)=R_i(t_i/t_{i+1}),\qquad 1\leq i<N,
\]
as rational $\textup{End}(H_0)$-valued
functions in $(t,\gamma)\in T\times T$. In particular, the $R_i(z)$
satisfy
\begin{equation}\label{YBform}
\begin{split}
R_i(z)R_i(z^{-1})&=\textup{id},\\
R_j(z)R_{j+1}(zz^\prime)R_j(z^\prime)&=
R_{j+1}(z^\prime)R_j(zz^\prime)R_{j+1}(z),
\end{split}
\end{equation}
for $1\leq i<N$ and $1\leq j<N-1$
as $\textup{End}(H_0)$-valued rational functions.
\end{lem}
\begin{proof}
For $1\leq i<N$ and $h\in H_0$ we have, as $H_0$-valued
meromorphic functions in $(t,\gamma)\in T\times T$,
\begin{equation*}
\begin{split}
C_{(s_i,e)}(t,\gamma)h&=(\tau(s_i,e)h)(t,\gamma)\\
&=d_{s_i}(t)^{-1}(\widetilde{S}_ih)(t,\gamma)\\
&=c_k(t_i/t_{i+1})^{-1}(\eta(T_i)-k)h+h,
\end{split}
\end{equation*}
in view of the explicit expression \eqref{Si} for
$\widetilde{S}_i$, $c_k$ \eqref{ck} and $d_{s_i}$ (see the proof
of Lemma \ref{lemIntProp}). For the second statement of the lemma,
note that the cocycle property of $C$ implies for $1\leq i<N$ and
$1\leq j<N-1$ that
\begin{equation*}
\begin{split}
C_{(s_i,e)}(t,\gamma)C_{(s_i,e)}(s_it,\gamma)&=\textup{id},\\
C_{(s_j,e)}(t,\gamma)C_{(s_{j+1},e)}(s_jt,\gamma)&C_{(s_j,e)}
(s_{j+1}s_jt,\gamma)\\
&=C_{(s_{j+1},e)}(t,\gamma)C_{(s_j,e)}(s_{j+1}t,\gamma)
C_{(s_{j+1},e)}(s_js_{j+1}t,\gamma),
\end{split}
\end{equation*}
as rational $\textup{End}(H_0)$-valued functions in $(t,\gamma)\in T\times T$.
Using $C_{(s_i,e)}(t,\gamma)=R_i(t_i/t_{i+1})$ these formulas imply
\eqref{YBform}.
\end{proof}
Observe that $C_\iota$ is the $\mathbb{K}$-linear
extension of the anti-algebra involution of $H_0$ mapping $T_w$ to
$T_{w^{-1}}$ for all $w\in S_N$. Note furthermore that
\begin{equation}\label{Cpi}
C_{(\pi,e)}=\eta(\pi).
\end{equation}
Together with the explicit description of $C_{(s_i,e)}$ ($1\leq i<N$)
from the previous lemma, these formulas determine the
values $C_{\mathrm{w}}$ ($\mathrm{w}\in\mathbb{W}$) uniquely
(cf. Corollary \ref{CorCoc}).
In particular, the cocycle property of $C$ implies that
\[C_{(e,w)}(t,\gamma)=C_\iota C_{(w,e)}(\gamma^{-1},t^{-1})C_\iota,\qquad
w\in W,
\]
as $\textup{End}(H_0)$-valued rational functions
in $(t,\gamma)\in T\times T$.
\begin{lem}\label{singularC}
Let $w\in W$.\\
{\bf (i)} $C_{(w,e)}\in (\C(T)\otimes \C[T])\otimes
\textup{End}(H_0)$.\\
{\bf (ii)} The $\C[T]\otimes\textup{End}(H_0)$-valued rational function
$t\mapsto C_{(w,e)}(t,\cdot)$ in $t\in T$ is regular at
$t\in T\setminus \mathcal{S}$, where
\begin{equation}\label{Scal}
\mathcal{S}=\{t\in T \,\, | \,\,
t^\alpha\in k^{-2}q^\Z\quad \textup{for some } \alpha\in R \}.
\end{equation}
\end{lem}
\begin{proof}
By the cocycle condition, $C_{(w,e)}(t,\gamma)$
can be written as a product of factors $C_{(s_i,e)}(ut,\gamma)$
($1\leq i<N$, $u\in W$) and $\eta(\pi^{\pm 1})(\gamma)$.
By Lemma \ref{etaexplicit} and the fact that $R_i(z)$ has a single pole
at $z=k^{-2}$, we conclude (i) and (ii).
\end{proof}

\subsection{The $q$-connection matrices}\label{qconSection}
Besides the standard $\Z$-basis $\{\epsilon_i\}_{i=1}^N$ of
$\Z^N$, we also have the $\Z$-basis $\{\varpi_i\}_{i=1}^N$
consisting of the fundamental weights
$\varpi_i:=\sum_{j=1}^i\epsilon_j\in\Z^N$. Note that
\begin{equation}\label{omegasigma}
\varpi_i=\pi^i\sigma^{-i}
\end{equation}
in $W$ for all $1\leq i\leq N$. In the following lemma, we compute
the $q$-connection matrices $C_{(\lambda,e)}$ for $\lambda\in\Z^N$
of BqKZ explicitly in case $\lambda$ is one of these two types of
basis elements of $\Z^N$.
\begin{lem}\label{explicit}
{\bf (i)} For $1\leq j\leq N$ we have
\begin{equation*}
\begin{split}
C_{(\epsilon_j,e)}(t,\gamma)&=
R_{j-1}(t_{j-1}/t_j)R_{j-2}(t_{j-2}/t_{j})\cdots R_1(t_1/t_j)\\
&\times\eta(\pi)(\gamma) R_{N-1}(qt_N/t_j)\cdots
R_{j+1}(qt_{j+2}/t_j)R_j(qt_{j+1}/t_j)
\end{split}
\end{equation*}
as rational $\textup{End}(H_0)$-valued functions in $(t,\gamma)\in T\times T$.\\
{\bf (ii)} For $1\leq i<N$ we have
\begin{equation*}
\begin{split}
C_{(\varpi_i,e)}(t,\gamma)&=\bigl(\eta(\pi)(\gamma)\bigr)^i
(R_{N-i}(qt_N/t_1)\cdots R_2(qt_{i+2}/t_1)R_1(qt_{i+1}/t_1))\\
&\times\cdots\times (R_{N-2}(qt_N/t_{i-1})\cdots R_{i}(qt_{i+2}/t_{i-1})
R_{i-1}(qt_{i+1}/t_{i-1}))\\
&\qquad\times (R_{N-1}(qt_N/t_i)\cdots
R_{i+1}(qt_{i+2}/t_i)R_i(qt_{i+1}/t_i))
\end{split}
\end{equation*}
as rational $\textup{End}(H_0)$-valued functions in $(t,\gamma)\in T\times T$.\\
{\bf (iii)} We have
\[C_{(\varpi_N,e)}(t,\gamma)=\gamma^{\varpi_N}\textup{id}
\]
as rational $\textup{End}(H_0)$-valued functions in $(t,\gamma)\in T\times T$.
\end{lem}
\begin{proof}
(i) By the cocycle property of $C$ and by the expression
\eqref{ej} for $\epsilon_j\in W$, we obtain an explicit expression
for $C_{(\epsilon_j,e)}$ in terms of the $C_{(s_i,e)}$ ($1\leq
i<N$) and $C_{(\pi,e)}$. Combining \eqref{Cpi} and the previous
lemma then gives the desired expression for
$C_{(\epsilon_j,e)}(t,\gamma)$.\\
(ii) $\sigma^i$ is the permutation
$$\left(\begin{array}{cccccccc}
    1   & 2   & \cdots & N-i & N-i+1 & N-i+2 & \cdots & N\\
    i+1 & i+2 & \cdots & N   & 1     & 2     & \cdots & i
    \end{array}\right),$$
so we find a reduced expression
\begin{equation}\label{redex}
\sigma^i=(s_i\cdots s_{N-1})(s_{i-1}\cdots s_{N-2})\cdots
(s_2\cdots s_{N-i+1})(s_1\cdots s_{N-i}).
\end{equation}
Combined with \eqref{omegasigma} we get a reduced expression for
$\varpi_i$. Using the cocycle condition for $C_{\mathrm{w}}$
repeatedly, we get the desired
result.\\
(iii) Since $\sigma^N=1$, we get $C_{(\varpi_N,e)}(t,\gamma)=
\bigl(\eta(\pi)(\gamma)\bigr)^N$, which maps $T_w$ to
$\gamma^{\varpi_N}T_w$ for all $w\in S_N$ in view of Lemma
\ref{etaexplicit}.
\end{proof}

We end this subsection by computing the asymptotic leading terms
of the $q$-connection matrices $C_{(\lambda,e)}(t,\gamma)$
($\lambda\in\Z^N$) as $|t^{-\alpha_i}|\rightarrow 0$ ($1\leq
i<N$), where we take $\alpha_i:=\epsilon_i-\epsilon_{i+1}$ ($1\leq
i<N$) as a base of the root system
$R=\{\epsilon_i-\epsilon_j\}_{1\leq i\not=j\leq N}$ of type
$A_{N-1}$. Let $R_+=\{\epsilon_i-\epsilon_j\}_{1\leq i<j\leq N}$
denote the associated set of positive roots and
$Q_+=\bigoplus_{i=1}^{N-1}\mathbb{Z}_{\geq 0}\alpha_i$ the
corresponding cone in the root lattice $Q$ of $R$. Let furthermore
$\delta:=(N-1,N-3,\ldots,1-N)\in\Z^N$ and write $w_0\in S_N$ for
the longest Weyl group element (mapping $i$ to $N-i+1$ for $1\leq
i\leq N$).

Consider the subring
$\mathcal{A}:=\C[x^{-\alpha_1},\ldots,x^{-\alpha_{N-1}}]$ of
$\C[T\times\{1\}]=\C[x_1^{\pm 1},\ldots,x_N^{\pm
1}]\subset\C[T\times T]$. We write $Q(\mathcal{A})$ for its
quotient field and $Q_0(\mathcal{A})$ for the subring of
$Q(\mathcal{A})$ consisting of rational functions which are
analytic at the point $x^{-\alpha_i}=0$ ($1\leq i<N$). We consider
$Q_0(\mathcal{A})\otimes\C[T]$ as a subring of $\C(T\times T)$ in
the natural way. The first part of the following corollary is a
refinement of Lemma \ref{singularC}(i) in case $w\in\Z^N$.
\begin{cor}\label{asymptotic}
Let $\lambda\in\Z^N$.
We have
\begin{equation}\label{convok}
C_{(\lambda,e)}\in (Q_0(\mathcal{A})\otimes \C[T])\otimes
\textup{End}(H_0)).
\end{equation}
Writing
\[C_{(\lambda,e)}^{(0)}=
C_{(\lambda,e)}|_{x^{-\alpha_1}=0,\ldots,x^{-\alpha_{N-1}}=0}
\in\C[T]\otimes\textup{End}(H_0),
\]
we have $C_{(\lambda,e)}^{(0)}=k^{\langle\delta,\lambda\rangle}
\eta(T_{w_0}Y^{w_0(\lambda)}T_{w_0}^{-1})$, where $\langle\cdot,\cdot\rangle$
is the standard scalar product on $\mathbb{R}^N$.
\end{cor}
\begin{proof}
To prove \eqref{convok} it suffices, in view of the cocyle
property of $C$, to verify \eqref{convok} for $\lambda=\epsilon_i$.
The statement then follows from Lemma \ref{explicit}(i),
Lemma \ref{etaexplicit} and the explicit expression of $R_i(z)$.

Observe that $\lim_{z\rightarrow 0}R_i(z)=k\eta(T_i^{-1})$ for
$1\leq i<N$. Combined with Lemma \ref{explicit}(i) and the
explicit expression for $Y_j$ (see \eqref{eqYdef}) we obtain for
$1\leq j\leq N$,
\[C_{(\epsilon_j,e)}(t,\gamma)\rightarrow k^{2j-N-1}\eta(Y_j)(\gamma)
\]
as $|t^{\alpha_i}|\rightarrow 0$ for all $1\leq i<N$, hence
\[C_{(\lambda,e)}(t,\gamma)\rightarrow k^{-\langle\delta,\lambda\rangle}
\eta(Y^\lambda)(\gamma)
\]
as $|t^{\alpha_i}|\rightarrow 0$ for all $1\leq i<N$.
To derive the asymptotics of $C_{(\lambda,e)}(t,\gamma)$ as
$|t^{-\alpha_i}|\rightarrow 0$ for $1\leq i<N$ we use the
cocycle property to write
\[C_{(\lambda,e)}(t,\gamma)=C_{(w_0,e)}(t,\gamma)
C_{(w_0(\lambda),e)}(w_0t,\gamma)C_{(w_0,e)}(q^{-w_0(\lambda)}w_0t,\gamma).
\]
Note that $C_{(w_0,e)}(t,\gamma)\rightarrow
k^{\ell(w_0)}\eta(T_{w_0}^{-1})$ if $|t^{\alpha_i}|\rightarrow 0$
for all $1\leq i<N$. Hence
\[C_{(\lambda,e)}^{(0)}=k^{\langle\delta,\lambda\rangle}\eta(T_{w_0}
Y^{w_0(\lambda)}T_{w_0}^{-1}),
\]
as desired.
\end{proof}

\subsection{Relation to quantum KZ equations}
Fix $\zeta\in T$ and let $\chi_\zeta\colon\C[\{1\}\times
T]\rightarrow \C$ denote the corresponding evaluation character
$\chi_\zeta(f)=f(\zeta)$. Recall the formal principal series
$\eta:H\rightarrow \textup{End}_{\C[\{1\}\times T]}(M)$. The
corresponding complex $H$-representation
$M(\zeta)=\C\otimes_{\chi_\zeta}M$ of dimension $N!$ is the
principal series module of $H$ with central character $\zeta$. We
identify $M(\zeta)$ with $H_0$ as a complex vector space, and push
the $H$-action on $M(\zeta)$ through the linear isomorphism to
$H_0$. We denote the corresponding representation map by
\[\eta_\zeta: H\rightarrow \textup{End}(H_0).
\]
As in Subsection \ref{formal} we have as identities in $H_0$,
\begin{equation*}
\eta_\zeta(T_i)T_w=
\begin{cases}
T_{s_iw}\qquad &\hbox{ if } \ell(s_iw)=\ell(w)+1,\\
(k-k^{-1})T_w+T_{s_iw}\qquad &\hbox{ if } \ell(s_iw)=\ell(w)-1,
\end{cases}
\end{equation*}
for $1\leq i<N$ and $w\in S_N$,
\[
\eta_\zeta(\pi)T_w=\zeta_{w^{-1}(N)}T_{\sigma w},
\]
for $w\in S_N$, as well as
\[\eta_\zeta(f(Y))\xi_w(\zeta)=(w^{-1}f)(\zeta)\xi_w(\zeta),
\]
for $w\in S_N$, where $\xi_w(\zeta)\in H_0$ is the regular
$H_0$-valued function $\xi_w(\gamma)$ in $\gamma\in T$ specialized
at $\gamma=\zeta$. Extending the base field to $\mathcal{M}(T)$ we
get an algebra homomorphism $H\rightarrow
\textup{End}_{\mathcal{M}(T)}(H_0^{\mathcal{M}(T)})$, which is
also denoted by $\eta_\zeta$.

In this subsection we consider the BqKZ for specialized values of
$\gamma$. In view of Lemma \ref{singularC}(i) we may specialize
$C_{(w,e)}(t,\gamma)$ ($w\in W$) at $\gamma=\zeta$. We write
\[C_{w}^\zeta(t):=C_{(w,e)}(t,\zeta),\qquad w\in W,
\]
for the resulting specialized cocycle values, viewed as
$\textup{End}(H_0)$-valued rational functions in $t\in T$.
For $\zeta\in T$ the map $W\ni w\mapsto C_w^\zeta$ defines
a cocycle of $W$ with values in the $W$-group $\textup{GL}_{\C(T)}(H_0^{\C(T)})$.
In other words,
\[C_{ww^\prime}^\zeta(t)=C_w^\zeta(t)C_{w^\prime}^\zeta(w^{-1}t),
\qquad w,w^\prime\in W,
\]
as rational $H_0$-valued functions in $t\in T$. Comparing the
cocycle values $C_{\lambda}^\zeta$ ($\lambda\in \Z^N$) to the ones
in \cite[\S 1.3]{C} we obtain the following result.
\begin{cor}\label{qKZcor}
Fix $\zeta\in T$. The holonomic system of $q$-difference equations
\begin{equation}\label{qKZ}
C_{\lambda}^\zeta(t)f(q^{-\lambda}t)=f(t),\qquad
\forall\lambda\in\Z^N
\end{equation}
for $f\in H_0^{\mathcal{M}(T)}$ is Cherednik's quantum affine
KZ equation associated to the principal $H$-module $M(\zeta)$ with
central character $\zeta$.
\end{cor}
Let $\textup{SOL}_\zeta\subset H_0^{\mathcal{M}(T)}$ denote the
set of solutions of the quantum KZ equations \eqref{qKZ}. Write
$\mathcal{E}(T)\subset\mathcal{M}(T)$ for the subfield of
meromorphic functions $f$ satisfying $f(q^\lambda t)=f(t)$ for all
$\lambda\in\Z^N$ as meromorphic functions in $t\in T$. The set
$\textup{SOL}_\zeta$ of solutions is a $\mathcal{E}(T)$-subspace
of $H_0^{\mathcal{M}(T)}$. Furthermore, $\textup{SOL}_\zeta$ is
invariant for the $S_N$-action
\begin{equation}\label{actionone}
(\varsigma(w)f)(t):=C_{w}^\zeta(t)f(w^{-1}t), \qquad w\in S_N
\end{equation}
on $H_0^{\mathcal{M}(T)}$ (note that $\varsigma$ does not depend on
$\zeta$ since $C_w^\zeta(t)=C_{(w,e)}(t,\zeta)$ is independent
of $\zeta$ for $w\in S_N$).
\begin{rema}\label{QGperspective}
The quantum KZ equations \eqref{qKZ} are gauge equivalent to
Frenkel and Reshetikhin's \cite{FR}
quantum KZ equations associated with the $N$-fold
tensor product representation
$\C^N(t_1)\otimes\cdots\otimes\C^N(t_N)$
of the quantum affine algebra
$\mathcal{U}_k(\widehat{sl}_N)$,
where $\C^N(t_i)$ is the evaluation
representation of the vector representation $\C^N$
of $\mathcal{U}_k(sl_N)$
(see \cite[\S 1.3.2]{C} and \cite{EFK} for the details).
\end{rema}

In view of Corollary \ref{qKZcor} the BqKZ equations \eqref{BqKZ}
are a holonomic extension of the quantum KZ equations \eqref{qKZ} by
$q$-difference equations in the central character $\zeta$ of
$M(\zeta)$. These may be thought of as analogs of isomonodromy
transformations; in fact, in view of Lemma \ref{commoneig} and
Corollary \ref{asymptotic} the $q$-difference equations in $\zeta$
(which are essentially the quantum KZ equations again!)
are reminiscent of Schlesinger transformations. This should be
compared with the quantum isomonodromic interpretation of (rational)
KZ equations as quantizations of Schlesinger equations, see
\cite{Re} and \cite{Ha}.

From a different perspective we may think of the cocycle values
$C_{(e,w)}$ ($w\in W$) as shift operators, in the sense that they
map solutions of quantum KZ equations to solutions of quantum KZ
equations with respect to shifted central characters. To formulate
the precise result, we view in the following proposition
$\gamma\mapsto C_{(e,w)}(\cdot,\gamma)$ as
$\C[T]\otimes\textup{End}(H_0)$-valued rational function in
$\gamma\in T$.
\begin{prop}\label{shift}
Let $w\in W$ and $\zeta\in T$ such that $\gamma\mapsto
C_{(e,w)}(\cdot,\gamma)$ is regular at $\gamma=\zeta$. Then
$f\mapsto C_{(e,w)}(\cdot,\zeta)f$ defines an $S_N$-equivariant linear
map $\textup{SOL}_{w^\diamondsuit{}^{-1}\zeta}\rightarrow \textup{SOL}_\zeta$.
\end{prop}
\begin{proof}
By the cocycle property we have for
$f\in\textup{SOL}_{w^\diamondsuit{}^{-1}\zeta}$ and $\lambda\in\Z^N$,
\begin{equation*}
\begin{split}
C_{\lambda}^\zeta(t)\bigl(C_{(e,w)}
(q^{-\lambda}t,\zeta)f(q^{-\lambda}t)\bigr)
&=C_{(\lambda,w)}(t,\zeta)f(q^{-\lambda}t)\\
&=C_{(e,w)}(t,\zeta)
C_{\lambda}^{w^\diamondsuit{}^{-1}\zeta}(t)f(q^{-\lambda}t)\\
&=C_{(e,w)}(t,\zeta)f(t).
\end{split}
\end{equation*}
Hence $C_{(e,w)}(\cdot,\zeta)f\in\textup{SOL}_{\zeta}$.
The $S_N$-equivariance of the map is again a consequence of the cocycle
property of $C_{\mathrm{w}}$ ($\mathrm{w}\in\mathbb{W}$); indeed, for
$v\in S_N$ and $f\in \textup{SOL}_{w^\diamondsuit{}^{-1}\zeta}$ we have
\begin{equation*}
\begin{split}
C_{v}^\zeta(t)\bigl(C_{(e,w)}(v^{-1}t,\zeta)f(v^{-1}t)\bigr)
&=C_{(v,w)}(t,\zeta)f(v^{-1}t)\\
&=C_{(e,w)}(t,\zeta)\bigl(C_v^{w^\diamondsuit{}^{-1}\zeta}(t)
f(v^{-1}t)\bigr),
\end{split}
\end{equation*}
which is the desired result.
\end{proof}
From the quantum group perspective (see Remark
\ref{QGperspective}), Proposition \ref{shift} resembles the action
of the dynamical Weyl group on solutions of quantum KZ equations
from \cite{EV}. We expect that the second half of the BqKZ is
closely related to the Varchenko-Etingof dynamical difference
equations \cite[\S 9]{EV}; see also \cite{FMTV}, \cite{TL},
\cite{T}, \cite{TV}, \cite{TV2} and \cite{Le} for detailed studies
of various degenerate cases. An interesting aspect, e.g., in
\cite{TL} and \cite{TV}, is the observation that KZ equations are
dual to the associated dynamical equations using
$(\mathfrak{gl}_r,\mathfrak{gl}_s)$ duality (our set-up relates to
$r=s=N$). In the present theory this duality is incorporated by
the cocycle value $C_\iota$, which relates the $q$-connection
matrices $C_{(\lambda,e)}$ ($\lambda\in\Z^N$) of the quantum KZ
equation to the dual $q$-connection matrices $C_{(e,\lambda)}$ by
conjugation,
\[C_{(e,\lambda)}(t,\gamma)=C_\iota C_{(\lambda,e)}(\gamma^{-1},t^{-1})C_\iota
\]
as $\textup{End}(H_0)$-valued meromorphic functions in
$(t,\gamma)\in T\times T$. In turn, $C_\iota$ is a direct reflection
of (the existence of) Cherednik's duality anti-isomorphism of the double
affine Hecke algebra (see Theorem \ref{dualthm}).

\section{Solutions of the bispectral quantum KZ equations}
\label{SectionSols}

In this section we use asymptotic analysis to construct a
$\iota$-invariant solution $\Phi_\kappa$ of BqKZ, which we call
the basic asymptotically free solution. It depends in a mild way
on an auxiliary parameter $\kappa\in\C^\times$ (in fact,
$\mathbb{F}\Phi_\kappa$ is independent of $\kappa$). The orbit of
$\Phi_\kappa$ under the action of $\{e\}\times
S_N\subset\mathbb{S}_N$ turns out to be an $\mathbb{F}$-basis of
$\textup{SOL}$ consisting of asymptotically free solutions. Along
the way we derive various additional properties of $\Phi_\kappa$.
\subsection{The leading term}
Let $\theta\in\mathcal{M}(T)$ denote the renormalized Jacobi theta
function
\begin{equation}\label{theta}
\theta(z):=\prod_{m\geq0}(1-q^mz)(1-q^{m+1}/z)
\end{equation}
for $z\in \C^\times$. It satisfies
\begin{equation}\label{thetafunc}
\theta(q^mz)=(-z)^{-m}q^{-\frac{1}{2}m(m-1)}\theta(z),\qquad m\in\Z.
\end{equation}
For $\kappa\in\C^\times$ we define $W_\kappa\in\mathbb{K}$ by
\begin{equation}\label{W}
W_\kappa(t,\gamma):=\prod_{i=1}^N\frac{\theta(\kappa t_i\gamma_{N-i+1}^{-1})}
{\theta(\kappa k^{\langle\delta,\epsilon_i\rangle}t_i)
\theta(\kappa k^{-\langle\delta,\epsilon_i\rangle}\gamma_{N-i+1}^{-1})}.
\end{equation}
By Corollary \ref{asymptotic}, the formal asymptotic form of the quantum
KZ equations
\[C_{(\lambda,e)}(t,\gamma)f(q^{-\lambda}t,\gamma)=f(t,\gamma),\qquad
\lambda\in\Z^N
\]
in the asymptotic region $|t^{\alpha_i}|\gg 0$ ($1\leq i<N$) is
\begin{equation}\label{qKZasymptotic}
k^{\langle\delta,\lambda\rangle}\eta(T_{w_0}Y^{w_0(\lambda)}T_{w_0}^{-1})
(\gamma)f(q^{-\lambda}t,\gamma)=f(t,\gamma),\qquad \lambda\in \Z^N.
\end{equation}

\begin{lem} \label{lemWproprts} $W_\kappa\in\mathbb{K}$
enjoys the following properties.
\begin{enumerate}
    \item[\bf{(i)}] $f_\kappa^{(0)}(t,\gamma):=W_\kappa(t,\gamma)T_{w_0}$ is a
solution of \eqref{qKZasymptotic}.
\item[{\bf (ii)}] $\iota(W_\kappa)=W_\kappa$ and $\tau(\iota)f_\kappa^{(0)}=
f_\kappa^{(0)}$.
\end{enumerate}
\end{lem}
\begin{proof}
(i) Since
$\eta(T_{w_0}Y^{w_0(\lambda)}T_{w_0}^{-1})(\gamma)T_{w_0}=
\gamma^{w_0(\lambda)}T_{w_0}$ for all $\lambda\in\Z^N$, it
suffices to show that
\[W_\kappa(q^{-\lambda}t,\gamma)=k^{-\langle\delta,\lambda\rangle}
\gamma^{-w_0(\lambda)}W_\kappa(t,\gamma),\qquad \lambda\in\Z^N,
\]
which follows from \eqref{thetafunc}.

(ii) Clearly $\iota(W_\kappa)=W_\kappa$, i.e.
$W_\kappa(\gamma^{-1},t^{-1})=W_\kappa(t,\gamma)$. Since
$C_\iota(T_{w_0})=T_{w_0}$, it follows that
$\tau(\iota)f_\kappa^{(0)}=f_\kappa^{(0)}$.
\end{proof}
Observe that, more generally, $W_\kappa\in\mathbb{K}$ satisfies
the $q$-difference equations
\begin{equation}\label{fff}
W_\kappa(q^{-\lambda}t,q^{\mu}\gamma)=
k^{-\langle\delta,\lambda+\mu\rangle}t^{w_0(\mu)}\gamma^{-w_0(\lambda)}
q^{-\langle w_0(\lambda),\mu\rangle}W_\kappa(t,\gamma),\qquad
\lambda,\mu\in\Z^N.
\end{equation}

\subsection{The basic asymptotically free solution $\Phi_\kappa$}
We now gauge BqKZ by $W_\kappa\in\mathbb{K}$.
Concretely, for $\lambda,\mu\in\Z^N$ we write
\[D_{(\lambda,\mu)}(t,\gamma)=W_\kappa(t,\gamma)^{-1}C_{(\lambda,\mu)}
(t,\gamma)W_\kappa(q^{-\lambda}t,q^{\mu}\gamma)
\]
as $\textup{End}(H_0)$-valued meromorphic functions in
$(t,\gamma)\in T\times T$. It is independent of $\kappa$ in view
of \eqref{fff}. For $f\in H_0^{\mathbb{K}}$ we have
$f\in\textup{SOL}$ if and only if $g:=W_{\kappa}^{-1}f\in
H_0^{\mathbb{K}}$ satisfies the holonomic system of $q$-difference
equations
\begin{equation}\label{gaugedeqn}
D_{(\lambda,\mu)}(t,\gamma)g(q^{-\lambda}t,q^{\mu}\gamma)=g(t,\gamma),\qquad
\lambda,\mu\in\Z^N
\end{equation}
as $H_0$-valued rational functions in $(t,\gamma)\in T\times T$.

The existence of a solution $\Psi\in H_0^{\mathbb{K}}$ of \eqref{gaugedeqn}
admitting a convergent $H_0$-valued
power series expansion
\begin{equation}\label{convseries}
\Psi(t,\gamma)=\sum_{\alpha,\beta\in Q_+}K_{\alpha,\beta}t^{-\alpha}
\gamma^\beta,\qquad K_{0,0}=T_{w_0}
\end{equation}
in the asymptotic region $|t^{\alpha_i}|\gg 0$ and
$|\gamma^{-\alpha_i}|\gg 0$ ($1\leq i<N$) is guaranteed by the
following properties of the gauged $q$-connection matrices $D_{(\lambda,\mu)}$.

Consider the subring
$\mathcal{B}:=\C[y^{\alpha_1},\ldots,y^{\alpha_{N-1}}]$ of
$\C[\{1\}\times T]=\C[y_1^{\pm 1},\ldots,y_N^{\pm 1}]$.
Write $Q(\mathcal{B})$
for its quotient field and
$Q_0(\mathcal{B})$ for the subring of $Q(\mathcal{B})$
consisting of rational functions
which are analytic at the point $y^{\alpha_j}=0$ ($1\leq j<N$).
We consider $Q_0(\mathcal{A})\otimes\mathcal{B}$ and $\mathcal{A}\otimes
Q_0(\mathcal{B})$ as subrings of $\C(T\times T)$ in the natural way.
\begin{lem}\label{qconnLem}
Set $A_i=D_{(\varpi_i,e)}$ and
$B_i=D_{(e,\varpi_i)}$ for $1\leq i\leq N$.
\begin{enumerate}
\item[{\bf (i)}] $A_N=B_N=\textup{id}$ on $H_0^{\mathbb{K}}$.
\item[{\bf (ii)}] $A_i\in(Q_0(\mathcal{A})\otimes
\mathcal{B})\otimes\textup{End}(H_0)$ and
$B_j\in (\mathcal{A}\otimes Q_0(\mathcal{B}))\otimes
\textup{End}(H_0)$.
\item[{\bf (iii)}] Set $A_i^{(0,0)}\in\textup{End}(H_0)$ and
$B_j^{(0,0)}\in\textup{End}(H_0)$ for the value of $A_i$ and $B_j$
at $x^{-\alpha_r}=0=y^{\alpha_s}$ ($1\leq r,s<N$).
For $w\in S_N$ we have
\begin{equation}\label{Aas}
A_i^{(0,0)}(T_{w_0}T_w)=
\begin{cases}
0\quad &\hbox{ if }\,\, w^{-1}w_0(\varpi_i)\not=w_0(\varpi_i),\\
T_{w_0}T_w\quad &\hbox{ if }\,\, w^{-1}w_0(\varpi_i)=w_0(\varpi_i)
\end{cases}
\end{equation}
and
\begin{equation}\label{Bas}
B_i^{(0,0)}(T_{w_0}T_w)=
\begin{cases}
0\quad &\hbox{ if }\,\, w(\varpi_i)\not= \varpi_i,\\
T_{w_0}T_w\quad &\hbox{ if }\,\, w(\varpi_i)=\varpi_i.
\end{cases}
\end{equation}
\end{enumerate}
\end{lem}
\begin{proof}
(i) We only give the proof of $A_N=\textup{id}$. Since
$\varpi_N=\pi^N$ in $W$, we have
\[
A_N(t,\gamma)=W_\kappa(t,\gamma)^{-1}C_{(\varpi_N,e)}(t,\gamma)
W_\kappa(q^{-\varpi_N}t,\gamma)=\gamma^{-\varpi_N}(\eta(\pi)(\gamma))^N
=\textup{id}
\]
where we use \eqref{fff} and \eqref{Cpi} for the second equality,
and Lemma \ref{etaexplicit} and $\sigma^N=e$ for the third equality.

(ii)  Note that
\[A_i(t,\gamma)=k^{-\langle\delta,\varpi_i\rangle}\gamma^{-w_0(\varpi_i)}
C_{(\varpi_i,e)}(t,\gamma)
\]
by \eqref{fff}. Since $\varpi_i=\pi^i\sigma^{-i}$ the
cocycle property of $C$ gives
\[A_i(t,\gamma)=k^{-\langle\delta,\varpi_i\rangle}
\gamma^{-w_0(\varpi_i)}\bigl(\eta(\pi)(\gamma)\bigr)^i
C_{(\sigma^{-i},e)}(\pi^{-i}t,\gamma).
\]
It follows from
the explicit expressions for the cocycle values $C_{(s_i,e)}$ ($1\leq i<N$)
that the $\textup{End}(H_0)$-valued
rational function $C_{(\sigma^{-i},e)}(\pi^{-i}t,\gamma)$ in
$(t,\gamma)\in T\times T$ lies in
$Q_0(\mathcal{A})\otimes\textup{End}(H_0)$
(in particular, it is independent of $\gamma$).
Furthermore, for $w\in S_N$
\[
\gamma^{-w_0(\varpi_i)}\bigl(\eta(\pi)(\gamma)\bigr)^i\bigl(T_w\bigr)
=\gamma^{w^{-1}w_0(\varpi_i)-w_0(\varpi_i)}T_{\sigma^iw}
\]
by Lemma \ref{etaexplicit}, hence the $\textup{End}(H_0)$-valued
regular function $\gamma^{-w_0(\varpi_i)}\bigl(\eta(\pi)(\gamma)\bigr)^i$
in $\gamma\in T$ lies in $\mathcal{B}\otimes\textup{End}(H_0)$.
Consequently, $A_i\in (Q_0(\mathcal{A})\otimes\mathcal{B})\otimes
\textup{End}(H_0)$.
The statement for $B_j$ follows from this using the cocycle property
$C_{(e,\varpi_j)}(t,\gamma)=C_\iota C_{(\varpi_j,e)}
(\gamma^{-1},t^{-1})C_\iota$.

(iii) Recall that $\xi_w=\eta(\widetilde{S}_{w^{-1}}^*)T_e$ with
$\widetilde{S}_w$ the intertwiners of $\mathbb{H}$ (see
Proposition \ref{lemIntProp}). By induction on $\ell(w)$, using
the explicit expression \eqref{Si} of the intertwiners
$\widetilde{S}_i$, it follows that $\xi_w\in \mathcal{B}\otimes
H_0$ and that the value of $\xi_w$ at $y^{\alpha_i}=0$ ($1\leq
i<N$) is $T_w\in H_0$. Set
\[A_i^{(0)}=A_i|_{x^{-\alpha_1}=0,\ldots,x^{-\alpha_{N-1}}=0}
\in\mathcal{B}\otimes\textup{End}(H_0).
\]
By Corollary \ref{asymptotic} and \eqref{fff},
\[
A_i^{(0)}=y^{-w_0(\varpi_i)}\eta(T_{w_0}Y^{w_0(\varpi_i)}T_{w_0}^{-1}).
\]
Lemma \ref{commoneig} then gives
\begin{equation}\label{afterstepone}
A_i^{(0)}\bigl(\eta(T_{w_0})\xi_w\bigr)=
y^{w^{-1}w_0(\varpi_i)-w_0(\varpi_i)}\eta(T_{w_0})\xi_w,\qquad
\forall\, w\in S_N
\end{equation}
as identities in $\mathcal{B}\otimes H_0$. Specializing \eqref{afterstepone}
at $y^{\alpha_j}=0$ ($1\leq j<N$) yields \eqref{Aas}.

To prove \eqref{Bas} we consider
\[\widetilde{B}_j^{(0)}=B_j|_{y^{\alpha_1}=0,\ldots,y^{\alpha_{N-1}}=0}\in
\mathcal{A}\otimes\textup{End}(H_0).
\]
It is the rational $\textup{End}(H_0)$-valued function
\[
\widetilde{B}_j^{(0)}(t)=t^{w_0(\varpi_j)}C_\iota
\bigl(\eta(T_{w_0}Y^{w_0(\varpi_j)}T_{w_0}^{-1})(t^{-1})\bigr)C_\iota
\]
in $t\in T$. Denoting $\widetilde{\xi}_w\in\mathcal{A}\otimes H_0$
for the rational $H_0$-valued function $\xi_w(t^{-1})$ in $t\in
T$, it follows that
\begin{equation}\label{afterstepone2}
\widetilde{B}_j^{(0)}\bigl(C_\iota\eta(T_{w_0})\widetilde{\xi}_w\bigr)=
x^{-w^{-1}w_0(\varpi_j)+w_0(\varpi_j)}C_\iota\eta(T_{w_0})\widetilde{\xi}_w
\end{equation}
for all $w\in S_N$. The value of
$C_\iota\eta(T_{w_0})\widetilde{\xi}_w$ at $x^{-\alpha_i}=0$
($1\leq i<N$) is $C_\iota(T_{w_0}T_w)$. In addition, $C_\iota$
restricts to the anti-algebra involution on $H_0$ mapping $T_w$ to
$T_{w^{-1}}$ for $w\in S_N$, hence
\[C_{\iota}(T_{w_0}T_w)=T_{w^{-1}}T_{w_0}=T_{w_0}T_{w_0w^{-1}w_0}.
\]
Formula \eqref{Bas} then follows from specializing
\eqref{afterstepone2} at $x^{-\alpha_i}=0$ ($1\leq i<N$) and
replacing $w$ by $w_0w^{-1}w_0$
in the resulting formula.
\end{proof}
For $\epsilon>0$, put $B_{\epsilon}:=\{t\in T\mid
|t^{\alpha_i}|<\epsilon,\: \forall i\}$ and
$B_{\epsilon}^{-1}:=\{t\in T\mid t^{-1}\in B_\epsilon\}$.
\begin{thm}\label{asymTHM}
There exists a unique solution $\Psi\in H_0^{\mathbb{K}}$ of the
gauged equations \eqref{gaugedeqn} satisfying, for some
$\epsilon>0$,\\
{\bf (i)} $\Psi(t,\gamma)$ admits an $H_0$-valued power series
expansion
    \begin{equation}\label{ps}
        \Psi(t,\gamma)=\sum_{\alpha,\beta\in
        Q_+}K_{\alpha,\beta}t^{-\alpha}\gamma^{\beta},\qquad
        (K_{\alpha,\beta}\in H_0)
    \end{equation}
for $(t,\gamma)\in B_\epsilon^{-1}\times B_\epsilon$ which is
normally convergent on compacta of $B_\epsilon^{-1}\times
B_\epsilon$. In particular, $\Psi(t,\gamma)$ is analytic at
$(t,\gamma)\in B_\epsilon^{-1} \times B_\epsilon$;\\
{\bf(ii)} $K_{0,0}=T_{w_0}$.
\end{thm}
\begin{proof}
It follows from the previous lemma that the commuting
endomorphisms $A_i^{(0,0)}, B_j^{(0,0)}\in\textup{End}(H_0)$
($1\leq i,j<N$) are semisimple. For $a,b\in\C^{N-1}$ set
\[H_0[(a,b)]=\{ v\in H_0 \,\, | \,\, A_i^{(0,0)}v=a_iv \,\, \hbox{ and } \,\,
B_j^{(0,0)}v=b_{j}v\,\, (1\leq i,j<N)\},
\]
so that $H_0=\bigoplus_{(a,b)\in S}H_0[(a,b)]$ with $S$ the finite
set of $(a,b)\in\C^{N-1}\times\C^{N-1}$ for which
$H_0[(a,b)]\not=0$. By the previous lemma, $(1^{N-1},1^{N-1})\in
S$ and we have
$H_0[(1^{N-1},1^{N-1})]=\textup{span}_{\C}\{T_{w_0}\}$.
Furthermore, $a_i,b_i\not\in q^{-\mathbb{N}}$ for all $(a,b)\in S$
and $i$. Under these conditions, the holonomic system of
$q$-difference equations \eqref{gaugedeqn} admits a unique
solution $\Psi$ satisfying the desired properties; see Theorem
\ref{merversion} in the appendix (to show that the gauged BqKZ
falls in the class of holonomic systems of $q$-difference
equations to which Theorem \ref{merversion} applies, one should
take $M=2(N-1)$, $q_i=q$ for $1\leq i<N$ and variables
$z_i=x^{-\alpha_i}$ and $z_{N-1+j}=y^{\alpha_j}$ for $1\leq i,j<N$
in the appendix).
\end{proof}
\begin{rema}\label{merocontRemark}
In a small neighborhood of a fixed $(t^\prime,\gamma^\prime)\in
T\times T$, the meromorphic solution $\Psi$ of \eqref{gaugedeqn}
can be expressed in terms of the power series expansion \eqref{ps}
by the formula
\begin{equation*}
\begin{split}
\Psi(t,\gamma)&=
D_{(\lambda,\mu)}(t,\gamma)\Psi(q^{-\lambda}t,q^{\mu}\gamma)\\
&=D_{(\lambda,\mu)}(t,\gamma)\sum_{\alpha,\beta\in Q_+}K_{\alpha,\beta}
(q^{-\lambda}t)^{-\alpha}(q^\mu\gamma)^\beta,
\end{split}
\end{equation*}
where $\lambda,\mu\in\Z^N$ are such that $(q^{-\lambda}t^\prime,
q^\mu\gamma^\prime)\in B_\epsilon^{-1}\times B_\epsilon$.
\end{rema}
\begin{defi}
We call $\Phi_\kappa:=W_\kappa \Psi\in\textup{SOL}$ the basic
asymptotically free solution of BqKZ.
\end{defi}
Note that $\Phi_\kappa\in\mathbb{F}^\times\Phi_{\kappa^\prime}$
for $\kappa,\kappa^\prime\in\C^\times$. The $\kappa$-flexibility
will come in handy when we consider specializations of
$\Phi_\kappa$. In the following subsections, we derive various
properties of the basic asymptotically free solution
$\Phi_\kappa$.

\subsection{Duality}

\begin{thm}\label{selfdualTHM}
The basic asymptotically free solution $\Phi_\kappa$ of
BqKZ is self-dual, in the sense that
\[\tau(\iota)\Phi_\kappa=\Phi_\kappa.
\]
\end{thm}
\begin{proof}
$\textup{SOL}$ is $\mathbb{S}_N$-invariant, hence
$\tau(\iota)\Phi_\kappa\in \textup{SOL}$. In addition,
\[\tau(\iota)\Phi_\kappa=W_\kappa(\tau(\iota)\Psi)
\]
because $\iota(W_\kappa)=W_\kappa$. Hence $\tau(\iota)\Psi$ is a
solution of the gauged equations \eqref{gaugedeqn} having a
convergent $H_0$-valued power series expansion
\[(\tau(\iota)\Psi)(t,\gamma)=C_\iota\Psi(\gamma^{-1},t^{-1})=
\sum_{\alpha,\beta\in Q_+}C_\iota(K_{\alpha,\beta})\gamma^\alpha t^{-\beta}
\]
for $(t,\gamma)\in B_\epsilon^{-1}\times B_\epsilon$. Since
$C_\iota(K_{0,0})=C_\iota(T_{w_0})=T_{w_0}$, we conclude from
Theorem \ref{asymTHM} that $\tau(\iota)\Psi=\Psi$, hence
$\tau(\iota)\Phi_\kappa= \Phi_\kappa$.
\end{proof}

\subsection{Singularities}\label{singularities}
Define
\[\Lambda=\{\lambda\in\Z^N \,\, | \,\, \lambda_1\geq\lambda_2\geq
\cdots\geq\lambda_N\}=\bigoplus_{i=1}^{N-1}\Z_{\geq 0}\varpi_i
\oplus\Z \varpi_N,
\]
i.e., $\Lambda$ consists of the $\lambda\in\Z^N$ such that
$\langle\lambda,\alpha\rangle\in\Z_{\geq 0}$ for all $\alpha\in
R_+$. Set
\[\mathcal{S}_+:=\{t\in T \,\, | \,\, t^{\alpha}\in k^{-2}q^{-\mathbb{N}}\,
\textup{ for some }\, \alpha\in R_+ \}.
\]
Write $\Psi(t,\gamma)=\sum_{\alpha\in
Q_+}\Gamma_\alpha(\gamma)t^{-\alpha}$ for $(t,\gamma)\in
B_\epsilon^{-1}\times B_\epsilon$, where $\Gamma_\alpha$ is the $H_0$-valued
analytic function on $B_\epsilon$ defined by the $H_0$-valued
power series
\[
\Gamma_\alpha(\gamma):=\sum_{\beta\in Q_+}K_{\alpha,\beta}\gamma^{\beta}.
\]

\begin{lem}\label{Consequence1}
The $\Gamma_\alpha$ ($\alpha\in Q_+$) extend uniquely to a
meromorphic $H_0$-valued function on $T$, analytic at
$T\setminus\mathcal{S}_+$, such that $\Psi(t,\gamma)$ admits an
$H_0$-valued power series expansion
\[
\Psi(t,\gamma)=\sum_{\alpha\in Q_+}\Gamma_\alpha(\gamma)t^{-\alpha}
\]
for $(t,\gamma)\in B_\epsilon^{-1}\times T\setminus\mathcal{S}_+$,
converging normally on compacta of $B_\epsilon^{-1}\times
T\setminus\mathcal{S}_+$.
\end{lem}
\begin{proof}
Using Lemma \ref{qconnLem} we write for $\mu\in\Lambda$,
\[D_{(e,\mu)}(t,\gamma)=\sum_{\beta\in
Q_+}F_\beta^\mu(\gamma)t^{-\beta}
\]
with $F_\beta^\mu\in Q_0(\mathcal{B})\otimes\textup{End}(H_0)$ for
all $\beta\in Q^\vee_+$. Note that $F_\beta^\mu\equiv 0$ for all
but finitely many $\beta\in Q_+$.

We first show that $F_\beta^\mu(\gamma)$ is regular at $\gamma\in
T\setminus\mathcal{S}_+$. By \eqref{fff} and by the cocycle
property, $F_\beta^\mu(\gamma)$ is regular at $\gamma=\zeta$ if
$C_{(e,\varpi_j)}(\cdot,q^\nu\gamma)\in\C[T]\otimes
\textup{End}(H_0)$ is regular at $\gamma=\zeta$ for all $1\leq
j\leq N$ and $\nu\in\Lambda$. The latter statement follows from
the fact that $R_i(z)$ has only a (simple) pole at $z=k^{-2}$ and
from the explicit expression
\begin{equation}\label{fundamental}
\begin{split}
C_{(e,\varpi_j)}(t,\gamma)&=C_\iota
\bigl(\eta(\pi)(t^{-1})\bigr)^j
(R_{N-j}(q\gamma_1/\gamma_N)\cdots R_2(q\gamma_1/\gamma_{j+2})
R_1(q\gamma_1/\gamma_{j+1}))\\
&\times\cdots\times (R_{N-2}(q\gamma_{j-1}/\gamma_N)\cdots
R_{j}(q\gamma_{j-1}/\gamma_{j+2})R_{j-1}(q\gamma_{j-1}/\gamma_{j+1}))\\
&\qquad\,\,\,\times (R_{N-1}(q\gamma_j/\gamma_N)\cdots
R_{j+1}(q\gamma_j/\gamma_{j+2})R_j(q\gamma_j/\gamma_{j+1}))C_\iota,
\end{split}
\end{equation}
which follows from Lemma \ref{explicit}(ii) and the cocycle property
of $C$.

Let $U\subset T\setminus\mathcal{S}_+$ be a relatively compact
open subset. Choose $\mu\in \Lambda$ such that the closure of
$q^{\mu}U$ is contained in $B_\epsilon$, where
$q^{\mu}U:=\{q^\mu\gamma \, | \, \gamma\in U\}$. As meromorphic
$H_0$-valued function in $(t,\gamma)\in B_\epsilon^{-1}\times U$,
we have
\begin{equation*}
\begin{split}
\Psi(t,\gamma)&=D_{(e,\mu)}(t,\gamma)\Psi(t,q^\mu\gamma)\\
&=\sum_{\alpha,\beta\in Q_+}F_\beta^\mu(\gamma)
\bigl(\Gamma_\alpha(q^\mu\gamma)\bigr)
t^{-\alpha-\beta}\\
&=\sum_{\alpha\in Q_+} \Bigl(\sum_{\beta\in Q_+: \alpha-\beta\in
Q_+}
F_{\beta}^\mu(\gamma)\bigl(\Gamma_{\alpha-\beta}(q^\mu\gamma)\bigr)\Bigr)
t^{-\alpha},
\end{split}
\end{equation*}
with the sums converging normally on compacta of $B_\epsilon^{-1}\times U$
(note that the sums over $\beta$ are finite).
It follows that $\Gamma_\alpha$ ($\alpha\in Q_+$) has a unique $H_0$-valued
meromorphic extension to $T$ which, on $U$, is given by
\begin{equation}\label{contformula}
\Gamma_\alpha(\gamma)=\sum_{\beta\in Q_+: \alpha-\beta\in Q_+}
F_{\beta}^\mu(\gamma)\bigl(\Gamma_{\alpha-\beta}(q^\mu\gamma)\bigr),
\end{equation}
such that $\Psi$ on $B_\epsilon^{-1}\times U$
admits the power series expansion
\[\Psi(t,\gamma)=\sum_{\alpha\in Q_+}\Gamma_\alpha(\gamma)t^{-\alpha},
\]
which converges normally on compacta of $B_\epsilon^{-1}\times U$.
It follows from \eqref{contformula} and the previous paragraph
that $\Gamma_\alpha$ is analytic on $T\setminus\mathcal{S}_+$.
\end{proof}
The arguments from the proof of Lemma \ref{Consequence1}, applied
to both torus variables of $\Psi(t,\gamma)$ at the same time,
directly lead to the following result.
\begin{prop}\label{Psising}
The $H_0$-valued meromorphic function $\Psi(t,\gamma)$ is analytic at
$(t,\gamma)\in T\setminus\mathcal{S}_+^{-1}\times T\setminus\mathcal{S}_+$.
\end{prop}
For specialized spectral parameter, we obtain the following
result.
\begin{prop}\label{spgauge}
Let $\zeta\in T\setminus \mathcal{S}_+$.\\
{\bf (i)} The $H_0$-valued meromorphic
function $\Psi(t,\gamma)$ in $(t,\gamma)\in T\times T$ can be specialized
at $\gamma=\zeta$, giving rise to a meromorphic $H_0$-valued function
$\Psi(t,\zeta)$ in $t\in T$. It has the power series expansion
\[
\Psi(t,\zeta)=\sum_{\alpha\in Q_+}\Gamma_\alpha(\zeta)t^{-\alpha}
\]
for $t\in B_\epsilon^{-1}$, normally converging on compacta of
$B_\epsilon^{-1}$.\\
{\bf (ii)} $\Psi(t,\zeta)$ satisfies the gauged $q$-difference equations
\begin{equation}\label{specialgauge}
D_{(\lambda,e)}(t,\zeta)\Psi(q^{-\lambda}t,\zeta)=\Psi(t,\zeta),\qquad
\forall \lambda\in\Z^N.
\end{equation}
\end{prop}
\begin{proof}
(i) Restricting to $t\in B_\epsilon^{-1}$ for $\epsilon>0$ small
enough, the statement is correct by Lemma \ref{Consequence1}. If
$t^\prime\in T$ is arbitrary then there exists a
$\lambda\in\Lambda$ such that $q^{-\lambda}t^\prime\in
B_\epsilon^{-1}$. For $t\in T$ in a small neighborhood of
$t^\prime$ we then have
\[\Psi(t,\gamma)=D_{(\lambda,e)}(t,\gamma)\Psi(q^{-\lambda}t,\gamma).
\]
Since $D_{(\lambda,e)}\in
(Q_0(\mathcal{A})\otimes\mathcal{B})\otimes\textup{End}(H_0)$
by Lemma \ref{qconnLem}(ii) the statement now follows
in a small open neighborhood of $t^\prime$.\\
(ii) Specializing the gauged $q$-difference equations
$D_{(\lambda,e)}(t,\gamma)\Psi(q^{-\lambda}t,\gamma)=\Psi(t,\gamma)$
($\lambda\in\Z^N$) to $\gamma=\zeta$ yields the desired result.
\end{proof}

\subsection{Evaluation formula}
We write $\bigl(z;q\bigr)_\infty=\prod_{m=0}^{\infty}(1-q^mz)$ for the
$q$-shifted factorial. Recall the power series expansion
$\Psi(t,\gamma)=\sum_{\alpha\in Q_+}\Gamma_\alpha(\gamma)t^{-\alpha}$
for $|t^{\alpha_i}|\gg 0$ ($1\leq i<N$)
from Subsection \ref{singularities}. We call the following result the
evaluation formula for the basic asymptotically free solution
$\Phi_\kappa=W_\kappa\Psi$ of BqKZ, since
it implies the celebrated evaluation formula for the Macdonald
polynomials (see Subsection \ref{2}).
\begin{thm}\label{duall}
We have
\[
\Gamma_0(\gamma)=K(\gamma)T_{w_0}
\]
with $K\in\mathcal{M}(T)$ explicitly given by
\begin{equation}\label{Kexp}
K(\gamma):=\prod_{1\leq i<j\leq N}
\frac{\bigl(q\gamma_i/\gamma_j;q\bigr)_{\infty}}
{\bigl(qk^2\gamma_i/\gamma_j;q\bigr)_{\infty}}.
\end{equation}
\end{thm}
\begin{proof}
We use the notations of Lemma
\ref{qconnLem}. Recall that $\Psi$ satisfies the gauged $q$-difference
equations
\[A_i(t,\gamma)\Psi(q^{-\varpi_i}t,\gamma)=\Psi(t,\gamma)
\]
for $1\leq i\leq N$. In view of the proof of Lemma \ref{qconnLem}
and Lemma \ref{Consequence1}, it reduces in the limit
$|t^{-\alpha_i}|\rightarrow 0$ ($1\leq i<N$) to
\[\gamma^{-w_0(\varpi_i)}\eta(T_{w_0}Y^{w_0(\varpi_i)}T_{w_0}^{-1})(\gamma)
\Gamma_0(\gamma)=\Gamma_0(\gamma)
\]
for $1\leq i\leq N$, as $H_0$-valued meromorphic functions in
$\gamma\in T$. This forces $\Gamma_0(\gamma)=
K(\gamma)\eta(T_{w_0})\xi_e(\gamma)=K(\gamma)T_{w_0}$ for some
$K\in\mathcal{M}(T)$; see Lemma \ref{commoneig}.

It remains to show that $K$ is explicitly given by \eqref{Kexp}.
Write $L(\gamma)$ for the right hand side of \eqref{Kexp}. Then
$L\in\mathcal{M}(T)$ is characterized by the following three
properties:
\begin{enumerate}
\item[{\bf (i)}] for some $\epsilon>0$ we have a power series
expansion
\[
L(\gamma)=\sum_{\alpha\in Q_+}l_\alpha\gamma^{\alpha}
\]
for $\gamma\in B_\epsilon$, converging normally on compacta of $B_\epsilon$;
\item[{\bf (ii)}] $l_0=1$; and
\item[{\bf (iii)}] $L(\gamma)$ satisfies the $q$-difference equations
\[\prod_{\stackrel{1\leq r\leq j}{\stackrel{j+1\leq s\leq N}{}}}
\frac{1-q\gamma_r/\gamma_s}
{1-qk^2\gamma_r/\gamma_s}L(q^{\varpi_j}\gamma)=L(\gamma), \qquad
1\leq j\leq N.
\]
\end{enumerate}
It thus suffices to show that $K(\gamma)$ satisfies the three
properties (i)--(iii). It is clear that $K\in\mathcal{M}(T)$
satisfies (i); see Subsection \ref{singularities}. Theorem
\ref{asymTHM}(ii) implies (ii) for $K$. What remains is the
verification of the $q$-difference equations (iii) for $K$. Using
the notations of Lemma \ref{qconnLem}, we write
\[B_j^{(0)}:=B_j|_{x^{-\alpha_1}=0,\ldots,x^{-\alpha_{N-1}}=0}\in
Q_0(\mathcal{B})\otimes\textup{End}(H_0).
\]
We view $B_j^{(0)}(\gamma)$ as an $\textup{End}(H_0)$-valued
meromorphic function in $\gamma\in T$. Taking the limit
$|t^{-\alpha_i}|\rightarrow 0$ ($1\leq i<N$) in the gauged
$q$-difference equations
\[B_j(t,\gamma)\Psi(t,q^{\varpi_j}\gamma)=\Psi(t,\gamma),\qquad 1\leq j\leq N
\]
and using $\Gamma_0(\gamma)=K(\gamma)T_{w_0}$ we obtain
\[K(q^{\varpi_j}\gamma)B_j^{(0)}(\gamma)T_{w_0}=
K(\gamma)T_{w_0}
\]
for $1\leq j\leq N$, as meromorphic $H_0$-valued functions in $\gamma\in T$.
Writing $B_j^{(0)}(\gamma)T_{w_0}=\sum_{w\in S_N}a_w^j(\gamma)T_w$ with
$a_w^j\in\mathcal{M}(T)$ it thus suffices to show that
\begin{equation}\label{coef}
a_{w_0}^j(\gamma)=\prod_{\stackrel{1\leq r\leq
j}{\stackrel{j+1\leq s\leq N}{}}} \frac{1-q\gamma_r/\gamma_s}
{1-qk^2\gamma_r/\gamma_s}=k^{-\langle\delta,\varpi_j\rangle}
\prod_{\stackrel{1\leq r\leq j}{\stackrel{j+1\leq s\leq N}{}}}
c_k(q\gamma_r/\gamma_s)^{-1}
\end{equation}
for $1\leq j\leq N$, where the second equality follows from a direct
computation using the explicit expression \eqref{ck} of $c_k$.

By \eqref{fff} we have
\[
B_j(t,\gamma)=k^{-\langle\delta,\varpi_j\rangle}t^{w_0(\varpi_j)}
C_{(e,\varpi_j)}(t,\gamma)
\]
and $C_{(e,\varpi_j)}(t,\gamma)$ is given explicitly by
\eqref{fundamental}. Since $R_i(z)=c_k(z)^{-1}(\eta(T_i)-k)+1$,
Lemma \ref{etaexplicit} and the reduced expression \eqref{redex}
for $\sigma^i$ imply that
\[B_j(t,\gamma)T_{w_0}=k^{-\langle\delta,\varpi_j\rangle}t^{w_0(\varpi_j)}
C_\iota(\eta(\pi)(t^{-1}))^j\bigl(\sum_{w\leq \sigma^{-j}}b_w^j(\gamma)
T_{ww_0}\bigr)
\]
with $\leq $ the Bruhat order on $S_N$ and with
\[
b_{\sigma^{-j}}^j(\gamma)=\prod_{\stackrel{1\leq r\leq
j}{\stackrel{j+1\leq s\leq N}{}}} c_k(q\gamma_r/\gamma_s)^{-1}.
\]
By Lemma \ref{etaexplicit} we have
\[t^{w_0(\varpi_j)}C_\iota(\eta(\pi)(t^{-1}))^jT_{ww_0}=
t^{w_0(\varpi_j)-w_0w^{-1}w_0(\varpi_j)}T_{w_0w^{-1}\sigma^{-j}}.
\]
Hence
\[B_j^{(0)}(\gamma)T_{w_0}=k^{-\langle\delta,\varpi_j\rangle}
\sum_{w}b_w^j(\gamma)T_{w_0w^{-1}\sigma^{-j}},
\]
with the sum running over $w\in S_N$ satisfying $w\leq \sigma^{-j}$ and
$w(\varpi_j)=w_0(\varpi_j)$. In particular,
$a_{w_0}^j(\gamma)=
k^{-\langle\delta,\varpi_j\rangle}b_{\sigma^{-j}}^j(\gamma)$. This completes the
proof of \eqref{coef}.
\end{proof}

\subsection{Consistency of the bispectral quantum KZ equations}
In this subsection, we show that BqKZ is a consistent system of
$q$-difference equations, i.e.,
$\textup{dim}_{\mathbb{F}}(\textup{SOL})=\textup{dim}_{\C}(H_0)$,
by explicitly constructing an $\mathbb{F}$-basis of
$\textup{SOL}$. Since the $q$-connection matrices
$C_{(\lambda,\mu)}(t,\gamma)$ ($\lambda,\mu\in\Z^N$) depend
rationally on $(t,\gamma)\in T\times T$, the consistency of BqKZ
follows also from the abstract arguments in \cite[\S 5]{Et}.

We start with a preliminary lemma on the cocycle values
$C_{(e,w)}$ for $w\in S_N$.
\begin{lem}
Let $w\in S_N$. We have $C_{(e,w)}\in Q_0(\mathcal{B})\otimes\textup{End}(H_0)$
and
\[C_{(e,w)}^{(0)}(h)=k^{-\ell(w)}hT_{w^{-1}},\qquad h\in H_0,
\]
where
\[C_{(e,w)}^{(0)}=C_{(e,w)}|_{y^{\alpha_1}=0,\ldots,y^{\alpha_{N-1}}=0}\in
\textup{End}(H_0).
\]
\end{lem}
\begin{proof}
Let $w=s_{i_1}s_{i_2}\cdots s_{i_r}$ be a reduced expression for
$w\in S_N$ ($1\leq i_j<N$) and write $\beta_j:=s_{i_1}\cdots
s_{i_{j-1}}(\alpha_{i_j})\in R_+$ for $1\leq j\leq r$, where
$\beta_1$ should be read as $\alpha_{i_1}$. By Subsection
\ref{cocyclevalues} and the cocycle property, we have
\[
C_{(e,w)}(t,\gamma)=C_\iota C_{(w,e)}(\gamma^{-1},t^{-1})C_\iota
=C_\iota\bigl(R_{i_r}(\gamma^{\beta_r})\cdots R_{i_2}(\gamma^{\beta_2})
R_{i_1}(\gamma^{\beta_1})\bigr)^{-1}C_\iota.
\]
{}From the expression for $R_i(z)$ it now follows that
$C_{(e,w)}\in Q_0(\mathcal{B})\otimes\textup{End}(H_0)$.
Since $\lim_{z\rightarrow 0}R_i(z)=k\eta(T_i^{-1})$ we furthermore
have
\[C_{(e,w)}^{(0)}=k^{-\ell(w)}C_\iota\eta(T_w)C_\iota.
\]
The map $C_\iota$ is the $\mathbb{K}$-linear extension of the
anti-algebra involution of $H_0$ mapping $T_w$ to $T_{w^{-1}}$.
Hence $C_{(e,w)}^{(0)}(h)=k^{-\ell(w)}hT_{w^{-1}}$ for $h\in H_0$.
\end{proof}
Define $U\in\textup{End}(H_0)^{\mathbb{K}}:=\mathbb{K}\otimes
\textup{End}(H_0)$ by
\begin{equation}
U\bigl(k^{-\ell(w)}T_{w_0}T_{w^{-1}}\bigr)=\tau(e,w)\Phi_\kappa,\qquad
w\in S_N.
\end{equation}
Since $\textup{SOL}$ is $\mathbb{S}_N$-invariant, $U$ is an
$\textup{End}(H_0)$-valued solution of the BqKZ, i.e.
\[C_{(\lambda,\mu)}(t,\gamma)U(q^{-\lambda}t,q^{\mu}\gamma)=
U(t,\gamma),\qquad \lambda,\mu\in\Z^N
\]
as $\textup{End}(H_0)$-valued meromorphic functions in
$(t,\gamma)\in T\times T$.
\begin{lem}\label{Consequence2}
$U\in\textup{End}(H_0)^{\mathbb{K}}$ is invertible.
\end{lem}
\begin{proof}
Using the natural identification $\textup{End}(H_0)^{\mathbb{K}}\simeq
\textup{End}_{\mathbb{K}}(H_0^{\mathbb{K}})$ as $\mathbb{K}$-algebras, we need
to verify that $U\in\textup{GL}_{\mathbb{K}}(H_0^{\mathbb{K}})$.

Set $\Phi_w:=\tau(e,w)\Phi_\kappa$ and $\Psi_w:=\tau(e,w)\Psi$ for
$w\in S_N$, so that
$\Phi_w(t,\gamma)=W_\kappa(t,w^{-1}\gamma)\Psi_w(t,\gamma)$. Since
$C_{(e,w)}(t,\gamma)$ is independent of $t\in T$, we simply write
it as $C_{(e,w)}(\gamma)$. Recall the $W$-invariant subset
$\mathcal{S}\subset T$ (see \eqref{Scal}), which contains
$\mathcal{S}_+$. By Lemma \ref{singularC} and Lemma
\ref{Consequence1}, we have for some $\epsilon>0$ the power series
expansion
\[\Psi_w(t,\gamma)=\sum_{\alpha\in Q_+}
C_{(e,w)}(\gamma)\bigl(\Gamma_\alpha(w^{-1}\gamma)\bigr)
t^{-\alpha}
\]
for $(t,\gamma)\in B_\epsilon^{-1}\times T\setminus\mathcal{S}$,
converging normally on compacta of $B_\epsilon^{-1}\times
T\setminus\mathcal{S}$. We write
$\Gamma_\alpha^w(\gamma):=C_{(e,w)}(\gamma)
\bigl(\Gamma_\alpha(w^{-1}\gamma)\bigr)$ in the remainder of the
proof. It is a meromorphic function in $\gamma\in T$, analytic on
$T\setminus\mathcal{S}$, and the power series expansion of
$\Psi_w$ becomes
\begin{equation}\label{PsiWexpansion}
\Psi_w(t,\gamma)=\sum_{\alpha\in
Q_+}\Gamma_\alpha^w(\gamma)t^{-\alpha}.
\end{equation}
Observe that
\[
\Gamma_0^w(\gamma)\rightarrow C_{(e,w)}^{(0)}(T_{w_0})=
k^{-\ell(w)}T_{w_0}T_{w^{-1}},
\]
in the limit $\gamma^{\alpha_i}\rightarrow 0$ ($1\leq i<N$),
in view of the previous lemma.

Write $U=V\Xi$ with $V,\Xi$ the $\mathbb{K}$-linear endomorphisms of
$H_0^{\mathbb{K}}$ given by
\begin{equation*}
\begin{split}
\Xi\bigl(k^{-\ell(w)}T_{w_0}T_{w^{-1}}\bigr)(t,\gamma)&=
W_\kappa(t,w^{-1}\gamma)k^{-\ell(w)}T_{w_0}T_{w^{-1}},\\
V\bigl(k^{-\ell(w)}T_{w_0}T_{w^{-1}}\bigr)&=\Psi_w,
\end{split}
\end{equation*}
for $w\in S_N$. Since
$\Xi\in\textup{GL}_{\mathbb{K}}(H_0^{\mathbb{K}})$ it suffices to
show that $V\in\textup{GL}_{\mathbb{K}}(H_0^{\mathbb{K}})$. Let
$M$ be the matrix of $V$ with respect to the $\mathbb{K}$-basis
$k^{-\ell(w)}T_{w_0}T_{w^{-1}}$ ($w\in S_N$) of
$H_0^{\mathbb{K}}$. Now fix $\zeta\in T\setminus\mathcal{S}$ such that 
$\zeta^\alpha\notin q^{\Z}$ for all $\alpha\in R$. The
matrix $M(t,\gamma)$ may be specialized at $\gamma=\zeta$ and the
limit of $M(t,\zeta)$ as $t^{-\alpha_i}\rightarrow 0$ ($1\leq
i<N$) exists. 
We write $M^{(0)}(\zeta)$ for the limit and $V^{(0)}(\zeta)$ 
for the corresponding linear endomorphism of
$H_0$. We then have
\[V^{(0)}(\zeta)\bigl(k^{-\ell(w)}T_{w_0}T_{w^{-1}}\bigr)=
\Gamma_0^w(\zeta)=K(w^{-1}\zeta)C_{(e,w)}(\zeta)T_{w_0},
\]
with $K(\gamma)$ given by \eqref{Kexp}. Note that $K(w^{-1}\zeta)\neq0$ since 
$\zeta^\alpha\notin q^{\Z}$ for all $\alpha\in R$. 
By the explicit expression
for the cocycle value $C_{(e,w)}(\zeta) \in\textup{End}(H_0)$ (see
the proof of the previous lemma) we have
\[C_{(e,w)}(\zeta)(T_{w_0})=\sum_{v\leq
w}a_v^w(\zeta)T_{w_0}T_{v^{-1}},
\]
with $a_w^w(\zeta)\not=0$ and with $\leq$ the Bruhat order on
$S_N$. This implies that $V^{(0)}(\zeta)$ is a linear automorphism
of $H_0$, hence $\det(M^{(0)}(\zeta))\not=0$. Consequently,
$\det(M)\in\mathbb{K}^\times$ and
$V\in\textup{GL}_{\mathbb{K}}(H_0^{\mathbb{K}})$.
\end{proof}
\begin{prop}\label{basisprop}
{\bf (i)} $U^\prime\in\textup{End}(H_0)^{\mathbb{K}}$ is an
$\textup{End}(H_0)$-valued meromorphic solution of BqKZ if and
only if
$U^\prime=UF$ for some $F\in\textup{End}(H_0)^{\mathbb{F}}$.\\
{\bf (ii)} $U$, viewed
as $\mathbb{K}$-linear endomorphism of $H_0^{\mathbb{K}}$,
restricts to an $\mathbb{F}$-linear isomorphism
$U: H_0^{\mathbb{F}}\rightarrow \textup{SOL}$.\\
{\bf (iii)} $\{\tau(e,w)\Phi_\kappa\}_{w\in S_N}$ is an
$\mathbb{F}$-basis of $\textup{SOL}$.
\end{prop}
\begin{proof}
(i) If $U^\prime$ is an $\textup{End}(H_0)$-valued meromorphic
solution of BqKZ then, since $U$ is invertible, we have for all
$\lambda,\mu\in\Z^N$,
\[U(q^{-\lambda}t,q^\mu\gamma)^{-1}U^\prime(q^{-\lambda}t,q^\mu\gamma)=
U(t,\gamma)^{-1}U^\prime(t,\gamma).
\]
Hence $U^\prime=UF$ with $F\in\textup{End}(H_0)^{\mathbb{F}}$.
The converse implication is clear.\\
(ii) By the previous lemma we have
$U: H_0^{\mathbb{F}}\hookrightarrow \textup{SOL}$.
It is surjective, since for $g\in\textup{SOL}$, $f:=U^{-1}g\in H_0^{\mathbb{K}}$
satisfies $f(q^{-\lambda}t,q^\mu\gamma)=f(t,\gamma)$ for all $\lambda,\mu\in\Z^N$
(cf. the proof of (i)), hence $f\in H_0^{\mathbb{F}}$.\\
(iii) This is clear from (ii) and from the definition of $U$.
\end{proof}
By Proposition \ref{spgauge} and by the proofs of Lemma
\ref{Consequence2} and Proposition \ref{basisprop} we obtain the
following consistency statement for the quantum KZ equation
\eqref{qKZ} with specialized central character (see \cite{CInd}
and \cite{C}). Recall the $W$-invariant subset $\mathcal{S}\subset
T$ given by \eqref{Scal}. Recall furthermore that
$C_{(e,w)}(t,\gamma)$ for $w\in S_N$ only depends on $\gamma$, so
we simply write it as $C_{(e,w)}(\gamma)$.
\begin{cor}\label{consistencyzeta}
Fix $\zeta\in T\setminus\mathcal{S}$ such that 
$\zeta^\alpha\notin q^{\Z}$ for all $\alpha\in R$. For generic $\kappa\in\C^\times$, the
$H_0$-valued meromorphic functions $\bigl(\tau(e,w)\Phi_\kappa\bigr)(t,\gamma)$
in $(t,\gamma)\in T\times T$ ($w\in S_N$) can be specialized at $\gamma=\zeta$,
giving rise to
\begin{enumerate}
\item[{\bf (i)}] a basis
$\{C_{(e,w)}(\zeta)\Phi_\kappa(\cdot,w^{-1}\zeta)\}_{w\in S_N}$
of $\textup{SOL}_\zeta$ over $\mathcal{E}(T)$;
\item[{\bf (ii)}] an invertible $\textup{End}(H_0)$-valued
meromorphic solution $U_\zeta$
of the quantum KZ equations \eqref{qKZ}, where
$U_\zeta\in \textup{End}(H_0)^{\mathcal{M}(T)}$
is explicitly defined by
\[
U_\zeta(k^{-\ell(w)}T_{w_0}T_{w^{-1}}):=
C_{(e,w)}(\zeta)\Phi_\kappa(\cdot,w^{-1}\zeta),\qquad w\in S_N.
\]
\end{enumerate}
 \end{cor}

\section{The correspondence with bispectral problems}
\label{SectionCM}

For a finite dimensional affine
Hecke algebra module $M$, Cherednik \cite[Thm. 4.2]{CInd}
derived, for arbitrary root systems,
a correspondence between solutions of the quantum affine
KZ equations associated to $M$ and solutions
of a system of (possibly matrix-valued)
$q$-difference equations.
This correspondence was considered before
in \cite[Thm. 3.4]{Cref} for a special class of modules $M$,
in which case the corresponding $q$-difference operators
contain the Macdonald
$q$-difference operator
(the latter observation is due to Cherednik \cite[Thm. 4.4]{Cref} for
$\textup{GL}_N$,
and due to Kato \cite[Thm. 4.6]{Ka} and Cherednik
\cite{CInd} for general root systems).
In the classical setting ($q=1$)
it goes back to Matsuo \cite{Mat}.

In the present $\textup{GL}_N$ setting, with $M=M(\zeta)$
the minimal principal series module with specialized central
character $\zeta\in T$, these constructions give rise to
an explicit map $\chi_+$ (independent
of $\zeta$) from the solution space $\textup{SOL}_{\zeta}$ of the
quantum KZ equation \eqref{qKZ}
to the solution space of the spectral problem
for the commuting Ruijsenaars' \cite{R} trigonometric $q$-difference operators
with spectral parameter $\zeta^{-1}$, see \cite[Thm. 4.4]{Cref}
(for $\textup{GL}_N$ the Macdonald $q$-difference operators coincide with
the Ruijsenaars' operators).  We recall this result in detail
in Subsection \ref{spec}.

In this section we investigate the map $\chi_+$ when applied to
solutions of the bispectral extension of the quantum KZ equations;
see Subsection \ref{BqKZsection}. In order to do so, we need to
replace in the correspondence as described in the previous
paragraph, the role of $M(\zeta)$ by the infinite dimensional,
formal principal series module of $H$ as defined in Subsection
\ref{formal}. We then show that the same map $\chi_+$ gives rise
to an embedding of the solution space $\textup{SOL}$ of BqKZ into
the solution space of a bispectral problem for the Ruijsenaars
operators. The techniques employed in this section are analogous
to the ones for the usual correspondence (see \cite[Chpt. 1]{C}).
For the convenience of the reader we provide full details of the
arguments involved.

We fix $\kappa\in\C^\times$ throughout Subsections \ref{mono}-\ref{BiHC}.
\subsection{The monodromy cocycle}\label{mono}

Observe that $F\in\textup{End}(H_0)^{\mathbb{K}}$ is
a $\textup{End}(H_0)$-valued meromorphic solution of
BqKZ if and only if
\[\tau(\mathrm{w})F=F,\qquad \mathrm{w}\in\Z^N\times\Z^N,\]
where the $\mathbb{W}$-action $\tau$
on $\textup{End}(H_0)^{\mathbb{K}}$
is defined by
\[
(\tau(\mathrm{w})F)(t,\gamma):=
C_{\mathrm{w}}(t,\gamma)(\mathrm{w}F)(t,\gamma)=
C_{\mathrm{w}}(t,\gamma)F(\mathrm{w}^{-1}(t,\gamma))
\]
for $\mathrm{w}\in\mathbb{W}$ and $F\in\textup{End}(H_0)^{\mathbb{K}}$,
viewed as identities between
$\textup{End}(H_0)$-valued meromorphic functions in
$(t,\gamma)\in T\times T$.

By Proposition \ref{basisprop}(i), given an
$\textup{End}(H_0)$-valued meromorphic solution $F$ of
BqKZ, there exists a unique $G\in\textup{End}(H_0)^{\mathbb{F}}$
such that $F=UG$. Accordingly, $G$ describes the deviation of
$F$ from the fundamental solution $U$ of BqKZ,
and therefore can be thought of as a connection matrix.
We will consider the special cases when $F$ are the $\textup{End}(H_0)$-valued
meromorphic solutions
$\tau(\mathrm{w})U$ ($\mathrm{w}\in\mathbb{W}$) of BqKZ.

For $\mathrm{w}\in\mathbb{W}$ we set
\[
\mathcal{T}_{\mathrm{w}}:=U^{-1}(\tau(\mathrm{w})U)\in
\textup{End}(H_0)^{\mathbb{F}},
\]
that is, $\mathcal{T}_{\mathrm{w}}$ ($\mathrm{w}\in\mathbb{W}$)
is the unique element of $\textup{End}(H_0)^{\mathbb{F}}$ such that
\[
\tau(\mathrm{w})U=U\mathcal{T}_\mathrm{w}.
\]
Note that $\textup{End}(H_0)^{\mathbb{F}}$ is a $\mathbb{W}$-stable
subalgebra of $\textup{End}(H_0)^{\mathbb{K}}$ with respect to the
action $(\mathrm{w}F)(t,\gamma)=F(\mathrm{w}^{-1}(t,\gamma))$. The
following lemma now shows that the $\mathcal{T}_{\mathrm{w}}$
($\mathrm{w}\in\mathbb{W}$) define a cocycle of $\mathbb{W}$ with
values in the group of units of $\textup{End}(H_0)^{\mathbb{F}}$.
\begin{lem} {\bf(i)} $\mathcal{T}_{\mathrm{w}}=\textup{id}$
for $\mathrm{w}\in\Z^N\times\Z^N$.\\
{\bf(ii)} For $\mathrm{w},\mathrm{w}^\prime\in\mathbb{W}$ we have the
cocycle relation
\[
\mathcal{T}_{\mathrm{w}\mathrm{w}^\prime}=\mathcal{T}_{\mathrm{w}}
\mathrm{w}(\mathcal{T}_{\mathrm{w}^\prime})
\]
in $\textup{End}(H_0)^{\mathbb{F}}$.
\end{lem}
\begin{proof}
(i) This follows immediately from the fact that $U$ is an
$\textup{End}(H_0)$-valued meromorphic solution of BqKZ.\\
(ii) Note that
$\mathcal{T}_\mathrm{w}=U^{-1}C_{\mathrm{w}}\mathrm{w}(U)$ for
$\mathrm{w}\in\mathbb{W}$. By the cocycle condition for $C_{\mathrm{w}}\in
\textup{End}(H_0)^{\mathbb{K}}$, which reads in the present notations
as $C_{\mathrm{w}\mathrm{w}^\prime}=
C_{\mathrm{w}}\mathrm{w}(C_{\mathrm{w}^\prime})$
for $\mathrm{w},\mathrm{w}^\prime\in\mathbb{W}$, we have
\[\begin{split}
\mathcal{T}_{\mathrm{w}\mathrm{w}^\prime}
&=U^{-1}C_{\mathrm{w}\mathrm{w}^\prime}\mathrm{w}\mathrm{w}^\prime(U)
=U^{-1}C_{\mathrm{w}}\mathrm{w}(C_{\mathrm{w}^\prime}\mathrm{w'}(U))\\
&=U^{-1}C_{\mathrm{w}}\mathrm{w}(U)\mathrm{w}(U^{-1}C_{\mathrm{w}^\prime}
\mathrm{w}^\prime(U))
=\mathcal{T}_{\mathrm{w}}\mathrm{w}(\mathcal{T}_{\mathrm{w}^\prime})
\end{split}\]
for all $\mathrm{w},\mathrm{w}^\prime\in\mathbb{W}$.
\end{proof}
\begin{defi}
In analogy with the terminology in \cite[\S 1.3.3]{C} for the
quantum KZ equation, we call
$\{\mathcal{T}_{\mathrm{w}}\}_{\mathrm{w}\in\mathbb{W}}$ the monodromy cocycle
of the BqKZ.
\end{defi}
\begin{rema}
Connection matrices and Riemann-Hilbert problems for ordinary
linear $q$-difference equations have been extensively studied;
see, e.g., \cite{Bi} and \cite{Sa}. For quantum KZ equations,
connection matrices have been computed explicitly in, e.g.,
\cite{FR}, \cite[\S 12]{EFK}, and \cite{KK}.
\end{rema}

\subsection{The correspondence}\label{Corr}

Consider the algebra $\C(T\times T)\#\mathbb{W}$, where
$\mathbb{W}$ acts as field automorphisms on $\C(T\times T)$ by
the formula \eqref{doubleaction}. Recall that $\C(T\times T)\#\mathbb{W}$
naturally acts on $\mathbb{K}$. We write $Df$ for the action of
$D\in\C(T\times T)\#\mathbb{W}$ on $f\in\mathbb{K}$.

We have a representation
$\vartheta\colon\C(T\times T)\#\mathbb{W}\to
\textup{End}(\textup{End}(H_0)^{\mathbb{K}})$ given by
\[\begin{split}
\vartheta(f)F&=fF,\qquad f\in\C(T\times T),\\
\vartheta(\mathrm{w})F&=\mathrm{w}(F),\qquad \mathrm{w}\in\mathbb{W}
\end{split}
\]
for $F\in\textup{End}(H_0)^{\mathbb{K}}$.
Let $\mathbb{D}$ be the subalgebra $\C(T\times T)\#(\Z^N\times\Z^N)$
of $\C(T\times T)\#\mathbb{W}$. Under the natural action of
$\C(T\times T) \#\mathbb{W}$ on $\C(T\times T)$, the subalgebra $\mathbb{D}$
identifies with the algebra of $q$-difference
operators on $T\times T$ with rational coefficients.

Set $H_0^*:=\textup{Hom}(H_0,\C)$. We will regard
a linear functional $\chi\in H_0^*$ also as an element of
$\textup{Hom}_{\mathbb{K}}(H_0^{\mathbb{K}},\mathbb{K})$ by
$\mathbb{K}$-linear extension.
For $F\in\textup{End}(H_0)^{\mathbb{K}}$, denote
\[
\phi^F_{\chi,v}:=\chi(Fv)\in\mathbb{K},\qquad \chi\in H_0^*,v\in
H_0
\]
for its matrix coefficients.
Note that for any $D\in\C(T\times T)\#\mathbb{W}$,
$\chi\in H_0^*$ and $v\in H_0$,
\begin{equation}\label{eqD}
D\phi_{\chi,v}^F=\phi_{\chi,v}^{\vartheta(D)F}
\end{equation}
for all $F\in\textup{End}(H_0)^{\mathbb{K}}$.

\begin{lem}\label{lemTheta}
For $\mathrm{w}\in\mathbb{W}$ we have
\[
\vartheta(\mathrm{w})U=C_\mathrm{w}^{-1}U\mathcal{T}_\mathrm{w}.
\]
In particular, $\vartheta(\mathrm{w})U=C_\mathrm{w}^{-1}U$ for
$\mathrm{w}\in\Z^N\times\Z^N$.
\end{lem}
\begin{proof}
For $\mathrm{w}\in\mathbb{W}$
\[
\vartheta(\mathrm{w})U=\mathrm{w}(U)=C_\mathrm{w}^{-1}(\tau(\mathrm{w})U)
=C_\mathrm{w}^{-1}U\mathcal{T}_\mathrm{w}.
\]
The second claim follows from the fact that
$\mathcal{T}_\mathrm{w}=\mathrm{id}$ for
$\mathrm{w}\in\Z^N\times\Z^N$.
\end{proof}

We now are going to look for a particular linear functional
$\chi$ such that the matrix coefficients $\phi_{\chi,v}^U$
($v\in H_0$) of $U$ solves a bispectral problem with respect to
two commuting families of Ruijsenaars' trigonometric $q$-difference operators
(one family acting on the first torus component, the second
on the second torus component).
In view of \eqref{eqD} and the previous lemma,
to obtain $q$-difference equations for $\phi_{\chi,v}$
we have to deal with the cocycle value
$C_\mathrm{w}$ and the monodromy matrix
$\mathcal{T}_\mathrm{w}$ in the equations
$\mathrm{w}\phi_{\chi,v}^U=
\phi_{\chi,v}^{C^{-1}_\mathrm{w}U\mathcal{T}_\mathrm{w}}$.
It is convenient to postpone the analysis of the monodromy cocycle
by initially absorbing it into the action
$\vartheta$ of $\C(T\times T)\#\mathbb{W}$
via the twisted algebra homomorphism
\[
\vartheta_\mathcal{T}\colon\C(T\times T)\#\mathbb{W}\to
\textup{End}(\textup{End}(H_0)^{\mathbb{K}}),
\]
defined by
\[\begin{split}
\vartheta_\mathcal{T}(f)F&=fF,\qquad\qquad f\in\C(T\times T),\\
\vartheta_\mathcal{T}(\mathrm{w})F&=
\mathrm{w}(F)\mathcal{T}_\mathrm{w}^{-1},\qquad
\mathrm{w}\in\mathbb{W}
\end{split}
\]
for $F\in\textup{End}(H_0)^{\mathbb{K}}$.
Note that $\vartheta_\mathcal{T}$ is indeed an algebra
homomorphism, thanks to the cocycle condition for $\mathcal{T}$. Moreover,
$\vartheta_\mathcal{T}|_{\mathbb{D}}=\vartheta|_{\mathbb{D}}$.

For $D\in\C(T\times T)\#\mathbb{W}$ we will occasionally use the
notations
\begin{equation}\label{convnotation}
D=\sum_{\mathrm{w}\in\mathbb{W}}d_\mathrm{w}\mathrm{w}=
\sum_{\mathrm{v}\in\mathbb{S}_N}D_\mathrm{v}\mathrm{v},
\end{equation}
where $d_\mathrm{w}\in\C(T\times T)$ ($\mathrm{w}\in\mathbb{W}$)
and
$D_\mathrm{v}=\sum_{\mathrm{u}\in\Z^N\times\Z^N}d_{\mathrm{uv}}\mathrm{u}\in
\mathbb{D}$ ($\mathrm{v}\in\mathbb{S}_N$). Reformulating
\eqref{eqD} and Lemma \ref{lemTheta} in terms of the twisted
action $\vartheta_{\mathcal{T}}$ yields the following result.

\begin{lem}\label{lemThetaT}
{\bf(i)} For $\mathrm{w}\in\mathbb{W}$ we have
\[
\vartheta_\mathcal{T}(\mathrm{w})U=C_\mathrm{w}^{-1}U.
\]
{\bf(ii)} For $D\in\C(T\times T)\#\mathbb{W}$ we have
\[
\phi_{\chi,v}^{\vartheta_\mathcal{T}(D)U}=
\sum_{\mathrm{v}\in\mathbb{S}_N}D_\mathrm{v}
(\phi_{\chi,v}^{C_\mathrm{v}^{-1}U})
\]
for all $\chi\in H_0^*$ and $v\in H_0$.
\end{lem}
\begin{proof}
(i) This is clear from Lemma \ref{lemTheta} and the definition of
$\vartheta_\mathcal{T}$.\\
(ii) By Lemma \ref{lemTheta} and \eqref{eqD}, we obtain
\[
\phi_{\chi,v}^{\vartheta_\mathcal{T}(D)U}=
\sum_{\mathrm{v}\in\mathbb{S}_N}\phi_{\chi,v}^{\vartheta(D_\mathrm{v})
\vartheta_\mathcal{T}(\mathrm{v})U}=
\sum_{\mathrm{v}\in\mathbb{S}_N}D_\mathrm{v}
(\phi_{\chi,v}^{\vartheta_\mathcal{T}(\mathrm{v})U}).
\]
The result now follows from (i).
\end{proof}
We define the restriction map
$\textup{Res}\colon\C(T\times T)\#\mathbb{W}\to\mathbb{D}$
to be the $\C(T\times T)$-linear map
\[
\textup{Res}(D)
:=\sum_{\mathrm{v}\in\mathbb{S}_N}D_{\mathrm{v}},\qquad D\in\mathbb{D}.
\]
Lemma \ref{lemTheta}(ii) implies that if we have a
linear functional $\chi_+\in H_0^*$ such that
$\chi_+(C^{-1}_\mathrm{v}U)=\chi_+(U)$ for all
$\mathrm{v}\in\mathbb{S}_N$, then the corresponding matrix coefficients
$\phi_{\chi_+,v}^U$ ($v\in H_0$) satisfy
\begin{equation}\label{eqResD}
\textup{Res}(D)(\phi_{\chi_+,v}^U)=
\phi_{\chi_+,v}^{\vartheta_{\mathcal{T}}(D)U}
\end{equation}
for all $D\in\C(T\times T)\#\mathbb{W}$.
\begin{lem}\label{chiok}
Define $\chi_+\in H_0^*$ by $\chi_+(T_w)=k^{\ell(w)}$ for all
$w\in S_N$. Then
\[\chi_+(C_{\mathrm{v}}^{-1}F)=\chi_+(F)
\]
for $F\in\textup{End}(H_0)^{\mathbb{K}}$ and
$\mathrm{v}\in\mathbb{S}_N$.
\end{lem}
\begin{proof}
Since $C_\iota(T_w)=T_{w^{-1}}$ for $w\in S_N$ we have
$\chi_+\circ C_\iota= \chi_+$. By the cocycle condition for
$C_{\mathrm{w}}$ ($\mathrm{w}\in\mathbb{S}_N$) it remains to prove
that $\chi_+\circ C_{(s_i,e)}=\chi_+$ for $1\leq i<N$. But this
follows from the expression
$C_{(s_i,e)}(t,\gamma)=c_k(t_i/t_{i+1})^{-1}(\eta(T_i)-k)+1$ (see
Lemma \ref{YBlem}) since
\begin{equation}\label{chiKills}
\chi_+((T_i-k)h)=0
\end{equation}
for $1\leq i< N$ and $h\in H_0$.
\end{proof}
If $D\in\C(T\times T)\#\mathbb{W}$ satisfies
$\vartheta_\mathcal{T}(D)U=\lambda U$ for some $\lambda\in\mathbb{K}$,
then it follows from \eqref{eqResD} that the matrix coefficients
$\phi_{\chi_+,v}^U$ ($v\in H_0$) are eigenfunctions of $\textup{Res}(D)$
with eigenvalue $\lambda$. We will now construct
such a commuting family of $D$'s. It leads to the interpretation of
the $\phi_{\chi_+,v}$ ($v\in H_0$) as solutions of a bispectral problem.

The appropriate elements $D\in\C(T\times T)\#\mathbb{W}$ are
obtained as images of elements from the center $Z(H)$ of the
affine Hecke algebra $H$ under the faithful algebra homomorphism
$\rho$ from Theorem \ref{thmRho}. Since we aim at a bispectral
version, we will interpret $\rho$ as algebra map $\rho:
H\rightarrow \C(T\times T)\#\mathbb{W}$ in two different ways. We
have, on the one hand, the algebra homomorphism
\[\rho_{k^{-1},q}^x: H(k^{-1})\rightarrow \C(T\times T)\#\mathbb{W},
\]
which is the map $\rho_{k^{-1},q}$ from Theorem \ref{thmRho},
interpreted as algebra homomorphism from $H(k^{-1})$ to the
subalgebra $\C(T\times \{1\})\#(W\times \{e\})$ of $\C(T\times
T)\#\mathbb{W}$. On the other hand, we have an algebra
homomorphism
\[\rho_{k,q^{-1}}^y: H(k)\rightarrow \C(T\times T)\#\mathbb{W},
\]
defined as the map $\rho_{k,q^{-1}}$ from Theorem \ref{thmRho},
interpreted as algebra homomorphism from $H(k)$ to the subalgebra
$\C(\{1\}\times T)\#(\{e\}\times W)$ of $\C(T\times T)\#\mathbb{W}$.
Note that they can be combined into an algebra homomorphism
\[
\rho_{k^{-1},q}^x\times\rho_{k,q^{-1}}^y: H(k^{-1})\otimes H(k)\rightarrow
\C(T\times T)\#\mathbb{W}.
\]

\begin{defi}
{\bf(i)}
For $h\in H(k^{-1})$, define
\[
D^x_h:=\rho_{k^{-1},q}^x(h)\in\C(T\times T)\#\mathbb{W}.
\]
{\bf(ii)}
For $h\in H(k)$, define
\[
D^y_h:=\rho_{k,q^{-1}}^y(h)\in\C(T\times T)\#\mathbb{W}.
\]
\end{defi}

\begin{rema}\label{circle}
Let ${}^\circ\colon H(k^{-1})\rightarrow H(k)$ be the algebra
isomorphism defined by $\pi^\circ=\pi$ and $T_i^\circ=T_i^{-1}$
for $1\leq i<N$. Then
\begin{equation}\label{circleformula}
D_{h^\circ}^y=\iota D_h^x\iota,\qquad \forall h\in H(k^{-1}).
\end{equation}
This follows by verifying the identity
\[\rho_{k,q^{-1}}^y(h^\circ)=\iota\rho_{k^{-1},q}^x(h)\iota
\]
for the algebraic generators
$\pi$ and $T_i$ ($1\leq i<N$) of $H(k^{-1})$ using Theorem \ref{thmRho}.
\end{rema}

Recall the formal principal series representation, encoded by the
algebra homomorphism $\eta\colon H(k)\rightarrow
\textup{End}(H_0)^{\mathbb{K}}$ (see Subsection \ref{formal}).
\begin{prop}\label{propThetaTdagger}
{\bf (i)} For $h\in H(k^{-1})$ we have
\begin{equation}\label{een}
\vartheta_{\mathcal{T}}(D_h^x)U=\eta(h^\dagger)U,
\end{equation}
where $\dagger\colon H(k^{-1})\to H(k)$ is the unique anti-algebra
isomorphism satisfying
\[T_i^\dagger=T^{-1}_i,\qquad\pi^\dagger=\pi^{-1}\]
for $1\leq i<N$.\\
{\bf (ii)} For $h\in H(k)$ we have
\begin{equation}\label{twee}
\vartheta_{\mathcal{T}}(D_h^{y})U=C_\iota\iota\bigl(\eta(h^{\ddagger})\bigr)
C_\iota U,
\end{equation}
where $\ddagger\colon H(k)\to H(k)$ is the unique anti-algebra
involution satisfying
\[T_i^\ddagger=T_i,\qquad\pi^\ddagger=\pi^{-1}\]
for $1\leq i<N$ (note that $\dagger=\ddagger\circ {}^{\circ}$).
\end{prop}
\begin{proof}
(i) We first show that it suffices to prove \eqref{een} for
algebraic generators of $H(k^{-1})$. Indeed, if \eqref{een} is valid
for $h,h^\prime\in H(k^{-1})$, then we have
\begin{equation*}
\begin{split}
\vartheta_{\mathcal{T}}(D_{hh^\prime}^x)U&=\vartheta_{\mathcal{T}}(D_h^t)
\vartheta_{\mathcal{T}}(D_{h^\prime}^x)U
=\vartheta_{\mathcal{T}}(D_h^x)\eta(h^{\prime\dagger})U\\
&=\eta(h^{\prime\dagger})\vartheta_{\mathcal{T}}(D_h^x)U
=\eta(h^{\prime\dagger})\eta(h^\dagger)U\\
&=\eta((hh^\prime)^\dagger)U,
\end{split}
\end{equation*}
where the third equality follows since
$\lbrack \vartheta_{\mathcal{T}}(D_h^x),\eta(h^{\prime\dagger})\rbrack=0$
as endomorphisms of $\textup{End}_{\mathbb{K}}(H_0^{\mathbb{K}})$
(here $\eta(h^{\prime\dagger})$ should be viewed as element in
$\textup{End}\bigl(\textup{End}(H_0)^{\mathbb{K}}\bigr)$ by
left multiplication).
Indeed, since $\eta(h^{\prime\dagger})(t,\gamma)$ does not depend on the torus
parameter $t\in T$, it commutes with
$\vartheta_{\mathcal{T}}(D_h^x)\in
\vartheta_{\mathcal{T}}(\C(T\times\{1\})\#(W\times \{e\}))$
(which involves, besides the action of $W\times\{e\}$, only
right multiplication by the monodromy cocycle).

So it remains to verify \eqref{een} for $h=\pi\in H(k^{-1})$ and for
$h=T_i\in H(k^{-1})$ ($1\leq i<N$). For $h=\pi\in H(k^{-1})$ we
have
\[\vartheta_{\mathcal{T}}(D_\pi^x)U=\vartheta_{\mathcal{T}}((\pi,e))U=
C_{(\pi,e)}^{-1}U=\eta(\pi^\dagger)U,
\]
where the last equality follows from \eqref{Cpi}. For $h=T_i\in
H(k^{-1})$ ($1\leq i<N$), we have
\begin{equation*}
\begin{split}
\vartheta_{\mathcal{T}}(D_{T_i}^x)U&=
(k^{-1}-c_k(x_{i+1}/x_i))U+c_k(x_{i+1}/x_i)\vartheta_{\mathcal{T}}((s_i,e))U\\
&=(k^{-1}-c_k(x_{i+1}/x_i)U+c_k(x_{i+1}/x_i)C_{(s_i,e)}^{-1}U\\
&=\eta(T_i^\dagger)U,
\end{split}
\end{equation*}
where we used that $c_{k^{-1}}(z^{-1})=c_k(z)$ in the first
equality, while the second equality follows from Lemma \ref{lemThetaT}(i)
and the third equality from Lemma \ref{YBlem}.\\
(ii) Unfortunately it is not possible to derive (ii) directly from
(i) and from \eqref{circleformula}. Instead, one has to repeat the
steps of the proof of (i). It again amounts to verifying \eqref{twee}
for $h=\pi\in H(k)$ and for $h=T_i\in H(k)$ ($1\leq i<N$).
We show the second case, the first case is left to the reader.

Let $1\leq i<N$. Then we have for $T_i\in H(k)$,
\begin{equation}\label{aaa}
\begin{split}
\vartheta_{\mathcal{T}}(D_{T_i}^y)U&=
(k-c_k(y_i/y_{i+1}))U+c_k(y_i/y_{i+1})\vartheta_{\mathcal{T}}((e,s_i))U\\
&=(k-c_k(y_i/y_{i+1}))U+c_k(y_i/y_{i+1})C_{(e,s_i)}^{-1}U.
\end{split}
\end{equation}
Since
\[
C_{(e,s_i)}=C_\iota \iota(C_{(s_i,e)})C_\iota
\]
by the cocycle condition (recall Remark \ref{Wtriv}) and since
$C_\iota^2=\textup{id}$ and
$C_{(s_i,e)}(t,\gamma)^{-1}=C_{(s_i,e)}(s_it,\gamma)$, Lemma
\ref{YBlem} implies that
\[C_{(e,s_i)}^{-1}=c_k(y_{i}/y_{i+1})^{-1}
(C_\iota\iota(\eta(T_i))C_\iota-k)+1.
\]
Substituting in \eqref{aaa} gives
$\vartheta_{\mathcal{T}}(D_{T_i}^y)U=
C_\iota\iota(\eta(T_i^{\ddagger}))C_\iota U$, as desired.
\end{proof}
The following lemma plays an important role in the bispectral
version of the correspondence. Recall that the center $Z(H)$ of
the affine Hecke algebra $H$ is given by $\C_Y[T]^{S_N}$
(Bernstein, see \cite{L}).
\begin{lem}\label{centerok}
For $p\in\mathbb{C}[T]^{S_N}$ we have
\[p(Y)^{\dagger}=p(Y^{-1}),\qquad p(Y)^{\ddagger}=p(Y^{-1}).
\]
\end{lem}
\begin{proof}
By \eqref{eqYdef} it immediately follows that
$Y_i^\dagger=Y_i^{-1}$ for $1\leq i\leq N$. This implies the first
formula.

For the second formula, it suffices to show that
\begin{equation}\label{ddaggerto}
Y_i^\ddagger=T_{w_0}Y_{N-i+1}^{-1}T_{w_0}^{-1}
\end{equation}
in $H(k)$ for $1\leq i\leq N$, since we then have, for
$p\in\C[T]^{S_N}$,
\[p(Y)^\ddagger=T_{w_0}p(Y_{N}^{-1},\ldots,Y_1^{-1})T_{w_0}^{-1}=
T_{w_0}p(Y^{-1})T_{w_0}^{-1}=p(Y^{-1}),
\]
where the last equality follows from the fact that $p(Y^{-1})\in Z(H(k))$.
To prove \eqref{ddaggerto}, note that $T_{w_0}T_i^{-1}=T_{N-i}^{-1}T_{w_0}$,
and $Y_{i+1}=T_i^{-1}Y_iT_i^{-1}$ by \eqref{eqYdef}, for
$1\leq i<N$. Hence \eqref{ddaggerto} holds for
$Y_{i+1}$ if it is true for $Y_{i}$. It thus remains
to prove \eqref{ddaggerto} for $i=1$. We will use the following
observation. Write $\sigma_i=s_is_{i+1}\cdots s_{N-1}$ ($1\leq i<N$) and
$\tau_i=s_{j}\cdots s_{N-2}$ ($1\leq j<N-1$), which are reduced
expressions in $S_N$. Then the longest Weyl group
element $w_0\in S_N$ can be written as
\begin{equation}\label{redw0}
\begin{split}
w_0&=\sigma_{N-1}\sigma_{N-2}\cdots\sigma_1\\
&=\sigma_1(\tau_{N-2}\tau_{N-3}\cdots \tau_1),
\end{split}
\end{equation}
and $\ell(w_0)$ is the sum of the lengths of the factors in the
respective products in \eqref{redw0}.

By \eqref{eqYdef}, formula \eqref{ddaggerto} for $i=1$
will be valid if
\begin{equation}\label{dd3}
T_{\sigma_1}\pi^{-1}=T_{w_0}\pi^{-1}T_{\sigma_1}T_{w_0}^{-1}
\end{equation}
in $H(k)$. By the first expression in \eqref{redw0} and the fact that
$\pi^{-1}T_{i+1}\pi=T_i$ for $1\leq i<N-1$, we have in $H(k)$,
\[\pi^{-1}T_{\sigma_1}T_{w_0}^{-1}=\pi^{-1}T_{\sigma_2}^{-1}T_{\sigma_3}^{-1}
\cdots T_{\sigma_{N-1}}^{-1}=T_{\tau_1}^{-1}T_{\tau_2}^{-1}\cdots
T_{\tau_{N-2}}^{-1}\pi^{-1},
\]
so that \eqref{dd3} will follow from
\[T_{\sigma_1}=T_{w_0}T_{\tau_1}^{-1}T_{\tau_2}^{-1}\cdots T_{\tau_{N-2}}^{-1}
\]
in $H(k)$. But this is a direct consequence of the second expression of $w_0$
in \eqref{redw0}.
\end{proof}
\begin{cor}\label{centerokcor}
For $p\in\C[T]^{S_N}$, we have $p(Y)^\circ=p(Y)$, where
${}^\circ\colon H(k^{-1})\rightarrow H(k)$ is the algebra
isomorphism defined in Remark \ref{circle}.
\end{cor}
\begin{proof}
This follows from the previous lemma and the fact that
$\dagger=\ddagger\circ{}^\circ$.
\end{proof}

\begin{defi}
{\bf (i)} Define
\[L_p^{x}:=\textup{Res}(D^x_{p(Y)})\in\mathbb{D},\qquad p\in\C[T]^{S_N},
\]
where $p(Y)$ is the corresponding
element in $\C_Y[T]^{S_N}=Z(H(k^{-1}))$.\\
{\bf (ii)} Define
\[L_p^{y}:=\textup{Res}(D^y_{p(Y)})\in\mathbb{D},\qquad p\in\C[T]^{S_N},
\]
where $p(Y)$ is the
corresponding element in $Z(H(k))$.
\end{defi}
By Corollary \ref{centerokcor} and \eqref{circleformula} we have
\[L_p^y=\iota L_p^x\iota,\qquad \forall\, p\in\C[T]^{S_N}.
\]
Furthermore, it is well-known (see \cite{C} and \cite{M}) that the
$L_p^x\in \C(T\times\{1\})\#(\Z^N\times \{e\})\subset\mathbb{D}$
are pairwise commuting and $S_N\times S_N$-invariant,
\[\mathrm{w}L_p^x\mathrm{w}^{-1}=L_p^x,\qquad \forall\,
\mathrm{w}\in S_N\times S_N.
\]
Similarly, the $L_p^y=\iota L_p^x\iota\in \C(\{1\}\times
T)\#(\{e\}\times\Z^N)\subset\mathbb{D}$ are pairwise commuting and
$S_N\times S_N$-invariant. Clearly also $\lbrack
L_p^x,L_{p^\prime}^y\rbrack=0$ for all $p,p^\prime\in\C[T]^{S_N}$
in $\mathbb{D}$.

For the elementary symmetric functions $e_i\in\C[T]^{S_N}$
($1\leq i\leq N$) given by
\[
e_i(t)=\sum_{\stackrel{I\subseteq\{1,\ldots,N\}}{\#I=i}}\prod_{j\in I}t_j,
\]
the corresponding $L_{e_i}^x$, viewed as elements in
\[\C(T)\#_q\mathbb{Z}^N\simeq \C(T\times \{1\})\#(\mathbb{Z}^N\times \{e\})
\subset\mathbb{D},
\]
are explicitly given by
\begin{equation}\label{Ruijsoper}
L_{e_i}^x=\sum_{\stackrel{I\subseteq\{1,\ldots,N\}}{\#I=i}}
\left(\prod_{\stackrel{r\in I}{s\not\in
I}}\frac{kx_r-k^{-1}x_s}{x_r-x_s}\right) \sum_{r\in
I}\epsilon_r\in\C(T)\#_q\Z^N,\qquad 1\leq i\leq N;
\end{equation}
see, e.g., \cite[\S 1.3.5]{C} and \cite{NK}. Hence, the
$L_{e_i}^x$ ($1\leq i\leq N$) are, under their natural
interpretation as $q$-difference operators on $\mathcal{M}(T)$,
Ruijsenaars' commuting, trigonometric $q$-difference operators
from \cite{R}.

\begin{defi}
Consider the bispectral problem
\begin{equation}\label{eqBiSP}
\begin{split}
(L_p^xf)(t,\gamma)&=p(\gamma^{-1})f(t,\gamma),\qquad \forall\,
p\in\C[T]^{S_N},\\
(L_p^{y}f)(t,\gamma)&=p(t)f(t,\gamma),\qquad \forall\, p\in\C[T]^{S_N}
\end{split}
\end{equation}
for $f\in\mathbb{K}$, where the equations \eqref{eqBiSP} are viewed as
identities between meromorphic functions in $(t,\gamma)\in T\times T$.
We write $\textup{BiSP}\subset\mathbb{K}$ for the set of
solutions $f\in\mathbb{K}$ of \eqref{eqBiSP}.
\end{defi}

\begin{rema}
The bispectral problem for ordinary linear differential operators
was introduced by Duistermaat and Gr{\"u}nbaum in \cite{DG}. Many
different types of bispectral problems have since been considered.
In particular, in \cite{GH} the bispectral problem for ordinary
linear second-order $q$-difference operators is investigated. For
$N=2$, our bispectral problem belongs to this class.
\end{rema}

The preceding remarks on the invariance properties of the
$L_p^x$ and the $L_p^y$ ($p\in\C[T]^{S_N}$) directly give
\begin{lem}
$\textup{BiSP}$ is an $\mathbb{S}_N$-invariant
$\mathbb{F}$-subspace of $\mathbb{K}$ with respect to the
usual $\mathbb{S}_N$-action $(\mathrm{w}f)(t,\gamma)=
f(\mathrm{w}^{-1}(t,\gamma))$ on $f\in\mathbb{K}$.
\end{lem}

We can now prove
the following bispectral version of the correspondence
between solutions of the quantum KZ equations and the spectral problem of
the $L_p^x$ ($p\in\C[T]^{S_N}$).

\begin{thm}\label{selfdualbi}
The linear functional $\chi_+\in H_0^*$ (see Lemma \ref{chiok})
defines a $\mathbb{S}_N$-equivariant $\mathbb{F}$-linear map
\[
\chi_+\colon \textup{SOL}\to\textup{BiSP}.
\]
\end{thm}
\begin{proof}
The $\mathbb{K}$-linear extended linear functional $\chi_+$
defines an $\mathbb{S}_N$-equivariant, $\mathbb{F}$-linear map
$\chi_+: H_0^{\mathbb{K}}\rightarrow \mathbb{K}$, since Lemma
\ref{chiok} implies that $\chi_+(\tau(\mathrm{w})f)=
\mathrm{w}(\chi_+f)$ for $\mathrm{w}\in\mathbb{S}_N$ and $f\in
H_0^{\mathbb{K}}$. Hence $\chi_+$ restricts to an
$\mathbb{S}_N$-equivariant, $\mathbb{F}$-linear map $\chi_+:
\textup{SOL}\rightarrow \mathbb{K}$.

It remains to show that $\chi_+(f)\in\textup{BiSP}$ if
$f\in\textup{SOL}$. Let $f\in\textup{SOL}$. By Proposition
\ref{basisprop} and $\mathbb{F}$-linearity, it suffices only to
consider $f$ of the form $f=Uv$ for $v\in H_0$. Then
$\chi_+(f)=\chi_+(Uv)=\phi_{\chi_+,v}^U$. For $p\in\C[T]^{S_N}$ we
have
\[\begin{split}
\left(L_p^{x}\phi_{\chi_+,v}^U\right)(t,\gamma)&=
\left(\textup{Res}(D^x_{p(Y)})(\phi_{\chi_+,v}^U)\right)(t,\gamma)=
\phi_{\chi_+,v}^{\vartheta_{\mathcal{T}}(D^x_{p(Y)})U}(t,\gamma)\\
&=\phi_{\chi_+,v}^{\eta(p(Y)^\dagger)U}(t,\gamma)
=p(\gamma^{-1})\phi_{\chi_+,v}^U(t,\gamma)
\end{split}\]
as meromorphic functions in $(t,\gamma)\in T\times T$,
where the last equality follows from Lemma \ref{centerok},
\eqref{eigfc} and the fact $p\in\C[T]^{S_N}$. Similarly,
\[\begin{split}
\left(L_p^{y}\phi_{\chi_+,v}^U\right)(t,\gamma)&=
\left(\textup{Res}(D^y_{p(Y)})(\phi_{\chi_+,v}^U)\right)(t,\gamma)=
\phi_{\chi_+,v}^{\vartheta_{\mathcal{T}}(D^y_{p(Y)})U}(t,\gamma)\\
&=\phi_{\chi_+,v}^{C_\iota\iota(\eta(p(Y)^\ddagger))C_\iota U}(t,\gamma)
=p(t)\phi_{\chi_+,v}^U(t,\gamma)
\end{split}\]
as meromorphic functions in $(t,\gamma)\in T\times T$,
hence $f=\phi_{\chi_+,v}^U\in\textup{BiSP}$.
\end{proof}

\subsection{Bispectral Harish-Chandra series}\label{BiHC}

\begin{defi}
We call $\Phi_\kappa^+:=\chi_+(\Phi_\kappa)\in\textup{BiSP}$ the basic
Harish-Chandra series solution of the bispectral problem.
\end{defi}
\begin{cor}
The solution $\Phi_\kappa^+\in\textup{BiSP}$ of the bispectral
problem is selfdual, i.e.,
\[\Phi_\kappa^+(t,\gamma)=\Phi_\kappa^+(\gamma^{-1},t^{-1})
\]
as meromorphic functions in $(t,\gamma)\in T\times T$.
\end{cor}
\begin{proof}
By Theorem \ref{selfdualTHM}, we have
\[\Phi_\kappa^+(t,\gamma)=\chi_+(C_\iota\Phi_\kappa(\gamma^{-1},t^{-1})).
\]
But $\chi_+C_\iota=\chi_+$, hence the result.
\end{proof}
\begin{rema}
In \cite{ES}, for special values of $k$, the function
$\Phi_\kappa^+$ is constructed as formal power series in terms of
generalized characters of Verma modules over the quantum group
$\mathcal{U}_q(\mathfrak{sl}_N)$ (see also \cite{EK1},
\cite{EK2}). The quantum group approach also leads to the
self-duality of $\Phi_\kappa^+$; see \cite[Thm. 5.6]{ES} (see
\cite{EK2}).
\end{rema}

Note that $\Phi_\kappa^+=W_\kappa\Psi^+$ with $\Psi^+=\chi_+(\Psi)$.
For $\alpha\in Q_+$ set
\[\Gamma_\alpha^+(\gamma)=\chi_+(\Gamma_\alpha(\gamma))
\]
as meromorphic function in $\gamma\in T$.
By Lemma \ref{Consequence1} and Theorem \ref{duall},
$\Gamma_\alpha^+$ is analytic at $T\setminus\mathcal{S}_+$
and
\[\Gamma_0^+(\gamma)=k^{{N\choose 2}}K(\gamma)
\]
with $K$ given by \eqref{Kexp}.
Recall that
the solution space $\textup{BiSP}$ of the bispectral problem
is $\mathbb{S}_N$-stable. In particular, we have solutions
$\Phi_w^+\in\textup{BiSP}$ given by
\begin{equation}\label{Phiw}
\Phi_w^+(t,\gamma):=\Phi_\kappa^+(t,w^{-1}\gamma).
\end{equation}
These are solutions of the bispectral problem
which are asymptotically free in the asymptotic sector
$\{t\in T \, | \, |t^{\alpha_i}|\gg 0\,\, \forall\, 1\leq i<N\}$
in the following sense: by Lemma \ref{Consequence1} we have
$\Phi_w^+(t,\gamma)=W_\kappa(t,w^{-1}\gamma)\Psi_w^+(t,\gamma)$
with $\Psi_w^+(t,\gamma):=\Psi^+(t,w^{-1}\gamma)$ admitting,
for $\epsilon>0$ sufficiently small, the power
series expansion
\begin{equation}\label{psPsi+}
\Psi_w^+(t,\gamma)=\sum_{\alpha\in Q_+}
\Gamma_\alpha^+(w^{-1}\gamma)t^{-\alpha}
\end{equation}
for $(t,\gamma)\in B_\epsilon^{-1}\times T\setminus w(\mathcal{S}_+)$,
converging normally in compacta of $B_\epsilon^{-1}\times T\setminus
w(\mathcal{S}_+)$.

\begin{prop}\label{HCindependent}
The set of asymptotic solutions $\{\Phi_w^+\}_{w\in
S_N}\subset\textup{BiSP}$ of the bispectral problem is
$\mathbb{F}$-linearly independent.
\end{prop}
\begin{proof}
Suppose that
\[\sum_{w\in S_N}a_w(t,\gamma)\Phi_w^+(t,\gamma)=0
\]
as meromorphic functions in $(t,\gamma)\in T\times T$ with
coefficients $a_w\in\mathbb{F}$ ($w\in S_N$).
Replacing $t$ by $q^{-m\delta}t$ ($m\in\mathbb{N}$)
and using \eqref{fff} we obtain
\begin{equation}\label{alm0}
\sum_{w\in S_N}k^{-m\langle\delta,\delta\rangle}\gamma^{-mww_0(\delta)}
a_w(t,\gamma)W_\kappa(t,w^{-1}\gamma)\Psi_w^+(q^{-m\delta}t,\gamma)=0
\end{equation}
as meromorphic functions in $(t,\gamma)\in T\times T$. Fix $u\in S_N$.
We are going to derive from \eqref{alm0}
that $a_u=0$. For this we will use the fact that for $w\not=u$,
\begin{equation}\label{thistouse}
\lim_{m\rightarrow\infty}\zeta^{m(uw_0(\delta)-ww_0(\delta))}=
\lim_{m\rightarrow\infty}(w_0u^{-1}\zeta)^{m(\delta-w_0u^{-1}ww_0(\delta))}=0
\end{equation}
if $\zeta\in uw_0(B_1)$.

Recall the $W$-invariant subset $\mathcal{S}\subset T$ (see \eqref{Scal}),
which contains $\mathcal{S}_+$.
For generic $\zeta\in T$ (concretely, $\zeta\not\in\mathcal{S}$, and
$a_w(t,\gamma)$ and $W_\kappa(t,w^{-1}\gamma)$ specializable at
$\gamma=\zeta$ for all $w\in S_N$),
it follows from Proposition \ref{spgauge}
and \eqref{alm0} that, for all $m\in\mathbb{N}$,
\begin{equation}\label{alm}
\sum_{w\in S_N}\zeta^{m(uw_0(\delta)-ww_0(\delta))}
a_w(t,\zeta)W_\kappa(t,w^{-1}\zeta)\Psi_w^+(q^{-m\delta}t,\zeta)=0
\end{equation}
as meromorphic function in $t\in T$.
Using \eqref{thistouse} and
the power series expansion \eqref{psPsi+},
the limit $m\rightarrow \infty$ of \eqref{alm} yields,
for generic $\zeta\in uw_0(B_1)$,
\[a_u(t,\zeta)W_\kappa(t,u^{-1}\zeta)\Gamma_0^+(u^{-1}\zeta)=0
\]
as meromorphic function in $t\in T$. This implies $a_u=0$, as desired.
\end{proof}
\begin{cor}
The map $\chi_+: \textup{SOL}\rightarrow \textup{BiSP}$ is
injective.
\end{cor}
\begin{proof}
Note that $\chi_+\bigl(\tau(e,w)\Phi_\kappa)=\Phi_w^+$ ($w\in S_N$).
The statement follows now directly from Proposition \ref{HCindependent} and
Proposition \ref{basisprop}.
\end{proof}

\subsection{Specialized central character and Harish-Chandra series}
\label{spec}
We write
\[\textup{SP}_\zeta=\{f\in \mathcal{M}(T) \,\, | \,\,
L_p^{x}f=p(\zeta^{-1})f\quad\forall\, p\in\C[T]^{S_N}\}
\]
for the spectral problem of the Ruijsenaars $q$-difference operators
with fixed spectral parameter $\zeta\in T$. Note that $\textup{SP}_\zeta\subset
H_0^{\mathcal{M}(T)}$ is $S_N$-stable, with $S_N$-action on
$H_0^{\mathcal{M}(T)}$
given by $(wf)(t)=f(w^{-1}t)$ for $f\in H_0^{\mathcal{M}(T)}$ and
$w\in S_N$.

By \cite[Prop. 5.2]{Et}, the quantum KZ equations \eqref{qKZ} are
consistent for all values $\zeta\in T$ of the central character.
The arguments from Subsection \ref{Corr}, applied to the quantum
KZ equations \eqref{qKZ} for fixed $\zeta$ and with the role of
$U$ taken over by an invertible matrix solution $U_\zeta$ of
\eqref{qKZ}, result in the following special case of the
Cherednik-Matsuo correspondence from \cite{Cref,CInd} (concretely,
in the notations of \cite{CInd}, take the principal series module
$V=M(\zeta)$ in \cite[Thm. 4.2]{CInd} and let $\tau$ be the
projection from $M(\zeta)$, along the direct sum decomposition of
$M(\zeta)$ in $H_0$-isotypical components, onto the trivial
component).
\begin{prop}\label{CMspecialized}
Let $\zeta\in T$. Then $\chi_+$ defines an $\mathcal{E}(T)$-linear
$S_N$-equivariant map $\chi_+\colon \textup{SOL}_\zeta \rightarrow
\textup{SP}_\zeta$.
\end{prop}
For a further analysis of the map
$\chi_+: \textup{SOL}_\zeta\rightarrow \textup{SP}_\zeta$,
we refer to \cite{CInd} and \cite[\S 1.3.4]{C}.

Harish-Chandra type series solutions of the spectral problem
of the Ruijsenaars $q$-difference operators $L_p^{x}$ ($p\in\C[T]^{S_N}$)
with fixed spectral parameter $\zeta\in T$ were studied in, e.g.,
\cite{EK1} and \cite{KK} (see also \cite{LS} for arbitrary
root systems). The results of the previous subsection allow
us to reobtain these solutions
by specialization of the basic Harish-Chandra series
$\Phi_\kappa^+$. It leads to
new results on the convergence and singularities of
these solutions,
which we state now explicitly.

By Subsection \ref{singularities}, for generic $\kappa\in
\mathbb{C}^\times$
the basic Harish-Chandra series
$\Phi_\kappa^+(t,\gamma)$ is specializable at $\gamma=\zeta$
when $\zeta\in T\setminus\mathcal{S}_+$. Concretely,
for $\zeta\in T\setminus\mathcal{S}_+$ and generic
$\kappa\in\mathbb{C}^\times$,
we can write
\[\Phi_\kappa^+(t,\zeta)=W_\kappa(t,\zeta)\Psi^+(t,\zeta)
\]
as meromorphic function in $t\in T$, where $\Psi^+=\chi_+(\Psi)$
(see Subsection \ref{BiHC}). Due to the results in Subsection
\ref{singularities} (see Proposition \ref{Psising}) we obtain the
following result.
\begin{cor}
For $\zeta\in T\setminus\mathcal{S}_+$, the meromorphic function
$\Psi^+(t,\zeta)$ in $t\in T$ is analytic at $t\in T\setminus
\mathcal{S}_+^{-1}$.
\end{cor}

Let $\zeta\in T\setminus \mathcal{S}$, where $\mathcal{S}\subset
T$ is the $W$-invariant set \eqref{Scal}. For $\kappa\in\C^\times$
such that $W_\kappa(t,w^{-1}\gamma)$ may be specialized at
$\gamma=\zeta$ for all $w\in S_N$, the asymptotic solutions
$\Phi_w^+(t,\gamma)$ ($w\in S_N$) of the bispectral problem (see
\eqref{Phiw}) may thus be specialized at $\gamma=\zeta$, giving
rise to solutions $\Phi_w^+(\cdot;\zeta)\in\textup{SP}_\zeta$
($w\in S_N$); see Corollary \ref{consistencyzeta} and Proposition
\ref{CMspecialized}. Observe that for $\epsilon>0$ sufficiently
small,
\[\Phi_w^+(t,\zeta)=W_\kappa(t,w^{-1}\zeta)
\sum_{\alpha\in Q_+}\Gamma_\alpha^+(w^{-1}\zeta)
t^{-\alpha}
\]
for $t\in B_\epsilon^{-1}$, with normal convergence of the power series
on compacta of $B_{\epsilon}^{-1}$. Since $\zeta\not\in\mathcal{S}$
we furthermore have
\[\Gamma_0^+(w^{-1}\zeta)=k^{{N\choose 2}}K(w^{-1}\zeta)\not=0,
\]
with $K$ given by \eqref{Kexp}.

\begin{defi} Let $\zeta\in T\setminus\mathcal{S}$.
The $\Phi_w^+(\cdot;\zeta)\in \textup{SP}_\zeta$ ($w\in S_N$)
are the Harish-Chandra series solutions of the spectral problem
$L_p^{x}f=p(\zeta^{-1})f$ ($p\in\C[T]^{S_N}$).
\end{defi}

\begin{rema}
In \cite{EK1} (and \cite{LS}) the Harish-Chandra series are
investigated as formal power series solutions to the spectral
problem of the Ruijsenaars operators. The advantage of the present
approach is the fact that it implies the convergence of the formal
power series, basically as a consequence of a general statement
about convergence of formal power series solutions of holonomic
systems of $q$-difference equations (see the appendix). Chalykh's
\cite{Ch} Baker-Akhiezer functions arise as Harish-Chandra series
solutions for special values of $k$; see \cite[\S 4.4]{LS}. In
\cite{KK}, the Harish-Chandra series solutions of the Ruijsenaars
operators are constructed as matrix coefficients of products of
vertex operators. By this approach, one obtains an explicit
integral representation of the Harish-Chandra series.
\end{rema}

\begin{rema}
Observe that
\begin{equation}\label{limform}
\lim_{\stackrel{\lambda\in\Lambda:}{\lambda\rightarrow\infty}}
\frac{\Phi_\kappa^+(t,q^\lambda k^{-\delta})}
{W_\kappa(t,q^\lambda k^{-\delta})}=\Gamma_0^+(t^{-1})=
k^{{N\choose 2}}K(t^{-1}),
\end{equation}
with $\lambda\rightarrow\infty$ meaning
$\lambda_i-\lambda_{i+1}\rightarrow \infty$ for all $1\leq i<N$.
Thus, $K$ (see \eqref{Kexp}) is a normalized limit of the
asymptotic solutions $\Phi_\kappa^+(\cdot,q^\lambda
k^{-\delta})\in \textup{SP}_{q^{\lambda}k^{-\delta}}$. The
solution space $\textup{SP}_{q^{\lambda}k^{-\delta}}$ contains the
symmetric Macdonald polynomial of degree $\lambda\in\Lambda$. It
turns out though that $\Phi_\kappa^+(\cdot,q^\lambda k^{-\delta})$
is not a multiple of the Macdonald polynomial of degree
$\lambda\in\Lambda$, but
$\Phi_\kappa^+(\cdot,q^{w_0(\lambda)}k^{\delta})$ is (this will
become apparent in the next section). On the other hand, the
leading coefficient $K$ (see \eqref{Kexp}) also naturally appears
as a normalized limit of the Macdonald polynomial when the degree
$\lambda\in\Lambda$ of the polynomial tends to infinity; see
\cite[Lemma 4.3]{CWhit} (this limit was proven in \cite{RuF} in
the $L^2$-sense).
\end{rema}

\section{Polynomial theory}\label{SectionPolynomial}

We assume throughout this section that
$k\in\mathbb{C}^\times$ satisfies the generic conditions
\begin{equation}\label{generic}
\begin{split}
k^{2j}&\not\in q^{\mathbb{Z}} ,\qquad \forall\, 1\leq j\leq N,\\
k^{\langle \delta, \varpi_j-w(\varpi_j)\rangle}&\not\in q^{\mathbb{Z}},\qquad
\forall\, 1\leq j<N, \forall\, w\in S_N: w(\varpi_j)\not=\varpi_j.
\end{split}
\end{equation}

\subsection{Polynomial solutions of the quantum KZ equation}
We are going to use a special case of Proposition \ref{shift} to
create $S_N$-invariant (with respect to the $S_N$-action
$\varsigma$ on $\textup{SOL}_\zeta$; see \eqref{actionone})
polynomial solutions of the quantum KZ equations.

\begin{lem}\label{poleex}
Let $\lambda\in\Lambda$. The
 possible
poles of the $\C[T]\otimes\textup{End}(H_0)$-valued rational function
\[
\gamma\mapsto C_{(e,-\lambda)}(\cdot,q^\lambda\gamma)=
C_{(e,\lambda)}(\cdot,\gamma)^{-1}
\]
in $\gamma\in T$ are at $\gamma^\alpha\in k^2q^{-\mathbb{N}}$ for some
$\alpha\in R^+$. The possible poles of
\[\gamma\mapsto C_{(e,\lambda)}(\cdot,\gamma)
\]
are at $\gamma^\alpha\in k^{-2}q^{-\mathbb{N}}$ for some $\alpha\in R^+$.
\end{lem}
\begin{proof}
Since $R_i(z)$ has only a (simple) pole at $z=k^{-2}$, this
follows from \eqref{fundamental} and the cocycle property of $C$;
see Lemma \ref{Consequence1}.
\end{proof}
Since $k$ satisfies $k^{2j}\not\in q^{\mathbb{Z}}$ for $1\leq
j\leq N$ by \eqref{generic}, the spectrum of $\eta_{q^\lambda
k^{-\delta}}\bigl(\C_Y[T])$ is simple and the $\xi_w(q^\lambda
k^{-\delta})$ ($w\in S_N$) form a $\C$-basis of $H_0$ for all
$\lambda\in\Lambda$. Furthermore, for such $k$ we have that
$\gamma\mapsto C_{(e,\lambda)}(\cdot,\gamma)^{\pm 1}$ is regular
at $\gamma=k^{-\delta}$ for all $\lambda\in\Lambda$; see Lemma
\ref{poleex}. The additional conditions on $k$ in \eqref{generic}
will play a role in Subsection \ref{1} and Subsection \ref{2}.

Proposition \ref{shift} now immediately implies the following result.
\begin{cor}\label{corshift}
Let $\lambda\in\Lambda$. Then $f\mapsto
C_{(e,\lambda)}(\cdot,k^{-\delta})^{-1}f$ defines an
$S_N$-equivariant isomorphism
$\textup{SOL}_{k^{-\delta}}\rightarrow \textup{SOL}_{q^\lambda
k^{-\delta}}$.
\end{cor}
The special interest in the quantum KZ equations for the
particular central characters $\gamma= q^\lambda
k^{-\delta}$ ($\lambda\in\Lambda$)
comes from the fact that it admits $S_N$-invariant
polynomial solutions. The key step in deriving this result is the
following lemma.
\begin{lem}\label{lemconstant}
The element $v_+:=\sum_{w\in S_N}k^{\ell(w)}T_w\in H_0$ is a
constant $S_N$-invariant solution of the quantum KZ equation with
central character $k^{-\delta}$. In other words,
\[C_{\lambda}^{k^{-\delta}}(t)v_+=v_+,\qquad \forall\,\lambda\in\Z^N.
\]
\end{lem}
\begin{proof}
Note that $R_i(z)v_+=v_+$, so by Lemma \ref{etaexplicit} and Lemma
\ref{explicit}(ii), for any $\zeta\in T$,
\[
C_{\varpi_i}^\zeta(t)v_+=\eta_\zeta(\pi)^iv_+= \sum_{w\in
S_N}k^{\ell(w)}\zeta^{w^{-1}w_0\varpi_i}T_{\sigma^iw}= \sum_{w\in
S_N}k^{\ell(\sigma^{-i}w)}\zeta^{w^{-1}\varpi_i}T_w.
\]
Then use $\ell(\sigma^{-i}w)-\ell(w)=\langle
\delta,w^{-1}\varpi_i\rangle$ for $1\leq i\leq N$ (for the proof
of this formula, it suffices to prove it for $i=1$. In that case,
look at the positive roots that are mapped to negative roots by
$\sigma^{-1}w$). It implies that
\[
C_{\varpi_i}^\zeta(t)v_+=\sum_{w\in S_N}
k^{\ell(w)}(k^\delta\zeta)^{w^{-1}\varpi_i}T_w.
\]
In particular, $C_{\varpi_i}^{k^{-\delta}}(t)v_+=v_+$ for all $i$.
Note, furthermore, that $R_i(z)v_+=v_+$ implies that
$\varsigma(s_i)v_+=C_{s_i}^{k^{-\delta}}(t)v_+=v_+$ for all $1\leq
i<N$ (with $\varsigma$ given by \eqref{actionone}). Hence,
$v_+\in\textup{SOL}_{k^{-\delta}}$ is $S_N$-invariant.
\end{proof}
\begin{prop}\label{trr}
For $\lambda\in\Lambda$, the nonzero
$S_N$-invariant solution
\[Q_\lambda:=C_{(e,\lambda)}(\cdot,k^{-\delta})^{-1}v_+\in
\textup{SOL}_{q^\lambda k^{-\delta}}
\]
of the quantum KZ equation is an $H_0$-valued Laurent polynomial
on $T$ satisfying
\begin{equation}\label{triangularity}
Q_\lambda(t)=\sum_{\alpha\in
Q_+}K_\alpha(\lambda)t^{\lambda-\alpha},
\end{equation}
with $K_\alpha(\lambda)\in H_0$ (all but finitely many terms zero).
\end{prop}
\begin{proof}
Note that Corollary \ref{corshift} and
Lemma \ref{lemconstant} imply that
$0\not=Q_\lambda\in\textup{SOL}_{q^\lambda k^{-\delta}}$ and that
$Q_\lambda$ is $S_N$-invariant. The triangularity property
\eqref{triangularity} follows from the cocycle property,
\eqref{fundamental}, the explicit form of the $R_i(z)$ and the fact that
\[\eta(\pi)(t^{-1})^{-i}T_w=t^{w^{-1}\varpi_i}T_{\sigma^{-i}w},\qquad
w\in S_N,
\]
which is a direct consequence of Lemma \ref{etaexplicit}.
\end{proof}
\subsection{Duality}
\begin{lem}\label{evaluationLemma}
For $\lambda\in\Lambda$, we have $Q_\lambda(k^\delta)=v_+$.
\end{lem}
\begin{proof}
Since $v_+\in\textup{SOL}_{k^{-\delta}}$ and $C_\iota v_+=v_+$, we
obtain
\begin{equation*}
\begin{split}
Q_\lambda(k^\delta)&=C_{(e,-\lambda)}(k^{\delta},q^\lambda
k^{-\delta})v_+\\
&=C_\iota C_{(-\lambda,e)}
(q^{-\lambda}k^{\delta},k^{-\delta})C_\iota v_+=v_+,
\end{split}
\end{equation*}
for $\lambda\in\Lambda$.
\end{proof}
The polynomial solutions $Q_\lambda$ of the quantum KZ equations
are self-dual in the following sense.
\begin{prop}\label{dualPROP}
For $\lambda,\mu\in\Lambda$, we have
\[Q_\lambda(q^{-\mu} k^{\delta})=C_\iota Q_\mu(q^{-\lambda} k^{\delta}).\]
\end{prop}
\begin{proof}
For $\lambda,\mu\in\Lambda$ we have, using $v_+\in\textup{SOL}_{k^{-\delta}}$
and the previous lemma,
\begin{equation}\label{dualstep}
\begin{split}
    Q_\lambda(q^{-\mu}k^{\delta})&=C_{(e,-\lambda)}(q^{-\mu}k^\delta,
    q^\lambda k^{-\delta})v_+\\
    &=C_{(e,-\lambda)}(q^{-\mu}k^\delta,
    q^\lambda k^{-\delta})C_{(-\mu,e)}(q^{-\mu}k^\delta
    ,k^{-\delta})v_+\\
    &=C_{(-\mu,-\lambda)}(q^{-\mu}k^\delta,
    q^\lambda k^{-\delta})v_+.
\end{split}
\end{equation}
Since $C_{(-\mu,-\lambda)}(q^{-\mu}k^\delta,q^\lambda k^{-\delta})=
C_\iota C_{(-\lambda,-\mu)}(q^{-\lambda}k^\delta,
    q^\mu k^{-\delta})C_\iota$ and $C_\iota v_+=v_+$, we conclude from
\eqref{dualstep} that
$Q_\lambda(q^{-\mu}k^\delta)=C_\iota Q_\mu(q^{-\lambda}k^{\delta})$.
\end{proof}
\subsection{Relation to the basic asymptotically free solution}\label{1}
In this subsection, we relate the polynomial solutions $Q_\lambda$
($\lambda\in\Lambda$) of the quantum KZ equations to the basic
asymptotic solution $\Phi_\kappa$. Some care is needed though: it
is not possible to specialize all the asymptotic solutions
$C_{(e,w)}(t,\gamma)\Phi_\kappa(t,w^{-1}\gamma)$ ($w\in S_N$) to
$\gamma=q^{\lambda}k^{-\delta}$ ($\lambda\in\Lambda$) since
$q^{\lambda}k^{-\delta}\in\mathcal{S}$; see Corollary
\ref{consistencyzeta}. We shall see that
$C_{(e,w_0)}(t,\gamma)\Phi_\kappa(t,w_0\gamma)$ can be specialized
at $\gamma=q^\lambda k^{-\delta}$, which is sufficient for our
purposes.

\begin{lem}\label{lemaltgauged}
Let $\lambda\in\Lambda$. There exists a unique
$\Xi_\lambda\in\textup{SOL}_{q^{\lambda}k^{-\delta}}$ such that,
for $\epsilon>0$ sufficiently small, we have an $H_0$-valued power
series expansion
\[\Xi_\lambda(t)=\sum_{\alpha\in Q^+}
\widetilde{\Gamma}_\alpha(\lambda)t^{\lambda-\alpha}
\]
converging normally on compacta of $B_\epsilon^{-1}$ and with
leading coefficient
\[\widetilde{\Gamma}_0(\lambda)=\eta_{q^{\lambda}k^{-\delta}}(T_{w_0})
\xi_{w_0}(q^{\lambda}k^{-\delta}).
\]
\end{lem}
\begin{proof}
Consider the gauged quantum KZ equations for $1\leq i\leq N$,
\begin{equation}\label{altgauged}
\widetilde{A}_i(t)\widetilde{\Xi}(q^{-\varpi_i}t)=\widetilde{\Xi}(t),
\qquad \widetilde{\Xi}\in H_0^{\mathcal{M}(T)},
\end{equation}
with $q$-connection matrices $\widetilde{A}_i(t)=
q^{-\langle\lambda,\varpi_i\rangle} C_{(\varpi_i,e)}(t,q^\lambda
k^{-\delta})$. Note that $\widetilde{A}_N(t)= \textup{id}$; see
Lemma \ref{qconnLem}. Observe that $\widetilde{\Xi}$ is a solution
of the holonomic system \eqref{altgauged} of $q$-difference
equations if and only if $x^\lambda\widetilde{\Xi}\in
\textup{SOL}_{q^\lambda k^{-\delta}}$. By Corollary
\ref{asymptotic}, we have $\widetilde{A}_i\in
Q_0(\mathcal{A})\otimes\textup{End}(H_0)$ and
\begin{equation}\label{tildeA}
\widetilde{A}_i^{(0)}=q^{-\langle\lambda,\varpi_i\rangle}
k^{\langle\delta,\varpi_i\rangle}
\eta_{q^{\lambda}k^{-\delta}}
\bigl(T_{w_0}Y^{w_0(\varpi_i)}T_{w_0}^{-1}\bigr).
\end{equation}
The $\widetilde{A}_i^{(0)}$
($1\leq i<N$) are semisimple endomorphisms of $H_0$.
A basis of simultaneous eigenvectors of $H_0$ is given by
$\eta_{q^{\lambda}k^{-\delta}}(T_{w_0})\xi_w(q^\lambda k^{-\delta})$
($w\in S_N$).
In fact,
\[\widetilde{A}_i^{(0)}
\bigl(\eta_{q^{\lambda}k^{-\delta}}(T_{w_0})\xi_w(q^\lambda k^{-\delta})\bigr)=
\gamma_{w,i}
\eta_{q^{\lambda}k^{-\delta}}(T_{w_0})\xi_w(q^\lambda k^{-\delta})
\]
for all $1\leq i<N$ and $w\in S_N$ with
\[\gamma_{w,i}
=q^{\langle \lambda, w^{-1}w_0(\varpi_i)-\varpi_i\rangle}
k^{\langle\delta,\varpi_i-w^{-1}w_0(\varpi_i)\rangle};
\]
see \eqref{afterstepone}. Note, in particular, that
$\gamma_{w,i}\not\in q^{-\mathbb{N}}$ for all $w\in S_N$ and all
$1\leq i<N$ by the generic conditions \eqref{generic} on $k$, and
that $\gamma_{w_0,i}=1$ for all $1\leq i<N$. Hence, Theorem
\ref{merversion} in the appendix, applied to \eqref{altgauged} by
taking $M=N-1$, $q_i=q$ and variables $z_i=x^{-\alpha_i}$ ($1\leq
i<N$) shows that there exists a unique
$\widetilde{\Xi}\in\mathcal{M}(T)\otimes H_0$ satisfying
\eqref{altgauged} and admitting an $H_0$-valued power series
expansion
\[
\widetilde{\Xi}(t)=
\sum_{\alpha\in Q^+}\widetilde{\Gamma}_\alpha(\lambda)t^{-\alpha}
\]
converging normally on compacta of $B_\epsilon^{-1}$ for some $\epsilon>0$
small enough, and having as leading coefficient
$\widetilde{\Gamma}_0(\lambda)=
\eta_{q^{\lambda}k^{-\delta}}(T_{w_0})\xi_{w_0}(q^\lambda k^{-\delta})$.
This directly implies the lemma.
\end{proof}

Recall that the cocycle values $C_{(e,w)}(t,\gamma)$ ($w\in S_N$)
are independent of $t\in T$. We suppress $t$ from the notation
and simply write $C_{(e,w)}(\gamma)$. Recall that $C_{(e,w)}(\gamma)$ for
$w\in S_N$ is an
$\textup{End}(H_0)$-valued regular function in $\gamma\in T$.
\begin{thm}\label{polred}
Fix $\lambda\in\Lambda$.
For $\kappa\not\in q^{\mathbb{Z}}$,
the basic asymptotic solution $\Phi_\kappa(t,\gamma)$
of BqKZ can be specialized at $\gamma=q^{w_0(\lambda)}k^\delta$,
giving rise to a $H_0$-valued meromorphic function $\Phi_\kappa(t,
q^{w_0(\lambda)}k^\delta)$ in $t\in T$.
Then
\begin{equation}\label{relation}
Q_\lambda(t)=r_\kappa C_{(e,w_0)}(q^{\lambda}k^{-\delta})
\Phi_{\kappa}(t,q^{w_0(\lambda)}k^{\delta})
\end{equation}
with
\begin{equation}\label{rkappa}
r_\kappa=\theta(\kappa)^{N}k^{-{N\choose 2}}\prod_{1\leq i<j\leq N}
\frac{\bigl(k^{2(j-i+1)};q\bigr)_{\infty}}
{\bigl(k^{2(j-i)};q\bigr)_{\infty}}.
\end{equation}
\end{thm}
\begin{proof}
We first show that both $Q_\lambda$ and the right-hand side of
\eqref{relation} are nonzero scalar multiples of $\Xi_\lambda$.

We start with the right-hand side of \eqref{relation}. Since
$\Phi$ is $\mathbb{S}_N$-stable, we have
\[\Phi_{w_0}:=\tau(e,w_0)\Phi_\kappa\in\textup{SOL}.
\]
Concretely, it is given by
\[\Phi_{w_0}(t,\gamma)=C_{(e,w_0)}(\gamma)\Phi_{\kappa}(t,w_0(\gamma))=
W_\kappa(t,w_0(\gamma))C_{(e,w_0)}(\gamma)\Psi(t,w_0(\gamma)).
\]
Since
$w_0(q^{\lambda}k^{-\delta})=q^{w_0(\lambda)}k^\delta\not\in\mathcal{S}_+$
by \eqref{generic}, we may, in view of Proposition \ref{spgauge},
specialize $\Phi_{w_0}(t,\gamma)$ at
$\gamma=q^{\lambda}k^{-\delta}$, obtaining
$\Phi_{w_0}(\cdot,q^{\lambda}k^{-\delta})\in
\textup{SOL}_{q^{\lambda}k^{-\delta}}$. By \eqref{W} we have
\begin{equation}\label{Wevaluation}
W_\kappa(t,w_0(q^{\lambda}k^{-\delta}))=
k^{\langle\delta,\lambda\rangle}\theta(\kappa)^{-N}t^\lambda,
\end{equation}
hence by Proposition \ref{spgauge},
\[\Phi_{w_0}(t,q^{\lambda}k^{-\delta})=
k^{\langle\delta,\lambda\rangle}\theta(\kappa)^{-N}
\sum_{\alpha\in Q_+}
\Gamma_\alpha^{w_0}t^{\lambda-\alpha}
\]
with $\Gamma_\alpha^{w_0}=C_{(e,w_0)}(q^{\lambda}k^{-\delta})
\Gamma_\alpha(q^{w_0(\lambda)}k^{\delta})$. From the definitions
of $C_{(e,w_0)}$, $d_w$, $\eta$ and $\xi_{w_0}$ (see Subsections
\ref{intsec}, \ref{constr} and \ref{formal}) we have
\begin{equation}\label{Cxirelation}
\begin{split}
C_{(e,w_0)}(\gamma)T_{w_0}&=
d_{w_0}(\gamma^{-1})^{-1}\eta(T_{w_0})\xi_{w_0}(\gamma)\\
&=\left(\prod_{\alpha\in R^+}\frac{1}{k-k^{-1}\gamma^\alpha}\right)
\eta(T_{w_0})\xi_{w_0}(\gamma)
\end{split}
\end{equation}
as $H_0$-valued regular functions in $\gamma\in T$. By Theorem
\ref{duall}, the leading coefficient $\Gamma_0^{w_0}$ thus
simplifies to
\begin{equation*}
\begin{split}
\Gamma_0^{w_0}&=
K(q^{w_0(\lambda)}k^{\delta})C_{(e,w_0)}(q^{\lambda}k^{-\delta})
T_{w_0}\\
&=K(q^{w_0(\lambda)}k^{\delta})d_{w_0}(q^{-\lambda}k^{\delta})^{-1}
\eta_{q^{\lambda}k^{-\delta}}(T_{w_0})\xi_{w_0}(q^{\lambda}k^{-\delta}),
\end{split}
\end{equation*}
where $K$ is given by \eqref{Kexp}. Combined with the previous
lemma, we conclude that
\begin{equation}\label{version1}
\Phi_{w_0}(t,q^{\lambda}k^{-\delta})=
k^{\langle\delta,\lambda\rangle}\theta(\kappa)^{-N}
K(q^{w_0(\lambda)}k^{\delta})d_{w_0}(q^{-\lambda}k^{\delta})^{-1}
\Xi_\lambda(t).
\end{equation}
In view of \eqref{generic}, $\Phi_{w_0}(t,q^{\lambda}k^{-\delta})$
thus is a nonzero constant multiple of $\Xi_\lambda(t)$.

Next, we consider
$0\not=Q_\lambda\in\textup{SOL}_{q^{\lambda}k^{-\delta}}$. By
Lemma \ref{lemaltgauged} and \eqref{triangularity}, it suffices to
note that $K_0(\lambda)$ is a constant multiple of
$\eta_{q^{\lambda}k^{-\delta}}(T_{w_0})
\xi_{w_0}(q^{\lambda}k^{-\delta})$, which follows directly from
the fact that $K_0(\lambda)\in H_0$ satisfies
\[\widetilde{A}_i^{(0)}K_0(\lambda)=K_0(\lambda), \qquad
\forall 1\leq i\leq N,
\]
where $\widetilde{A}_i$ is given by \eqref{tildeA}; see the proof
Lemma of \ref{lemaltgauged}. Thus, $Q_\lambda(t)$ is a nonzero
constant multiple of $\Xi_\lambda(t)$, and we conclude that
\[Q_\lambda(t)=r_\kappa(\lambda)\Phi_{w_0}(t,q^{\lambda}k^{-\delta}),
\]
for some $r_\kappa(\lambda)\in\C^\times$. We first show that
$r_\kappa(\lambda)$ is independent of $\lambda\in\Lambda$.

For $w\in S_N$, we write $C_{(w,e)}(t)$ for the
$\gamma$-independent value $C_{(w,e)}(t,\gamma)$ of the cocycle.
Let $\lambda,\mu\in\Lambda$. By the $S_N$-invariance of
$Q_\lambda$, we then have, on the one hand,
\[\begin{split}
Q_\lambda(q^{-\mu}k^{\delta})&=
C_{(w_0,e)}(q^{-\mu}k^{\delta})
Q_\lambda(q^{-w_0(\mu)}k^{-\delta})\\
&=r_\kappa(\lambda)C_{(w_0,e)}(q^{-\mu}k^{\delta})
C_{(e,w_0)}(q^\lambda k^{-\delta})
\Phi_\kappa(q^{-w_0(\mu)}k^{-\delta},q^{w_0(\lambda)}k^{\delta})\\
&=r_\kappa(\lambda)C_{(w_0,w_0)}(q^{-\mu}k^{\delta},q^\lambda k^{-\delta})
\Phi_\kappa(q^{-w_0(\mu)}k^{-\delta},q^{w_0(\lambda)}k^{\delta}).
\end{split}\]
On the other hand, using the self-duality of $Q_\lambda$
(see Proposition \ref{dualPROP})
and of $\Phi_\kappa$ (see Theorem \ref{selfdualTHM}),
\[\begin{split}
Q_\lambda(q^{-\mu}k^{\delta})&=C_\iota Q_\mu(q^{-\lambda}k^{\delta})=
C_\iota C_{(w_0,e)}(q^{-\lambda}k^{\delta})
Q_\mu(q^{-w_0(\lambda)}k^{-\delta})\\
&=r_\kappa(\mu)C_\iota C_{(w_0,e)}(q^{-\lambda}k^{\delta})
C_{(e,w_0)}(q^\mu k^{-\delta})
\Phi_\kappa(q^{-w_0(\lambda)}k^{-\delta},q^{w_0(\mu)}k^{\delta})\\
&=r_\kappa(\mu)C_\iota C_{(w_0,w_0)}(q^{-\lambda}k^{\delta},q^\mu k^{-\delta})
\Phi_\kappa(q^{-w_0(\lambda)}k^{-\delta},q^{w_0(\mu)}k^{\delta})\\
&=r_\kappa(\mu)C_{(w_0,w_0)}(q^{-\mu}k^{\delta},q^{\lambda}k^{-\delta})
C_\iota \Phi_\kappa(q^{-w_0(\lambda)}k^{-\delta},q^{w_0(\mu)}k^{\delta})\\
&=r_\kappa(\mu)C_{(w_0,w_0)}(q^{-\mu}k^{\delta},q^{\lambda}k^{-\delta})
\Phi_\kappa(q^{-w_0(\mu)}k^{-\delta},q^{w_0(\lambda)}k^{\delta}).
\end{split}\]
We conclude that $r_\kappa(\lambda)=r_\kappa(\mu)$ if
$Q_\lambda(q^{-\mu}k^{\delta})\not=0$. In particular, since
$Q_\lambda(k^{\delta})=v_+\not=0$, we have $r_\kappa(\lambda)=r_\kappa(0)$
for all $\lambda\in\Lambda$.

It remains to compute $r_\kappa:=r_\kappa(0)$. Using the fact that
$C_{(e,s_i)}(\gamma)=C_\iota R_i(\gamma^{-\alpha_i}) C_\iota$,
with $R_i(z)=c_k(z)^{-1}(\eta(T_i^{-1})-k^{-1})+1$ for $1\leq
i<N$, as well as that
$C_\iota(T_{w^{-1}}^{-1}T_{w_0})=T_{w_0w^{-1}}$ for all $w\in
S_N$, we get $C_{(e,w_0)}(\gamma)T_{w_0}=\sum_{w\leq
w_0}e_w(\gamma)T_{w_0w^{-1}}$ as $H_0$-valued regular function in
$\gamma\in T$ with $e_w\in\C[T]$ and with
\[e_{w_0}(\gamma)=\prod_{\beta\in R^+}c_k(\gamma^{-\beta})^{-1}.
\]
Taking the $T_e$-coefficient in the expansion of the formula
\[v_+=Q_0=r_\kappa\Phi_{w_0}(\cdot,k^{-\delta})=
r_\kappa\theta(\kappa)^{-N}K(k^{\delta})C_{(e,w_0)}(k^{-\delta})T_{w_0}
\]
with respect to the $\C$-basis $\{T_{w}\}_{w\in S_N}$ of $H_0$,
we conclude that
\[r_\kappa=\theta(\kappa)^NK(k^\delta)^{-1}
\prod_{\beta\in R^+}c_k(k^{\langle\delta,\beta\rangle}).
\]
Substituting the explicit expressions \eqref{ck} and \eqref{Kexp}
of $c_k$ and $K$, respectively, we get the desired formula
\eqref{rkappa} for $r_\kappa$.
\end{proof}
The following formula is an analog for the $Q_\lambda$
($\lambda\in\Lambda$) of the evaluation formula for the self-dual
symmetric Macdonald polynomials (see Subsection \ref{2}).
\begin{cor}\label{leadingQ}
Let $\lambda\in\Lambda$ and write $Q_\lambda(t)=
\sum_{\alpha\in Q_+}K_\alpha(\lambda)t^{\lambda-\alpha}$ with
$K_\alpha(\lambda)\in H_0$ (see Proposition \ref{trr}). The
leading coefficient $K_0(\lambda)$ is given by
\[K_0(\lambda)=k^{\langle\delta,\lambda\rangle}
\left(\prod_{1\leq i<j\leq N}\prod_{m=0}^{\lambda_i-\lambda_j-1}
\frac{1-q^{-m}k^{2(j-i)}}{1-q^{-m}k^{2(j-i+1)}}\right)
k^{-{N\choose 2}}P(k^2)C_{(e,w_0)}(q^{\lambda}k^{-\delta})T_{w_0}\]
with
\begin{equation}\label{Poincare}
P(k^2)=\prod_{1\leq i<j\leq N}\frac{1-k^{2(j-i+1)}}{1-k^{2(j-i)}}.
\end{equation}
\end{cor}
\begin{proof}
By \eqref{relation} and Theorem \ref{duall}, we have for
$\lambda\in\Lambda$,
\[K_0(\lambda)=r_\kappa k^{\langle\delta,\lambda\rangle}
\theta(\kappa)^{-N}K(q^{w_0(\lambda)}k^{\delta})
C_{(e,w_0)}(q^{\lambda}k^{-\delta})T_{w_0},
\]
with $K$ given by \eqref{Kexp} and $r_\kappa$ given by \eqref{rkappa}.
Substituting the explicit expressions for $K$ and $r_\kappa$
we get the desired expression.
\end{proof}
The following consequence should be compared with the general
expansion formula of $v_+=\sum_{w\in S_N}k^{\ell(w)}T_w\in H_0$ in
terms of the $\xi_w(\gamma)$ ($w\in S_N$); see \cite[Lemma 2.27
(2)]{Op}.
\begin{cor}
The element $v_+=\sum_{w\in S_N}k^{\ell(w)}T_w\in H_0$ can be
written as
\begin{equation}\label{expressionsvplus}
\begin{split}
v_+&=k^{-{N\choose 2}}P(k^2)C_{(e,w_0)}(k^{-\delta})T_{w_0}\\
&=\left(\prod_{1\leq i<j\leq N}\frac{1}{1-k^{2(i-j)}}\right)
\eta_{k^{-\delta}}(T_{w_0})\xi_{w_0}(k^{-\delta}).
\end{split}
\end{equation}
\end{cor}
\begin{proof}
We have $v_+=Q_0=K_0(0)$, hence the previous corollary
gives the first equality of \eqref{expressionsvplus}.
The second equality then follows from \eqref{Cxirelation}.
\end{proof}
Applying the map $\chi_+$ to the first line
of \eqref{expressionsvplus} gives
\[\sum_{w\in S_N}k^{2\ell(w)}=P(k^2)
\]
with $P(k^2)$ given by \eqref{Poincare}, which is a well-known
product formula for the Poincar{\'e} series of $S_N$; see
\cite[Cor. (2.5)]{MacP}.

\subsection{Relation to symmetric self-dual Macdonald polynomials}\label{2}
In this subsection we collect various consequences of the previous
subsections for the symmetric Laurent polynomials
$\chi_+\bigl(Q_\lambda)\in\C[T]^{S_N}$ ($\lambda\in\Lambda$). We
keep the generic conditions \eqref{generic} on
$k\in\mathbb{C}^\times$. We denote
\[E_\lambda:=P(k^2)^{-1}\chi_+(Q_\lambda)\in\C[T]^{S_N},\qquad \lambda\in\Lambda.
\]
By Proposition, \ref{trr} we have
\[E_\lambda(t)=\sum_{\alpha\in Q_+}K_\alpha^+(\lambda)t^{\lambda-\alpha}
\]
with $K_\alpha^+(\lambda)=P(k^2)^{-1}\chi_+(K_\alpha(\lambda))\in\C$
all but finitely many zero, and with leading coefficient
$K_0^+(\lambda)\not=0$ by Corollary \ref{leadingQ} and \eqref{generic}.
\begin{thm}\label{Macdonaldconc}
The $E_\lambda\in\C[T]^{S_N}$ ($\lambda\in\Lambda$) are the symmetric self-dual
Macdonald polynomials. In other words, the $E_\lambda$ are the
unique symmetric regular functions on $T$ satisfying
\begin{equation}\label{eigkar}
L_p^{x}(E_\lambda)=p(q^{-\lambda}k^{\delta})E_\lambda\quad
\forall\, p\in\C[T]^{S_N}
\end{equation}
and $E_\lambda(k^{\delta})=1$ for all $\lambda\in\Lambda$.
\end{thm}
\begin{proof}
By Proposition \ref{CMspecialized}, $E_\lambda\in\C[T]^{S_N}$
satisfies \eqref{eigkar}.
Since the $S_N$-orbits $S_N(q^{-\lambda}k^{\delta})$
($\lambda\in\Lambda$) in $T$ are pairwise different by \eqref{generic},
the eigenvalue equations \eqref{eigkar}
uniquely characterize $E_\lambda\in\C[T]^{S_N}$ up to a nonzero
constant multiple. Now
\[E_\lambda(k^{\delta})=P(k^2)^{-1}\chi_+(Q_\lambda(k^{\delta}))=1
\]
by Lemma \ref{evaluationLemma}, which fixes the normalization of
the solution $E_\lambda\in\C[T]^{S_N}$ of \eqref{eigkar} uniquely.
\end{proof}
The duality property of $Q_{\lambda}$ (see Proposition
\ref{dualPROP}) immediately gives the well-known duality property
of the Macdonald polynomials.
\begin{cor}
The Macdonald polynomials $E_\lambda$ ($\lambda\in\Lambda$) are self-dual,
in the sense that
\[E_\lambda(q^{-\mu}k^{\delta})=E_\mu(q^{-\lambda}k^{\delta})
\]
for all $\lambda,\mu\in\Lambda$.
\end{cor}
\begin{rema}
The self-duality of (the suitably normalized) Macdonald
polynomials was initially proved by Koornwinder using Pieri
formulas in an unpublished manuscript (the argument is reproduced
in \cite[VI (6.6)]{MacBook}). Cherednik (\cite[Thm. 1.4.6]{C} and
\cite[Thm. 3.2]{CMac}) reproduced the self-duality of the
Macdonald polynomials using the anti-involution $*$ (see Theorem
\ref{dualthm}) on the double affine Hecke algebra.
\end{rema}
We also immediately reobtain the well-known evaluation formula for
the symmetric Macdonald polynomials; see \cite[VI (6.11)]{MacBook}
(the parameters $(n,q,t)$ in \cite[Chpt. VI]{MacBook} correspond
to $(N,q^{-1},k^2)$ in our notations).
\begin{cor}
For $\lambda\in\Lambda$ let
$P_\lambda:=K_0^+(\lambda)^{-1}E_\lambda\in\C[T]^{S_N}$ be the
monic symmetric Macdonald polynomial of degree $\lambda$. Then
\begin{equation}\label{evalf}
P_\lambda(k^{\delta})=k^{-\langle\delta,\lambda\rangle}
\prod_{1\leq i<j\leq N}\prod_{m=0}^{\lambda_i-\lambda_j-1}
\frac{1-q^{-m}k^{2(j-i+1)}}{1-q^{-m}k^{2(j-i)}}.
\end{equation}
\end{cor}
\begin{proof}
By the previous theorem we have $P_\lambda(k^{\delta})=K_0^+(\lambda)^{-1}$.
Corollary \ref{leadingQ} gives
\[K_0^+(\lambda)=k^{\langle\delta,\lambda\rangle}
\prod_{1\leq i<j\leq N}\prod_{m=0}^{\lambda_i-\lambda_j-1}
\frac{1-q^{-m}k^{2(j-i)}}{1-q^{-m}k^{2(j-i+1)}},
\]
which implies the desired result.
\end{proof}

\section{Appendix on holonomic systems of
$q$-difference equations}\label{appendix}

In the appendix we detail the construction of power series
solutions of holonomic systems of $q$-difference equations.
Special cases have been investigated in, e.g., \cite{Ao},
\cite{FR} and \cite[\S 12]{EFK}. Many arguments go back to
classical works \cite{Ad}, \cite{Bi}, \cite{Carm}, \cite{Tr} on
ordinary linear $q$-difference equations.

We begin with the construction of formal asymptotic solutions to
holonomic systems of $q$-difference equations. Let
$\mathbb{C}[[z]]=\mathbb{C}[[z_1,\ldots,z_M]]$ denote the ring of
formal power series in $M$ indeterminates $z_1,\ldots,z_M$ over
the complex numbers. Let $V$ be a finite-dimensional complex
vector space and let
\[
A_i\in\mathbb{C}[[z]]\otimes \textup{End}(V)
\]
for $i=1,\ldots,M$. Since $\mathbb{C}[[z]]\otimes\textup{End}(V)$
is isomorphic to
$\textup{End}_{\mathbb{C}[[z]]}(\mathbb{C}[[z]]\otimes V)$ as
$\C[[z]]$-module, we can view the $A_i$ as
$\mathbb{C}[[z]]$-linear endomorphisms of $\mathbb{C}[[z]]\otimes
V$. Fix $0<q_i<1$ for $1\leq i\leq M$. Define the $q_i$-dilation
operators
\[
\mathcal{T}_i\colon\mathbb{C}[[z]]\to\mathbb{C}[[z]]
\]
for $i=1,\ldots,M$ as the complex linear maps
\[
\mathcal{T}_i(\sum_{\mathbf{m}}d_{\mathbf{m}}z^{\mathbf{m}}):=
\sum_{\mathbf{m}}q_i^{m_i}d_{\mathbf{m}}z^{\mathbf{m}}\qquad
(d_{\mathbf{m}}\in\C),
\]
where we use multi-index notation $z^{\mathbf{m}}=z_1^{m_1}\cdots
z_M^{m_M}$ for $\mathbf{m}=(m_1,\ldots,m_M)$ with
$m_j\in\mathbb{Z}_{\geq 0}$. We also view $\mathcal{T}_i$ as
operators on $\mathbb{C}[[z]]\otimes V$ and on
$\mathbb{C}[[z]]\otimes \textup{End}(V)$. Consider the system of
first-order linear $q$-difference equations
\begin{equation}\label{eqGeneralq-diff}
    A_i\mathcal{T}_if=f,\qquad (i=1,\ldots,M)
\end{equation}
for $f\in\mathbb{C}[[z]]\otimes V$.

For $f\in\mathbb{C}[[z]]\otimes V$ and $A\in\mathbb{C}[[z]]\otimes
\textup{End}(V)$, we introduce the notations
\[
\begin{split}
f^{(m)}&:=f|_{z_{m+1}=\ldots=z_M=0}\in\mathbb{C}[[z_1,\ldots,z_{m}]]\otimes V\\
A^{(m)}&:=A|_{z_{m+1}=\ldots=z_M=0}
\in\mathbb{C}[[z_1,\ldots,z_{m}]]\otimes\textup{End}(V)
\end{split}
\]
for $0\leq m\leq M$, with the convention that $f^{(M)}=f$ and
$A^{(M)}=A$. We make the following assumptions on the system of
$q$-difference equations \eqref{eqGeneralq-diff}:

\textbf{(a)} The system \eqref{eqGeneralq-diff} is holonomic, that
is
\begin{equation}\label{eqGeneralHolonomic}
    A_i\mathcal{T}_i(A_j)= A_j\mathcal{T}_j(A_i)
\end{equation}
for all $1\leq i,j\leq M$. Note that the holonomy implies that the
leading coefficients $A_i^{(0)}\in\textup{End}(V)$ of $A_i$
mutually commute, i.e.,
$$[A_i^{(0)},A_j^{(0)}]=0$$
for all $1\leq i,j\leq M$.

\textbf{(b)} The complex linear endomorphisms
$A_1^{(0)},\ldots,A_M^{(0)}$ of $V$ are semisimple. Combined with
{\bf (a)} we thus have
\[V=\bigoplus_{\gamma\in S}V[\gamma]\]
with $V[\gamma]:=\{v\in V\mid A_i^{(0)}v=\gamma_iv\:\:\forall i\}$
($\gamma\in\C^M$) and $S:=\{\gamma\in\mathbb{C}^M\mid
V[\gamma]\neq\{0\}\}$.

\textbf{(c)} $(1^M):=(1,\ldots,1)\in\mathbb{C}^M$ belongs to $S$.

\textbf{(d)} $\gamma_k\notin q_k^{-\N}$ for all $\gamma\in S$ and
$1\leq k\leq M$.

\begin{prop}\label{propGeneralSol}
Fix $v\in V[(1^M)]$. Consider the system \eqref{eqGeneralq-diff}
of $q$-difference equations and suppose that
\textup{\textbf{(a)}-\textbf{(d)}} are satisfied. Then there
exists a unique solution $\Phi_v\in\mathbb{C}[[z]]\otimes V$ of
\eqref{eqGeneralq-diff} such that
\[\Phi_v^{(0)}=v.\]
\end{prop}
\begin{proof}
The proposition is a consequence of the following lemma.

\begin{lem}\label{lemprop}
Let $0\leq m<M$. Suppose one has a solution
\[
f_m\in\mathbb{C}[[z_1,\ldots,z_{m}]]\otimes V
\]
of the system of equations
\begin{equation}\label{k}
\begin{split}
A_r^{(m)}\mathcal{T}_rf_m&=f_m,\qquad 1\leq r\leq m,\\
A_s^{(m)}f_m&=f_m,\qquad m<s\leq M.
\end{split}
\end{equation}
Then there exists a unique
\begin{equation*}
\begin{split}
f_{m+1}&= \sum_{n\geq 0}f_{m;n}z_{m+1}^n\in
\mathbb{C}[[z_1,\ldots,z_{m+1}]]\otimes V
\end{split}
\end{equation*}
with $f_{m;n}\in\mathbb{C}[[z_1,\ldots,z_{m}]]\otimes V$ and
$f_{m;0}=f_m$ satisfying \eqref{k} with the role of $m$ replaced
by $m+1$:
\begin{equation}\label{k+1}
\begin{split}
A_r^{(m+1)}\mathcal{T}_rf_{m+1}&=f_{m+1},\qquad 1\leq r\leq m+1,\\
A_s^{(m+1)}f_{m+1}&=f_{m+1},\qquad m+1<s\leq M.
\end{split}
\end{equation}
\end{lem}

The proposition follows directly from the lemma as follows. Note
that $f_0:=v\in V[(1^M)]$ is a solution of \eqref{k} for $m=0$.
The formal $V$-valued series $f_{M}\in\C[[z]]\otimes V$, obtained
by repeated application of the lemma starting from $f_0=v$, gives
a formal $V$-valued series solution of \eqref{eqGeneralq-diff}
satisfying $f_M^{(0)}=v$. For uniqueness, assume that $f\in
\mathbb{C}[[z]]\otimes V$ is another formal $V$-valued series
satisfying $f^{(0)}=v$ and solving \eqref{eqGeneralq-diff}. We
have $f^{(0)}=v=f_0$ and $f^{(m)}$ solves \eqref{k} for all $0\leq
m<M$. Hence, by the uniqueness part of the lemma,
$f=f^{(M)}=f_{M}$.

We now proceed to prove the lemma. We assume that we have a formal
power series solution $f_m\in\C[[z_1,\ldots,z_m]]\otimes V$ of
\eqref{k} for some $0\leq m<M$. We write
\begin{equation}\label{defArn}
A_{r}^{(m+1)}= \sum_{n\geq 0}A_{r;n}^{(m)}z_{m+1}^{n},
\end{equation}
where
$A_{r;n}^{(m)}\in\mathbb{C}[[z_1,\ldots,z_{m}]]\otimes\textup{End}(V)$
and $A_{r;0}^{(m)}=A_r^{(m)}$. By a direct computation one
verifies that
\[f_{m+1}=
\sum_{n\geq 0}f_{m;n}z_{m+1}^{n}\in\C[[z_1,\ldots,z_{m+1}]]\otimes
V
\]
with $f_{m;n}\in\mathbb{C}[[z_{1},\ldots,z_{m}]]\otimes V$ and
$f_{m;0}=f_m$ satisfies the $q$-difference equation
\[A_{m+1}^{(m+1)}\mathcal{T}_{m+1}f_{m+1}=
f_{m+1}
\]
if and only if
\begin{equation}\label{recurrence}
\big(1-q_{m+1}^{n}A_{m+1}^{(m)}\big)f_{m;n}
=\sum_{l=1}^nq_{m+1}^{n-l}A_{m+1;l}^{(m)} f_{m;n-l}
\end{equation}
for all $n\in\mathbb{Z}_{\geq 0}$. The recurrence relations
\eqref{recurrence} admit a unique solution
$(f_{m;n})_{n\in\mathbb{Z}_{\geq 0}}$ with $f_{m;n}\in
\C[[z_1,\ldots,z_m]]\otimes V$ and with initial condition
$f_{m;0}=f_m$. Indeed, note that \eqref{recurrence} is valid for
$n=0$ since $f_{m;0}=f_m$ satisfies \eqref{k}. For $n\geq1$, we
have
\[
\det\bigl(1-q_{m+1}^{n}A_{m+1}^{(m)}\bigr)\in
\mathbb{C}[[z_1,\ldots,z_{m}]]^{\times},
\]
since
\[\det\big(1-q_{m+1}^{n}A_{m+1}^{(m)}\big)|_{z_1=\ldots=z_{m}=0}=
\det\big(1-q_{m+1}^{n}A_{m+1}^{(0)}\big)= \prod_{\gamma\in
S}(1-q_{m+1}^{n}\gamma_{m+1}^{\dim(V[\gamma])}) \neq0
\]
by assumption \textbf{(d)}. Cramer's rule then implies that
\eqref{recurrence} admits a unique solution
$(f_{m;n})_{n\in\mathbb{Z}_{\geq 0}}$ with $f_{m;0}=f_m$.

We conclude that there exists a unique
\[f_{m+1}=\sum_{n\geq 0}f_{m;n}z_{m+1}^{n}
\]
with $f_{m;n}\in\mathbb{C}[[z_1,\ldots,z_{m}]]\otimes V$ and
$f_{m;0}=f_m$ satisfying the $q$-difference equation
\begin{equation}\label{qdiffm+1}
A_{m+1}^{(m+1)}\mathcal{T}_{m+1}f_{m+1}=f_{m+1}.
\end{equation}
It remains to show that $f_{m+1}$ also satisfies \eqref{k+1} for
$r=1,\ldots,m$ and for $s=m+2,\ldots,M$.

Fix $1\leq r\leq m$ and write $g_r:=
A_r^{(m+1)}\mathcal{T}_rf_{m+1}$. Its expansion in powers of
$z_{m+1}$ is written as
\[g_r=\sum_{n\geq 0}g_{r;n}z_{m+1}^n
\]
with $g_{r;n}\in\mathbb{C}[[z_1,\ldots,z_{m}]]\otimes V$ and
$g_{r;0}=A_r^{(m)}\mathcal{T}_rf_m=f_m$, where the last equality
follows from the fact that $f_m$ is assumed to satisfy \eqref{k}.
Furthermore, using the holonomy \eqref{eqGeneralHolonomic} and the
$q$-difference equation \eqref{qdiffm+1} in $z_{m+1}$ satisfied by
$f_{m+1}$, we have
\begin{equation*}
\begin{split}
A_{m+1}^{(m+1)}\mathcal{T}_{m+1}g_r
&=A_{m+1}^{(m+1)}\mathcal{T}_{m+1}(A_r^{(m+1)})\mathcal{T}_r\mathcal{T}_{m+1}
f_{m+1}\\
&=A_r^{(m+1)}\mathcal{T}_{r}(A_{m+1}^{(m+1)})
    \mathcal{T}_r\mathcal{T}_{m+1}f_{m+1}\\
&=A_r^{(m+1)}\mathcal{T}_r
\bigl(A_{m+1}^{(m+1)}\mathcal{T}_{m+1}f_{m+1}\bigr)\\
&=A_r^{(m+1)}\mathcal{T}_rf_{m+1}=g_r.
\end{split}
\end{equation*}
We conclude that $g_r$ satisfies the characterizing properties of
$f_{m+1}$. Hence $g_r=f_{m+1}$, i.e.,
\[A_r^{(m+1)}\mathcal{T}_rf_{m+1}=f_{m+1}.
\]

Fix $m+1<s\leq M$ and write $g_s:= A_s^{(m+1)}f_{m+1}$. By a
similar argument as used in the previous paragraph, we now show
that $g_s=f_{m+1}$. We write
\[g_s=\sum_{n\geq
0}g_{s;n}z_{m+1}^n
\]
with $g_{s;n}\in\C[[z_1,\ldots,z_{m}]]\otimes V$ and
$g_{s;0}=A_s^{(m)}f_m=f_m$, where the last equality follows by the
assumption that $f_m$ satisfies \eqref{k}. Using the holonomy
\eqref{eqGeneralHolonomic}, the $q$-difference equation
\eqref{qdiffm+1}, and the obvious fact that
$\mathcal{T}_s(A_{m+1}^{(m+1)})=A_{m+1}^{(m+1)}$ since $s>m+1$, we
have
\begin{equation*}
\begin{split}
A_{m+1}^{(m+1)}\mathcal{T}_{m+1}g_s
&=A_{m+1}^{(m+1)}\mathcal{T}_{m+1}(A_s^{(m+1)})\mathcal{T}_{m+1}f_{m+1}\\
&=A_s^{(m+1)}\mathcal{T}_s(A_{m+1}^{(m+1)})\mathcal{T}_{m+1}f_{m+1}\\
&=A_s^{(m+1)}A_{m+1}^{(m+1)}\mathcal{T}_{m+1}f_{m+1}\\
&=A_s^{(m+1)}f_{m+1}=g_s.
\end{split}
\end{equation*}
We conclude that $g_s$ satisfies the characterizing properties of
$f_{m+1}$. Hence $g_s=f_{m+1}$, i.e.
\[A_s^{(m+1)}f_{m+1}=f_{m+1}.
\]
This completes the proof of Lemma \ref{lemprop}, and hence the
proof of Proposition \ref{propGeneralSol}.
\end{proof}

We investigate the analytical properties of the solution $\Phi_v$
when the $q$-connection matrices $A_i$ ($1\leq i\leq M$) satisfy,
besides the conditions {\bf (a)-(d)}, the following analyticity
condition:

{\bf (e)} For some $\epsilon>0$ the formal
$\textup{End}(V)$-valued series $A_i\in\C[[z]]\otimes
\textup{End}(V)$ ($1\leq i\leq M$) converges normally on compacta
of the open polydisc $D_\epsilon^{M}:=\{z\in\C^M \, | \,
|z_i|<\epsilon\,\, \forall\, i\}$.

In other words, if we expand $A_i$ along a basis of
$\textup{End}(V)$, condition {\bf (e)} requires its coefficients
in $\C[[z]]$ to converge normally on compacta of $D_\epsilon^M$.
\begin{prop}\label{convprop}
Suppose that the $q$-connection matrices $A_i\in\C[[z]]\otimes
\textup{End}(V)$ ($1\leq i\leq M$) satisfy {\bf (a)-(e)}. Let
$v\in V[(1^M)]$. There exists an $\epsilon>0$ such that the formal
$V$-valued series $\Phi_v\in\C[[z]]\otimes V$ converges normally
on compacta of $D_\epsilon^M$.
\end{prop}
\begin{proof}
For ease of notation, we will write $\Phi$ instead of $\Phi_v$. By
induction on $m=0,\ldots,M$ we prove that there exists
$\epsilon>0$ such that $\Phi^{(m)}\in\C[[z_1,\ldots,z_m]]\otimes
V$ converges normally on compacta of $D_\epsilon^m$.

For $m=0$, there is nothing to prove. Fix $0\leq m<M$ and suppose
$\Phi^{(m)}$ converges normally on compacta of $D_\delta^m$ for
some $\delta>0$. Write
\[
\Phi^{(m+1)}=\sum_{n\geq 0}\Phi_{m;n}z_{m+1}^n
\]
with $\Phi_{m;n}\in\mathbb{C}[[z_1,\ldots,z_{m}]]\otimes V$ and
$\Phi_{m;0}=\Phi^{(m)}$. Recall from the proof of Lemma
\ref{lemprop} that the formal $V$-valued power series $\Phi_{m;n}$
($n\geq 1$) are unique characterized by the recurrence relations
\begin{equation}\label{rec2}
\Phi_{m;n}
=\sum_{l=1}^nq_{m+1}^{n-l}
\bigl(1-q_{m+1}^nA_{m+1}^{(m)}\bigr)^{-1}A_{m+1;l}^{(m)}
\Phi_{m;n-l}
\end{equation}
for all $n\geq 1$. We use this recurrence formula to find bounds
for $\Phi_{m;n}$ in a neighborhood of $0\in\C^m$.

Turn the finite-dimensional complex vector space $V$ into an inner
product space, with corresponding norm denoted by $\|\cdot\|$. We
also write $\|\cdot\|$ for the operator norm of the associated
finite-dimensional normed space $\textup{End}(V)$. We continue the
proof of the proposition with two technical sublemmas. First we
find a proper uniform bound for $A_{m+1;l}^{(m)}$ for all $l$ (see
\eqref{rec2}).

\begin{lem}
There exists $\epsilon>0$ and $M>0$ such that
$\|A_{m+1;l}^{(m)}\|\leq M\epsilon^{-l}$ on
$\overline{D}_\epsilon^m$ for all $l\geq 0$.
\end{lem}
\begin{proof}
By {\bf (e)} there exists an $\epsilon>0$ such that
$A_{m+1}^{(m+1)}\in\C[[z_1,\ldots,z_{m+1}]]\otimes
\textup{End}(V)$ converges normally on compacta of the polydisc
$D_{2\epsilon}^{m+1}$. Consequently, for
$\epsilon<\epsilon^\prime<2\epsilon$ we have that
$\|A_{m+1}^{(m+1)}\|$ is uniformly bounded on the polydisc
$D_{\epsilon^\prime}^{m+1}$, say by $M>0$. In particular, we get
\[
\|(\partial^l_{z_{m+1}}A^{(m+1)}_{m+1})(z_1,\ldots,z_{m+1})|_{z_{m+1}=0}\|
\leq M{\epsilon}^{-l}l!
\]
for all $(z_1,\ldots,z_m)\in \overline{D}_\epsilon^m$ and for all
$l\geq 0$ (see, e.g., \cite{Ho} Theorem 2.2.7). This proves the
lemma in view of the definition \eqref{defArn} of
$A_{m+1;l}^{(m)}$.
\end{proof}

\begin{lem}
There exists an $\epsilon>0$ such that
$\Phi_{m;n}\in\C[[z_1,\ldots,z_m]]\otimes V$ converges normally on
compacta of $D_\epsilon^m$ for all $n\geq 0$. Furthermore, there
exists a constant $C>0$ (independent of $n$) such that
    \[
    \|\Phi_{m;n}\|\leq
    \frac{C}{1+C}\left(\frac{1+C}{q_{m+1}\epsilon}\right)^n
    \|\Phi^{(m)}\|
    \]
on $D_\epsilon^m$ for all $n\geq 1$.
\end{lem}
\begin{proof}
In the proof of this lemma, we write $q$ instead of $q_{m+1}$. By
assumption, $\Phi_{m;0}=\Phi^{(m)}$ converges normally on compacta
of $D_\epsilon^m$ if $0<\epsilon<\delta$. We now use the
recurrence relation \eqref{rec2} to obtain the desired results for
$\Phi_{m;n}$ with $n\geq 1$.

By the proof of Lemma \ref{lemprop} and since $0<q<1$, there
exists some $\epsilon>0$ (independent of $n\geq 1$) such that
$\det(1-q^nA_{m+1}^{(m)})^{-1}$ is analytic on $D_{\epsilon}^m$
for all $n\geq 1$ and such that $|\det(1-q^nA_{m+1}^{(m)})^{-1}|$
is bounded on the closure $\overline{D}_{\epsilon}^m$ of
$D_{\epsilon}^m$, with bound independent of $n\geq 1$. For such
$\epsilon$, it follows from \eqref{rec2} that $\Phi_{m;n}$
converges normally on compacta of $D_\epsilon^m$ for all $n\geq
1$. Furthermore, by {\bf (e)}, $0<q<1$, and Cramer's rule, it
implies that for $\epsilon>0$ small enough,
\[\|(1-q^nA_{m+1}^{(m)})^{-1}\|\leq C^\prime
\]
on $\overline{D}_\epsilon^{m}$ for all $n\geq 1$, with
$C^\prime>0$ also independent of $n$. By \eqref{rec2}, $0<q<1$ and
the previous lemma, we thus obtain for $\epsilon>0$ small enough,
\begin{equation}\label{eqDelta}
    \|\Phi_{m;n}\|\leq C'\sum_{l=1}^nq^{n-l}\|A_{m+1;l}^{(m)}\|\:
    \|\Phi_{m;n-l}\|\leq C \sum_{l=1}^n\left(\frac{1}{q\epsilon}\right)^l
     \|\Phi_{m;n-l}\|
\end{equation}
on $\overline{D}_\epsilon^m$ for all $n\geq 1$ with the constant
$C=C'M>0$ (independent of $n$).

Now, we have the following claim (cf. \cite{EFK} \S10.6): the
recurrence relation
\[
g_n=C\sum_{l=1}^n\left(\frac{1}{q\epsilon}\right)^lg_{n-l},\quad
(n>0)
\]
with $g_0\in\mathbb{R}$ fixed is uniquely solved by
\[
g_n=\frac{C}{C+1}\left(\frac{1+C}{q\epsilon}\right)^ng_{0}
\]
for $n\geq 1$. Being obvious for $n=1$, the claim follows using
induction for $n>1$ by
\[
\begin{split}
    g_n&=C\left(\frac{1}{q\epsilon}\right)^ng_{0}
        +C\sum_{l=1}^{n-1}\left(\frac{1}{q\epsilon}\right)^lg_{n-l}\\
       &=C\left(\frac{1}{q\epsilon}\right)^ng_{0}
        +\frac{C^2}{C+1}\sum_{l=1}^{n-1}\left(\frac{1}{q\epsilon}\right)^{l}
        \left(\frac{(1+C)}{q\epsilon}\right)^{n-l}g_{0}\\
       &=C\left(\frac{1}{q\epsilon}\right)^ng_{0}
       \left(1+C\sum_{l=0}^{n-2}(1+C)^l\right)\\
       &=C\left(\frac{1}{q\epsilon}\right)^ng_{0}
       \left(1+C\left(\frac{(1+C)^{n-1}-1}{1+C-1}\right)\right)\\
       &=\frac{C}{C+1}\left(\frac{1+C}{q\epsilon}\right)^ng_{0}.
\end{split}
\]
Combined with \eqref{eqDelta}, the lemma now follows immediately.
\end{proof}
To conclude the proof of the proposition, note that the previous
lemma shows that
\[
\Phi^{(m+1)}=\sum_{n\geq
0}\Phi_{m;n}z_{m+1}^n\in\C[[z_1,\ldots,z_{m+1}]]
\otimes\textup{End}(V)
\]
converges normally on compacta of $D_{\epsilon^\prime}^{m+1}$ if
we take $\epsilon^\prime>0$ sufficiently small. This concludes the
proof of the induction step.
\end{proof}
We interpret the $q$-dilation operators $\mathcal{T}_i$ as
automorphisms of $\mathcal{M}(\C^M)$ by
\[(\mathcal{T}_if)(z)=f(z_1,\ldots,z_{i-1},q_iz_i,z_{i+1},\ldots,z_M).
\]
\begin{thm}\label{merversion}
Suppose $A_i\in\mathcal{M}(\C^M)\otimes\textup{End}(V)$ ($1\leq
i\leq M$) satisfy the holonomy conditions
\eqref{eqGeneralHolonomic} as meromorphic $\textup{End}(V)$-valued
functions on $\C^M$. Suppose that the $A_i$ are analytic at
$0\in\C^M$ and that their power series expansions at $0\in\C^M$
satisfy the conditions {\bf (b)-(d)}.

Let $v\in V[(1^M)]$. There exists a unique
$\Phi_v\in\mathcal{M}(\C^M)\otimes V$ solving the holonomic system
\eqref{eqGeneralq-diff} of $q$-difference equations and
coinciding, in a small neighborhood of $0\in\C^M$, with the
converging $V$-valued power series solution $\Phi_v$ from
Proposition \ref{convprop}.
\end{thm}
\begin{proof}
Since the $A_i$ are assumed to be analytic at $0\in\C^M$, their
power series expansions at $0\in\C^M$ are converging normally on
compacta of some open polydisc $D_\epsilon^M$ ($\epsilon>0$).
Hence, condition {\bf (e)} is automatically satisfied.

Let $\Phi_v\in\C[[z]]\otimes V$ be the power series solution from
Proposition \ref{convprop} and let $\epsilon>0$ such that $\Phi_v$
converges normally on compacta of $D_\epsilon^M$. Let
$z^\prime\in\C^M$ and $U\subset\C^M$ some open locally compact
neighborhood of $z^\prime$. Since $0<q_i<1$ ($1\leq i\leq M$),
there exists a $\lambda\in\mathbb{Z}_{\geq 0}^M$ such that
$q^\lambda U\subset D_\epsilon^M$, where $q^\lambda
z=(q_1^{\lambda_1}z_1,\ldots,q_M^{\lambda_M}z_M)$. Define $\Phi_v$
as $V$-valued meromorphic function on $z\in U$ by
\begin{equation}\label{Phiextend}
\Phi_v(z)=A_\lambda(z)\Phi_v(q^\lambda z),
\end{equation}
where $A_\lambda\in\mathcal{M}(\C^M)\otimes\textup{End}(V)$ is
defined inductively by
\[A_{\lambda+\mu}(z)=A_\lambda(z)A_{\mu}(q^\lambda z),\qquad
\forall \lambda,\mu\in\mathbb{Z}_{\geq 0}^M,
\]
and $A_{\epsilon_i}=A_i$ ($1\leq i\leq M$), where the $\epsilon_i$
($1\leq i\leq M$) are the standard generators of the additive
monoid $\mathbb{Z}_{\geq 0}^M$. Of course, the definition of
$A_\lambda(z)$ makes sense by the holonomy conditions for the
$A_i$. Furthermore, \eqref{Phiextend} together with the holonomy
conditions for the $A_i$ show that the power series solution
$\Phi_v$ of \eqref{eqGeneralHolonomic} has a unique extension to a
meromorphic $V$-valued solution on $\C^M$ of
\eqref{eqGeneralHolonomic}.
\end{proof}


\end{document}